%% file: RWM_gap_revision_arxiv.tex
\documentclass[english,aap]{article}
\usepackage[T1]{fontenc}
\usepackage[latin9]{inputenc}
\usepackage{geometry}
\usepackage{color}
\usepackage{babel}
\usepackage{mathrsfs}
\usepackage{amsmath}
\usepackage{amsthm}
\usepackage{amssymb}
\usepackage{graphicx}
\usepackage[authoryear]{natbib}
\usepackage[unicode=true,pdfusetitle,
 bookmarks=true,bookmarksnumbered=false,bookmarksopen=false,
 breaklinks=false,pdfborder={0 0 1},backref=false,colorlinks=true]
 {hyperref}
\hypersetup{
 linkcolor=blue,citecolor=blue}

\makeatletter
\theoremstyle{plain}
\newtheorem{thm}{\protect\theoremname}
\theoremstyle{definition}
\newtheorem{defn}[thm]{\protect\definitionname}
\theoremstyle{plain}
\newtheorem{lem}[thm]{\protect\lemmaname}
\theoremstyle{plain}
\newtheorem{assumption}[thm]{\protect\assumptionname}
\theoremstyle{plain}
\newtheorem{cor}[thm]{\protect\corollaryname}
\theoremstyle{remark}
\newtheorem{rem}[thm]{\protect\remarkname}
\theoremstyle{definition}
\newtheorem{example}[thm]{\protect\examplename}
\theoremstyle{plain}
\newtheorem{prop}[thm]{\protect\propositionname}

\usepackage{tikz}
\usetikzlibrary{shapes.geometric, arrows}

\tikzstyle{block} = [rectangle, rounded corners, minimum width = 4cm, minimum height=1cm, text centered, draw = black, text width = 6cm]
\tikzstyle{arrow} = [thick, ->, >=stealth]
\tikzstyle{imparrow} = [double, double equal sign distance, -implies]

\makeatother

\providecommand{\assumptionname}{Assumption}
\providecommand{\corollaryname}{Corollary}
\providecommand{\definitionname}{Definition}
\providecommand{\examplename}{Example}
\providecommand{\lemmaname}{Lemma}
\providecommand{\propositionname}{Proposition}
\providecommand{\remarkname}{Remark}
\providecommand{\theoremname}{Theorem}

\begin{document}
\title{Explicit convergence bounds for Metropolis Markov chains: isoperimetry,
spectral gaps and profiles}
\author{Christophe Andrieu$^{*}$, Anthony Lee$^{*}$, Sam Power$^{*}$, Andi
Q. Wang$^{\dagger}$\\
School of Mathematics, University of Bristol$^{*}$\\
Department of Statistics, University of Warwick$^{\dagger}$}
\maketitle
\begin{abstract}
We derive the first explicit bounds for the spectral gap of a random
walk Metropolis algorithm on $\mathbb{R}^{d}$ for any value of the
proposal variance, which when scaled appropriately recovers the correct
$d^{-1}$ dependence on dimension for suitably regular invariant distributions.
We also obtain explicit bounds on the ${\rm L}^{2}$-mixing time for
a broad class of models. In obtaining these results, we refine the
use of isoperimetric profile inequalities to obtain conductance profile
bounds, which also enable the derivation of explicit bounds in a much
broader class of models. We also obtain similar results for the preconditioned
Crank--Nicolson Markov chain, obtaining dimension-independent bounds
under suitable assumptions.
\end{abstract}
\global\long\def\dif{\mathrm{d}}%
\global\long\def\Var{\mathrm{Var}}%
\global\long\def\R{\mathbb{R}}%
\global\long\def\X{\mathcal{X}}%
\global\long\def\calE{\mathcal{E}}%
\global\long\def\E{\mathsf{E}}%
\global\long\def\Ebb{\mathbb{E}}%

\global\long\def\ELL{\mathrm{L}^{2}}%
\global\long\def\osc{\mathrm{osc}}%
\global\long\def\Id{\mathrm{Id}}%
\global\long\def\essup{\mathrm{ess\,sup}}%
 
\global\long\def\shortminus{\mathrm{\shortminussymb}}%
\global\long\def\muess{\mathrm{ess_{\mu}}}%

\global\long\def\GapR{\mathrm{Gap_{{\rm R}}}}%


\section{Introduction}

\subsection{Results for Metropolis Markov chains}

Let $\pi$ be a probability distribution on $\left(\mathsf{E},\mathcal{\mathscr{E}}\right)$,
where $\E=\R^{d}$ and $\mathscr{E}$ denotes its Borel $\sigma$-algebra,
and suppose we seek to approximately sample from $\pi$. Markov chain
Monte Carlo (MCMC) algorithms address this problem by simulating an
ergodic, time-homogeneous Markov chain $\left(X_{n}\right)_{n\in\mathbb{N}}$
with invariant distribution $\pi$. One of the most simple and yet
enduringly popular MCMC algorithms is the \textit{Metropolis algorithm}
of \citet{Metropolis1953}. Assuming that $\pi$ has density $\varpi={\rm d}\pi/{\rm d}\nu$
with respect to some $\sigma$-finite measure $\nu$, and $Q$ is
a $\nu$-reversible Markov kernel, a \textit{Metropolis Markov kernel}
may be written 
\begin{equation}
P\left(x,A\right)=\int_{A}Q\left(x,{\rm d}y\right)\alpha\left(x,y\right)+{\bf 1}_{A}\left(x\right)\bar{\alpha}\left(x\right),\qquad x\in\mathsf{E},A\in\mathscr{E}.\label{eq:metropolis-kernel}
\end{equation}
where for $x,y\in\E$,
\begin{equation}
\alpha\left(x,y\right):=\min\left\{ 1,\frac{\varpi\left(y\right)}{\varpi\left(x\right)}\right\} ,\quad\bar{\alpha}\left(x\right):=1-\alpha\left(x\right),\quad\alpha\left(x\right):=\int_{\mathsf{E}}Q\left(x,{\rm d}y\right)\alpha\left(x,y\right).\label{eq:alpha_defs}
\end{equation}

In many applications, $Q\left(x,\cdot\right)$ is a multivariate normal
distribution with mean $x$ and covariance matrix $\sigma^{2}\cdot I_{d}$,
where $I_{d}$ is the identity matrix, in which case $\nu$ is the
Lebesgue measure and $P$ is the \textit{Random-Walk Metropolis} (RWM)
Markov kernel. Despite its simplicity, the RWM algorithm is known
to perform very well for certain classes of target distributions,
and furthermore to be a \textit{robust} algorithm \citep[see, e.g.,][]{Roberts1997,christensen2005scaling,chenJMLR:v21:19-441,Livingstone2022}.
In this paper, quantitative analysis of the $\ELL$-mixing time and
spectral gap of the RWM Markov chain is the primary application, with
a particular emphasis on the dependence of these quantities on dimension.
This analysis relies on a more general theory applicable beyond the
specific scenarios considered here; see Section~\ref{subsec:Roadmap}.

For a given target distribution $\pi$, after fixing the coordinate
system, the only tuning parameter of the RWM kernel is the proposal
variance $\sigma^{2}$. It is well-known that if $\sigma^{2}$ is
too large, then the acceptance function $\alpha$ will deteriorate,
and the Markov chain will tend to get ``stuck'' for long periods.
On the other hand, if $\sigma^{2}$ is too small, then the Markov
chain will tend to make very small steps. Both of these regimes correspond
intuitively to slow convergence of the Markov chain. In the celebrated
optimal scaling paper of \citet{Roberts1997}, it was shown, for a
fairly restrictive class of target distributions, that the proposal
variance $\sigma^{2}$ of RWM on $\R^{d}$ should scale like $d^{-1}$
to obtain a stable acceptance ratio in the high-dimensional limit,
and that the complexity of sampling depends linearly on dimension,
via a particular but indicative weak convergence result to a Langevin
diffusion. In this paper, we study the high-dimensional properties
of the RWM algorithm from a different angle: we seek to explicitly
bound the spectral gap of the RWM kernel in arbitrary dimension $d$
and for \textit{any} value of $\sigma^{2}$. For appropriately regular
distributions, we find that scaling $\sigma^{2}$ as $d^{-1}$ does
indeed imply a spectral gap that is precisely of order $d^{-1}$,
and that this choice of polynomial scaling is optimal. The following
is a combination of Corollary~\ref{cor:rwm-conductance-lower-Lm}
and Theorem~\ref{thm:upper-bound-gap}: 
\begin{thm}
\label{thm:intro-gap}Let $\pi$ have density $\pi\left(x\right)\propto\exp\left(-U\left(x\right)\right)$
with respect to Lebesgue measure on $\R^{d}$, where the potential
$U$ is $L$-smooth, $m$-strongly convex and twice continuously differentiable.
If $P$ is the $\pi$-reversible RWM kernel with $\mathcal{N}\left(0,\sigma^{2}\cdot I_{d}\right)$
proposal increments, then the spectral gap $\gamma_{P}$ of $P$ satisfies
\begin{equation}
C\cdot L\cdot d\cdot\sigma^{2}\cdot\exp\left(-2\cdot L\cdot d\cdot\sigma^{2}\right)\cdot\frac{m}{L}\cdot\frac{1}{d}\leqslant\gamma_{P}\leqslant\min\left\{ \frac{1}{2}\cdot L\cdot\sigma^{2},\left(1+m\cdot\sigma^{2}\right)^{-d/2}\right\} ,\label{eq:gap-lower-upper}
\end{equation}
where $C=1.972\times10^{-4}$.
\end{thm}

Twice continuous differentiability of $U$ is only used to obtain
the upper bound. For some intuition, densities with $m$-strongly
convex and $L$-smooth potentials $U$ can be sandwiched between $\mathcal{N}\left(x_{*},L^{-1}\cdot I_{d}\right)$
and $\mathcal{N}\left(x_{*},m^{-1}\cdot I_{d}\right)$, up to constant
factors, where $x_{*}$ is the maximizer of the density of $\pi$;
see Lemma~\ref{lem:gauss-bd}.

Both the lower and upper bounds in~(\ref{eq:gap-lower-upper}) demonstrate
that taking $\sigma^{2}$ too small or too large causes $\gamma_{P}$
to decrease. The lower bound in (\ref{eq:gap-lower-upper}) is maximized
by taking $\sigma^{2}=1/\left(2\cdot L\cdot d\right)$, while the
rate at which the upper bound decreases with $d$ is also minimized,
among polynomial scalings, by scaling $\sigma^{2}$ with $d^{-1}$.
Taking $\sigma=\varsigma\cdot L^{-1/2}\cdot d^{-1/2}$ for any constant
$\varsigma>0$, we obtain 
\begin{equation}
C\cdot\varsigma^{2}\cdot\exp\left(-2\cdot\varsigma^{2}\right)\cdot\frac{m}{L}\cdot\dfrac{1}{d}\leqslant\gamma_{P}\leqslant\frac{\varsigma^{2}}{2}\cdot\dfrac{1}{d},\label{eq:gap-upper-lower-varsigma}
\end{equation}
so the $\mathcal{O}\left(d^{-1}\right)$ dimension dependence is tight.
The lower bound is maximized by taking $\varsigma^{2}=\frac{1}{2}$,
although it is unlikely that this is optimal in practice due to the
results of \citet{Roberts1997}. Similarly, it seems likely that the
optimal value of $C$ is possibly a few orders of magnitude larger. 

We also study the $\ELL$-convergence complexity of the RWM Markov
chain, noting that convergence can initially be \textit{faster} than
that indicated by the spectral gap alone and this turns out to be
crucial to establish our dimension dependence results for $m$-strongly
convex and $L$-smooth potentials. Under the same conditions as Theorem~\ref{thm:intro-gap}
and taking $\sigma=\varsigma\cdot L^{-1/2}\cdot d^{-1/2}$ as above,
we obtain that for at least two types of feasible initial distribution
$\mu$ (see Theorem~\ref{thm:RWM_mixing} and Remarks~\ref{rem:normal-L-initialization-RWM}
and~\ref{rem:bc-acc-warm}) one may take
\[
n\in\mathcal{O}\left(\exp\left(2\cdot\varsigma^{2}\right)\cdot\varsigma^{-2}\cdot\kappa\cdot d\cdot\left\{ \log d+\log\kappa+\log\left(\varepsilon_{{\rm Mix}}^{-1}\right)\right\} \right),
\]
and obtain $\chi^{2}\left(\mu P^{n},\pi\right)\leqslant\varepsilon_{{\rm Mix}}$,
where $\chi^{2}\left(\mu,\nu\right)$ denotes the $\chi^{2}$ divergence
between $\mu$ and $\nu$ and $\kappa:=L/m$ is the condition number.
In contrast, an analysis based only on the spectral gap bound $\gamma_{P}\in\Omega\left(1/\left(\kappa\cdot d\right)\right)$
would suggest a mixing time in $\mathcal{O}\left(d^{2}\kappa\log\kappa\right)$.

In practice, fluctuations of ergodic averages of $f\in\ELL\left(\pi\right)$
are also of interest, and one may consider the \textit{asymptotic
variance}, given by
\[
{\rm var}\left(P,f\right):=\lim_{n\to\infty}n\cdot{\rm var}\left(\frac{1}{n}\sum_{i=1}^{n}f\left(X_{i}\right)\right),
\]
where $X_{0}\sim\pi$. We show in Proposition~\ref{prop:result-avar-RWM}
that with $\sigma=\varsigma\cdot L^{-1/2}\cdot d^{-1/2}$, 
\begin{align*}
{\rm var}\left(P,f\right) & \leqslant10141\cdot\varsigma^{-2}\cdot\exp\left(2\cdot\varsigma^{2}\right)\cdot\kappa\cdot d\cdot\left\Vert f\right\Vert _{2}^{2},\qquad\varsigma>0.
\end{align*}
We also show that linear functions satisfy ${\rm var}\left(P,f\right)\geqslant2\cdot\varsigma^{-2}\cdot d\cdot\left\Vert f\right\Vert _{2}^{2}.$

We also analyze the \textit{preconditioned Crank--Nicolson} (pCN)
Markov chain via essentially analogous theory to the RWM chain, since
it is also a Metropolis Markov chain. For example, we show in Theorem~\ref{thm:pcn-gap}
if $\pi({\rm d}x)\propto\mathcal{N}({\rm d}x;0,\mathsf{C})\exp\left(-\Psi\left(x\right)\right)$
with $\Psi$ convex, $L$-smooth and minimized at $x=0$ then an appropriately
tuned pCN Markov chain's spectral gap satisfies
\[
\gamma_{P}\geqslant3.62784\times10^{-5}\cdot(L\cdot{\rm Tr}(\mathsf{C}))^{-1},
\]
giving dimension-independent bounds when $L\cdot{\rm Tr}(\mathsf{C})$
is bounded independent of dimension.

To prove these results we apply a general result, Theorem~\ref{thm:IP-to-mix},
which requires quantitative lower bounds on the \textit{isoperimetric
profile} of $\pi$ for some metric $\mathsf{d}$, complemented with
a quantitative \textit{close coupling} condition for $P$:
\begin{defn}[Close coupling]
For a metric $\mathsf{d}$ on $\E$ and $\epsilon,\delta>0$, a Markov
kernel $P$ evolving on $\mathsf{E}$ is \textit{$\left(\mathsf{d},\delta,\varepsilon\right)$-close
coupling} if
\[
\mathsf{d}\left(x,y\right)\;\leqslant\delta\Rightarrow\left\Vert P\left(x,\cdot\right)-P\left(y,\cdot\right)\right\Vert _{{\rm TV}}\leqslant1-\varepsilon,\qquad x,y\in\mathsf{E}.
\]
\end{defn}

This is to be contrasted with what is known about the overdamped Langevin
diffusion, which solves the stochastic differential equation
\[
\dif X_{t}=\nabla\log\pi\left(X_{t}\right)\,\dif t+\sqrt{2}\cdot\dif W_{t},
\]
for which knowledge of the isoperimetric profile \textit{alone} can
provide information on its convergence. For example, the overdamped
Langevin diffusion is associated with the classical Dirichlet form
$f\mapsto\pi\left(\left|\nabla f\right|^{2}\right)$; \citep[see, e.g.][Section~4.5]{pavliotis2014stochastic},
and this allows one to deduce Poincaré and log-Sobolev inequalities
in the presence of appropriate isoperimetric inequalities \citep[see, e.g.,][Section~2.2]{milman2012properties}.
The RWM chain may indeed be viewed as a discretization of this diffusion,
but our results do not explicitly compare the diffusion with the Markov
chain; indeed our quantitative bounds are valid in \textit{any} dimension
and for \textit{any} value of $\sigma^{2}$. The additional close
coupling condition required for RWM in fact introduces a penalty in
the convergence bounds, by which convergence degrades as the product
$\delta\cdot\varepsilon$ decreases. To demonstrate close coupling
for Metropolis chains, we show that for $\alpha_{0}:=\inf_{z\in\mathsf{E}}\alpha\left(z\right)$,
with $\alpha$ as in (\ref{eq:alpha_defs}),
\[
\left\Vert P\left(x,\cdot\right)-P\left(y,\cdot\right)\right\Vert _{{\rm TV}}\leqslant\left\Vert Q\left(x,\cdot\right)-Q\left(y,\cdot\right)\right\Vert _{{\rm TV}}+1-\alpha_{0},
\]
and we show that $\alpha_{0}$ can be lower bounded for any $\sigma^{2}$
under the assumption of $L$-smoothness. One may then take $\delta$
such that $\left|x-y\right|\leqslant\delta\Rightarrow\left\Vert Q\left(x,\cdot\right)-Q\left(y,\cdot\right)\right\Vert _{{\rm TV}}\leqslant\frac{1}{2}\cdot\alpha_{0}$
to obtain that $P$ is close coupling with $\varepsilon\geqslant\frac{1}{2}\cdot\alpha_{0}$.
In our analysis, we find that to maximize the spectral gap of $P$
as a function of dimension, it is sufficient to scale $\sigma^{2}$
as $d^{-1}$. Ultimately, one may view the penalty for running an
appropriately tuned RWM instead of Langevin as being of order $d^{-1}$
in terms of the spectral gap.

\subsection{Roadmap \label{subsec:Roadmap}}

In Section~\ref{sec:Conductance-Profile,-Spectral}, we review the
notions of \textit{conductance profile} and \textit{spectral profile}
for Markov chains, and show how these can be used to establish bounds
on the spectral gap and mixing time of the chain.

In Section~\ref{sec:From-Isoperimetric-Profiles}, we introduce notions
of \textit{isoperimetric profiles} of probability measures with respect
to a given metric. We show that when combined with the \textit{close
coupling} condition for an invariant Markov kernel, one can deduce
bounds on the conductance profile of the chain, and hence on the spectral
gap and mixing time. We then give a number of concrete examples in
which the isoperimetric profile can be well-controlled, and discuss
the implications on convergence. See Figure~\ref{fig:secs2-3-map}
for a diagram of these results. We also prove a general close-coupling
result for Metropolis algorithms. 

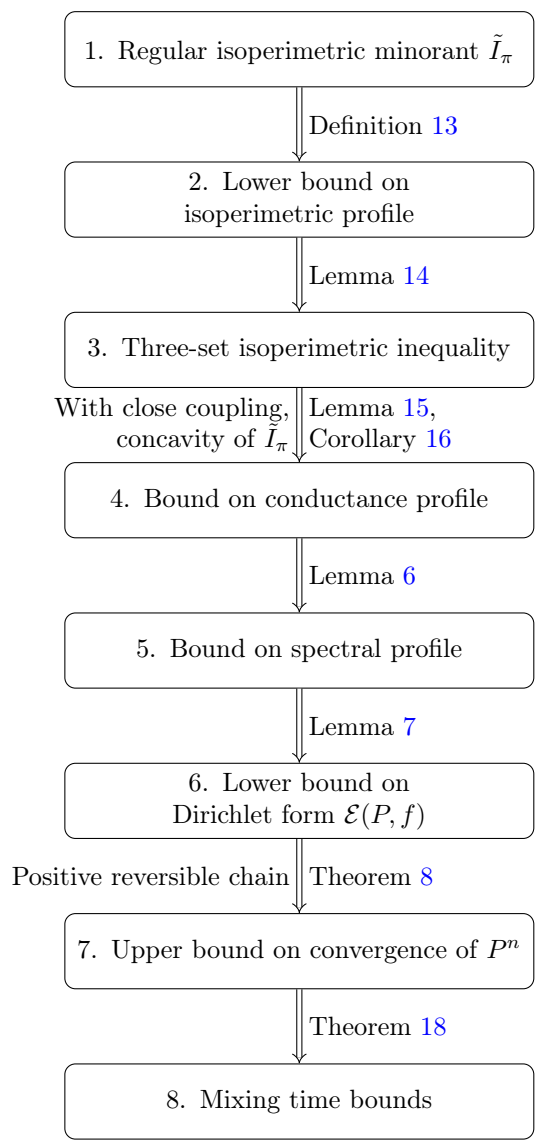
\begin{figure}
\begin{centering}
\input{flowchart.tex}
\par\end{centering}
\caption{Outline of the results of Sections~\ref{sec:Conductance-Profile,-Spectral},~\ref{sec:From-Isoperimetric-Profiles}}
\label{fig:secs2-3-map}
\end{figure}

In Sections~\ref{sec:Spectral-gap-of-RWM}--\ref{sec:Convergence-and-mixing},
we apply these tools to study the RWM algorithm. Although one can
obtain non-asymptotic bounds for the spectral gap and mixing time
under various isoperimetric and smoothness assumptions using our techniques,
we focus on obtaining concrete results when $U$ is \textit{$m$-strongly
convex} and \textit{$L$-smooth}. In this case, we obtain non-asymptotic
estimates on both the spectral gap and the mixing time. Furthermore,
we show that for this class of densities the dependence of the spectral
gap on dimension is sharp, up to sub-polynomial factors. We also demonstrate
that, when appropriately tuned, the asymptotic variance of a RWM chain
is upper bounded by a linear factor $d$ times the ideal variance,
and that there is a matching linear lower bound for linear functions.

In Section~\ref{sec:Cvg-mix-pcn}, we apply the same tools to the
study of the pCN algorithm for sampling perturbations of Gaussian
measures. We again obtain non-asymptotic estimates on both the spectral
gap and mixing time of the chain.

\subsection{Related work \label{subsec:Related-work}}

This paper develops in a systematic manner a comparatively crude analysis
in a technical report by the authors \citep[Sections~5.3--5.4]{Andrieu2022}.

One of the first attempts to establish directly the existence of a
(right) $\ELL$-spectral gap for RWM on $\mathbb{R}^{d}$ is \citet{miclo2000trous},
where a quantization approach is used to approximate the initial problem,
and results concerning Markov chains on graphs are leveraged. The
authors focus on the scenario $d=1$, although it is argued that the
results could be generalized to multiple dimensions. The assumptions
made on the negative log-density $U$ are less stringent than those
considered here, however no quantitative bounds were provided. Another
contribution in this direction is \citet{jarner-yuen} where a lower
bound on the conductance was obtained for monotone log-concave distributions
on $\mathbb{R}$. The existence of an $\ELL$ spectral gap can also
be established using drift and minorization techniques: for example,
\citet{jarner2000geometric} demonstrate that the RWM chain is ${\rm L}^{1}$-
geometrically ergodic under fairly mild conditions on $\pi$, and
since RWM chains are reversible those conditions also imply $\ELL$-geometric
ergodicity and the existence of a spectral gap \citep{roberts1997geometric,roberts2001geometric}.
However, these techniques typically do not provide accurate quantitative
bounds on the size of the gap; see for instance, \citet{qin2021limitations}. 

The results and the approach we take here are inspired by several
papers. The recent papers \citet{belloni2009computational}, \citet{Dwivedi-Chen-Wainwright-JMLR:v20:19-306}
and \citet{chenJMLR:v21:19-441} are most closely related to our approach.
All three of these papers also consider bounds on the conductance
or conductance profile of the RWM kernel $P$, but \textit{restricted}
to some compact subset $K$ of $\mathbb{R}^{d}$. As a consequence,
they do not provide a positive lower bound on the spectral gap of
the unrestricted $P$. More precisely, \citet{belloni2009computational}
and \citet{Dwivedi-Chen-Wainwright-JMLR:v20:19-306} prove a restricted
variant of Corollary~\ref{cor:iso-couple-conductance} for the conductance,
corresponding to connection 3--4 in Figure~\ref{fig:secs2-3-map}.
\citet{chenJMLR:v21:19-441} prove restricted variants of connections
3--8. \citet{Dwivedi-Chen-Wainwright-JMLR:v20:19-306} and \citet{chenJMLR:v21:19-441}
obtain complexity bounds for convergence of $\mu P^{n}$ to $\pi$
on $\mathbb{R}^{d}$ using the notion of $s$-conductance, which entails
delicate balancing of the desired final error, the size of $K$, properties
of $\mu$ and even $\sigma^{2}$. In particular, we emphasize that
all three prior complexity analyses involve using specific, theoretically-motivated
and typically unknown values of $\sigma^{2}$, so the results do not
cover the arbitrary values of $\sigma^{2}$ used in practice. The
restriction to $K$ in these papers is necessary since the authors
only verify the close coupling condition for $P$ on $K$. In contrast,
we are able to verify this condition globally, and hence there is
no need to consider restrictions. As a result, we can obtain a positive
lower bound on the spectral gap, and the convergence analysis does
not require the same type of fine balancing. In particular, we also
find an improved dependence of the mixing time on the condition number
$\kappa$, in comparison to the dependence in \citet{chenJMLR:v21:19-441}.
We also mention \citet{MATHE2007673}, who proved that the Metropolis
chain with a ball-walk proposal for $\pi$ log-concave with Lipschitz
potential and restricted to a ball has a spectral gap in $\Omega(d^{-2})$;
see also \citet{RUDOLF200911}.

\citet{belloni2009computational} and \citet{Dwivedi-Chen-Wainwright-JMLR:v20:19-306}
use a type of 3-set exponential isoperimetric profile inequality to
infer a bound on the conductance of the chain restricted to $K$,
in the presence of the close coupling condition. The isoperimetric
inequality is verified for (perturbations of) $m$-strongly convex
potentials. In \citet{chenJMLR:v21:19-441}, a Gaussian 3-set isoperimetric
profile inequality is used to infer a bound on the conductance profile
of the restricted chain in the presence of the same coupling condition,
and an isoperimetric profile inequality is verified for strongly convex
potentials. Our main contribution in relation to this part of the
theory is to show that any sufficiently regular isoperimetric profile
implies a corresponding 3-set isoperimetric inequality. In fact, \citet{chenJMLR:v21:19-441}'s
consideration of the Gaussian 3-set isoperimetric inequality and its
implication for rapid convergence far from equilibrium was the main
inspiration for our results relating classical isoperimetric profiles
and conductance profiles more generally. Our subsequent mixing time
results are mostly direct consequences of the relationships between
the conductance profile, spectral profile and $\ELL$-convergence,
as developed by \citet{goel2006mixing}. 

\citet{hairer2014spectral} show the existence and stability of the
spectral gap of pCN as $d\to\infty$ under quite general conditions,
but the bounds so obtained are understood to be somewhat loose numerically,
and their dependence on the various parameters of the target measure
is implicit. Here, we make more restrictive assumptions on the target
measure, which allows us to obtain bounds which are more interpretable
and perhaps sharper.

\subsection{Notation}

Notation is collected for convenience in Appendix~\ref{sec:Notation}.

\section{Conductance profile, spectral profile, and mixing time bounds\label{sec:Conductance-Profile,-Spectral}}

The spectral gap $\gamma_{P}$ of a $\pi$-reversible Markov kernel
$P$ provides important information on the convergence of the chain.
Indeed, for any $n\in\mathbb{N}_{0}$,
\begin{equation}
\left\Vert P^{n}f\right\Vert _{2}\leqslant\left\Vert f\right\Vert _{2}\cdot\left(1-\gamma_{P}\right)^{n},\qquad f\in\ELL_{0}\left(\pi\right),\label{eq:gap-convergence}
\end{equation}
and the factor $\left(1-\gamma_{P}\right)$ cannot be reduced in general,
motivating quantitative lower bounds on $\gamma_{P}$. By taking $f={\rm d}\mu/{\rm d}\pi-1$
in (\ref{eq:gap-convergence}) we may deduce bounds on $\chi^{2}\left(\mu P^{n},\pi\right)$,
the chi-squared divergence between $\mu P^{n}$ and $\pi$, and thereby
upper bound mixing times. However, using only this bound can give
very conservative bounds when $\chi^{2}\left(\mu,\pi\right)$ is large,
and so we will use more refined techniques to control the convergence
behaviour of the chain when it is far from equilibrium. In particular,
we make use of the \textit{spectral profile} \citep{goel2006mixing}
and \textit{conductance profile} of the Markov chain \citep{lovasz1999faster,morris2005evolving}.
These techniques are able to capture the following phenomenon: many
Markov chains, when far from equilibrium, are able to mix at \textit{faster}
than exponential rates, or equally, that sets of small measure in
the state space are comparatively easier to escape from. Moreover,
these techniques are capable of providing greatly-improved bounds
on mixing times, and in some cases, nearly-optimal bounds \citep[see, e.g.,][]{kozma2007precision}.
\begin{defn}[Conductance and conductance profile]
\label{def:conductance-profile}The \textit{conductance profile}
of a $\pi$-invariant Markov kernel $P$ is
\[
\Phi_{P}\left(v\right):=\inf\left\{ \frac{\left(\pi\otimes P\right)\left(A\times A^{\complement}\right)}{\pi\left(A\right)}:A\in\mathscr{E},0<\pi\left(A\right)\leqslant v\right\} ,\qquad v\in\left(0,\frac{1}{2}\right].
\]
The \textit{conductance} of $P$ is $\Phi_{P}^{*}:=\Phi_{P}\left(\frac{1}{2}\right)$.
\end{defn}

\begin{defn}[Spectral profile]
Let $P$ be a $\pi$-invariant Markov kernel, then we define
\[
C_{0}^{+}\left(A\right):=\left\{ g:\mathsf{E}\to\mathbb{R}\;\mid\;{\rm supp}\,g\subseteq A,g\geqslant0,g\neq\text{const. \ensuremath{\pi}-a.s.}\right\} ,\qquad A\in\mathscr{E},
\]
where $\mathrm{supp}\,g$ is the closure of $\left\{ x\in\E:\left|g\left(x\right)\right|>0\right\} $,
and
\[
\lambda_{P}\left(A\right):=\inf_{g\in C_{0}^{+}\left(A\right)}\frac{\mathcal{E}\left(P,g\right)}{{\rm Var}_{\pi}\left(g\right)},\qquad A\in\mathscr{E},\pi(A)>0.
\]
The \textit{spectral profile} of $P$ is
\[
\Lambda_{P}\left(v\right):=\inf\left\{ \lambda_{P}\left(A\right):A\in\mathscr{E},0<\pi\left(A\right)\leqslant v\right\} ,\qquad v>0.
\]
\end{defn}

We note that for all $v>0$ and $\pi$-reversible $P$, we have that
$\Lambda_{P}\left(v\right)\geqslant{\rm Gap}_{{\rm R}}\left(P\right)\geqslant\gamma_{P}$.
To proceed from here, we first use a Cheeger-type argument to bound
the spectral profile using the conductance profile. The statement
and proof of Lemma~\ref{lem:cond-spec-profile-gap} are very similar
to \citet[Lemma~12]{chenJMLR:v21:19-441}, with one difference being
that we do not restrict the state space. The proof can be found in
Appendix~\ref{sec:Proofs-of-auxiliary}. We also recall Cheeger's
inequalities.
\begin{lem}[{\citealt[Theorem~3.5; Cheeger's inequalities]{lawler1988bounds}}]
\label{lem:lawlersokal}If $P$ is a $\pi$-reversible Markov kernel,
then
\[
\frac{1}{2}\cdot\left[\Phi_{P}^{*}\right]^{2}\leqslant\GapR\left(P\right)\leqslant2\Phi_{P}^{*}.
\]
\end{lem}

\begin{lem}
\label{lem:cond-spec-profile-gap}If $P$ is a $\pi$-reversible Markov
kernel, then
\[
\Lambda_{P}\left(v\right)\geqslant\begin{cases}
\frac{1}{2}\cdot\Phi_{P}\left(v\right)^{2} & 0<v\leqslant\frac{1}{2},\\
\frac{1}{2}\cdot\left[\Phi_{P}^{*}\right]^{2} & v>\frac{1}{2}.
\end{cases}
\]
\end{lem}

We will make use of the following lower bound on the Dirichlet form
in terms of the spectral profile.
\begin{lem}[{\citealt[Lemma~2.1]{goel2006mixing}}]
\label{lem:dirichlet-variance-spectral-profile}For $g\in\ELL\left(\pi\right)$
non-negative and not constant $\pi$-a.s.,
\[
\mathcal{E}\left(P,g\right)\geqslant{\rm Var}_{\pi}\left(g\right)\cdot\frac{1}{2}\cdot\Lambda_{P}\left(4\cdot\frac{\left[\pi\left(g\right)\right]^{2}}{{\rm Var}_{\pi}\left(g\right)}\right).
\]
\end{lem}

Our final result in this section shows how the conductance profile
can be used to deduce bounds on convergence of $P$. We build on previous
work, particularly \citet{goel2006mixing} in the discrete setting
and \citet{chenJMLR:v21:19-441} on general state spaces with `restricted'
conductance profiles.
\begin{thm}
\label{thm:CP-to-mix}Let $P$ be a positive, $\pi$-reversible Markov
kernel with $\Phi_{P}^{*}>0$, $\mu\ll\pi$ a probability measure,
and $\varepsilon_{{\rm Mix}}\in\left(0,8\right)$. To ensure $\chi^{2}(\mu P^{n},\pi)\leqslant\varepsilon_{{\rm Mix}}$,
it suffices to take 
\[
n\geqslant2+4\cdot\int_{\min\left\{ 4\cdot u_{0}^{-1},1/2\right\} }^{1/2}\frac{1}{v\cdot\Phi_{P}\left(v\right)^{2}}\,{\rm d}v+\left[\Phi_{P}^{*}\right]^{-2}\cdot\log\left(\max\left\{ \frac{\min\left\{ u_{0},8\right\} }{\varepsilon_{{\rm Mix}}},1\right\} \right),
\]
where $u_{0}=\chi^{2}(\mu,\pi)$.
\end{thm}

\begin{proof}
Writing $h=\frac{{\rm d}\mu}{{\rm d}\pi}$ and $u_{n}:={\rm Var}_{\pi}\left(P^{n}h\right)=\chi^{2}\left(\mu P^{n},\pi\right)$,
compute that
\[
u_{n}-u_{n+1}=\mathcal{E}\left(P^{2},P^{n}h\right)\geqslant\mathcal{E}\left(P,P^{n}h\right)\geqslant u_{n}\cdot\frac{1}{2}\cdot\Lambda_{P}\left(4\cdot u_{n}^{-1}\right),
\]
where we have used the positivity of $P$ to bound $\mathcal{E}\left(P^{2},f\right)\geqslant\mathcal{E}\left(P,f\right)$,
Lemma~\ref{lem:dirichlet-variance-spectral-profile}, and recalled
that $\pi\left(P^{n}h\right)=1$ for all $n$. Defining $L_{P}\left(\eta\right):=\frac{1}{2}\cdot\eta\cdot\Lambda_{P}\left(4\cdot\eta^{-1}\right)$
for $\eta>0$, it thus holds that $u_{n}-u_{n+1}\geqslant L_{P}\left(u_{n}\right)$.

We now distinguish between whether $u_{0}$ is greater or smaller
than $8$, noting that in the latter case, using the spectral gap
directly allows for tighter control of the increment $u_{n}-u_{n+1}$.
Supposing that $u_{0}\geqslant8$, we will first estimate how long
it takes for $u_{n}$ to drop below $8$. Recalling that $\Lambda_{P}$
is a decreasing function, it is straightforward to see that $L_{P}$
is an increasing function, and hence measurable. Additionally, since
$\Lambda_{P}$ is bounded below by $\gamma_{P}>0$, it follows that
$L_{P}$ is bounded away from $0$ on intervals not containing $0$.
Assuming that both $u_{n}$, $u_{n+1}$ are at least $8$, we can
then write
\[
\int_{u_{n+1}}^{u_{n}}\frac{{\rm d}\eta}{L_{P}\left(\eta\right)}\geqslant\frac{u_{n}-u_{n+1}}{L_{P}\left(u_{n}\right)}\geqslant1.
\]
Moreover, if $u_{0}\geqslant u_{1}\geqslant\cdots\geqslant u_{n}\geqslant8$,
then we may sum up these inequalities to see that
\[
\int_{8}^{u_{0}}\frac{{\rm d}\eta}{L_{P}\left(\eta\right)}\geqslant\int_{u_{n}}^{u_{0}}\frac{{\rm d}\eta}{L_{P}\left(\eta\right)}\geqslant n.
\]
In particular, for $n\geqslant1+\int_{8}^{u_{0}}\frac{{\rm d}\eta}{L_{P}\left(\eta\right)}$,
it must hold that $u_{n}<8$. Now, recall that for $\eta\geqslant8$,
we can bound
\[
L_{P}\left(\eta\right)=\frac{1}{2}\cdot\eta\cdot\Lambda_{P}\left(4\cdot\eta^{-1}\right)\geqslant\frac{1}{2}\cdot\eta\cdot\left\{ \frac{1}{2}\cdot\Phi_{P}\left(4\cdot\eta^{-1}\right)^{2}\right\} =\frac{\Phi_{P}\left(4\cdot\eta^{-1}\right)^{2}}{4\cdot\eta^{-1}}.
\]
We then compute that
\[
\int_{8}^{u_{0}}\frac{{\rm d}\eta}{L_{P}\left(\eta\right)}\leqslant\int_{8}^{u_{0}}\frac{4\cdot\eta^{-1}}{\Phi_{P}\left(4\cdot\eta^{-1}\right)^{2}}\,{\rm d}\eta=4\cdot\int_{4\cdot u_{0}^{-1}}^{1/2}\frac{1}{v\cdot\Phi_{P}\left(v\right)^{2}}\,{\rm d}v,
\]
noting that $\Phi_{P}$ is monotone, hence measurable, and bounded
below by $\Phi_{P}\left(\frac{1}{2}\right)>0$, hence the integral
exists.

For $u_{0}<8$, we control the decay of $u_{n}$ more tightly by using
the spectral gap of $P$, $\gamma_{P}$. We obtain
\begin{align*}
u_{n}-u_{n+1} & =\mathcal{E}\left(P^{2},P^{n}h\right)\geqslant\left(1-\left(1-\gamma_{P}\right)^{2}\right)\cdot{\rm Var}_{\pi}\left(P^{n}h\right)\\
\implies\quad u_{n+1} & \leqslant\left(1-\gamma_{P}\right)^{2}\cdot u_{n}
\end{align*}
and thus that 
\[
u_{n}\leqslant\left(1-\gamma_{P}\right)^{2\cdot n}\cdot u_{0}\leqslant\exp\left(-2\cdot\gamma_{P}\cdot n\right)\cdot u_{0}.
\]
One can then deduce that for $n\geqslant1+\frac{1}{2}\cdot\gamma_{P}^{-1}\cdot\log\left(\frac{u_{0}}{\varepsilon_{{\rm Mix}}}\right)$,
$u_{n}\leqslant\varepsilon_{{\rm Mix}}$. By Lemma~\ref{lem:lawlersokal},
we recall that $\gamma_{P}\geqslant\frac{1}{2}\cdot\left[\Phi_{P}^{*}\right]^{2}$.
The result follows by assembling the various cases.
\end{proof}
We note a similarity between the consequences of the spectral profile
and the so-called `super-Poincaré' inequalities of \citet{wang2000functional}.
See Proposition~\ref{prop:spec-to-super} in Appendix~\ref{sec:Proofs-of-auxiliary}
for some details on this connection, which may be known among experts
but does not appear to have been explicitly documented. It is well
known that stronger functional inequalities than the Poincaré inequality
allow improved dependence on dimension/initialization in Markov chain
mixing time results \citep[see, e.g.,][]{DIACONIS199820} and one
perspective on what we pursue in the sequel is that one may combine
bounds on the isoperimetric profile with the close coupling condition
to deduce bounds on the conductance profile and hence spectral profile,
which contains comparable information to functional inequalities like
the super-Poincaré inequality. Theorem~\ref{thm:CP-to-mix} demonstrates
how this functional inequality can provide sharper mixing time bounds
than those based on the spectral gap alone.

\section{From isoperimetric profiles and close coupling to mixing time bounds\label{sec:From-Isoperimetric-Profiles}}

\subsection{General results}

From this point onwards, the following assumption is in force. All
statements are made with respect to a given metric $\mathsf{d}$ on
$\E$, the dependence on which may be suppressed when no ambiguity
can result.
\begin{assumption}
\label{assu:pi-density-lebesgue}The probability distribution $\pi$
on $\E=\mathbb{R}^{d}$ has a positive density w.r.t. Lebesgue, given
by $\pi\propto\exp\left(-U\right)$, for some potential $U:\mathbb{R}^{d}\to\R$. 
\end{assumption}

\begin{defn}[Three-set isoperimetric inequality]
\label{def:three-set-iso-ineq}A probability measure $\pi$ satisfies
a \emph{three-set isoperimetric inequality} with metric $\mathsf{d}$
and function $F:\left(0,\frac{1}{2}\right]\to\left[0,\infty\right)$
if for all measurable partitions of the state space $\mathsf{E}=S_{1}\sqcup S_{2}\sqcup S_{3}$
with $\pi(S_{1}),\pi(S_{2})>0$,
\begin{equation}
\pi\left(S_{3}\right)\geqslant\mathsf{d}\left(S_{1},S_{2}\right)\cdot F\left(\min\left\{ \pi\left(S_{1}\right),\pi\left(S_{2}\right)\right\} \right).\label{eq:3-set-iso}
\end{equation}
\end{defn}

\begin{defn}[Isoperimetric Profile]
\label{def:isop_prof}For $A\in\mathscr{E}$ and $r\geqslant0$,
let $A_{r}:=\left\{ x\in\E:\mathsf{d}\left(x,A\right)<r\right\} $,
and define the \textit{Minkowski content} of $A$ under $\pi$ with
respect to $\mathsf{d}$ by 
\[
\pi^{+}\left(A\right)=\lim\inf_{r\to0^{+}}\frac{\pi\left(A_{r}\right)-\pi\left(A\right)}{r}.
\]
The \textit{isoperimetric profile of} $\pi$ with respect to the metric
$\mathsf{d}$ is
\begin{equation}
I_{\pi}\left(p\right):=\inf\left\{ \pi^{+}\left(A\right):A\in\mathscr{E},\pi\left(A\right)=p\right\} ,\qquad p\in\left(0,1\right).\label{eq:iso-profile-def}
\end{equation}
\end{defn}

We note briefly that the isoperimetric profile can be controlled explicitly
in many cases of interest; examples to this effect are provided in
Section~\ref{subsec:Examples-of-Isoperimetric}. The following is
a special case of a non-trivial result for distributions defined on
Riemannian manifolds which is the product of extensive research by
several authors, and holds specifically for $\mathsf{d}$ being the
natural metric induced by the given Riemannian structure; we recall
it here to emphasize that the notion of regularity on the isoperimetric
profile that we assume is reasonable. 
\begin{lem}[{\citealt[Theorem~1.8]{milman2009role}}]
If $U$ is convex and twice-continuously differentiable, then $I_{\pi}$
is symmetric about $\frac{1}{2}$ and concave. 
\end{lem}

\begin{defn}
\label{def:iso-ineq}We say that $\tilde{I}_{\pi}:\left(0,1\right)\to\left(0,\infty\right)$
is an \textit{isoperimetric minorant of $\pi$} if $\tilde{I}_{\pi}\leqslant I_{\pi}$
pointwise. We furthermore say that $\tilde{I}_{\pi}$ is \textit{regular}
if it is symmetric about $\frac{1}{2}$, continuous, and monotone
increasing on $\left(0,\frac{1}{2}\right]$.
\end{defn}

To begin with, we show that the existence of a regular isoperimetric
minorant is equivalent to the existence of a corresponding three-set
isoperimetric inequality.
\begin{lem}
\label{lem:iso-to-3set}$\pi$ has a regular isoperimetric minorant
$\tilde{I}_{\pi}$ w.r.t. the metric $\mathsf{d}$ $\iff$ $\pi$
satisfies a three-set isoperimetric inequality with metric $\mathsf{d}$
and function $F=\tilde{I}_{\pi}$ on $(0,\frac{1}{2}]$.
\end{lem}

\begin{proof}
$(\Leftarrow)$ Following \citet[Section~2]{bobkov1997some}, we may
consider only closed sets $A$ in (\ref{eq:iso-profile-def}). For
arbitrary closed $A\in\mathscr{E}$ with $\pi(A)\in(0,1)$, for $r>0$
let $A_{r}$ be as defined in Definition~\ref{def:isop_prof}. We
may take $S_{1}=A$, $S_{3}=A_{r}\setminus A=\{x\in\mathsf{E}\setminus A:\mathsf{d}(x,A)<r\}$
and $S_{2}=\mathsf{E}\setminus A_{r}$. From Definition~\ref{def:three-set-iso-ineq},
\[
\pi(A_{r})-\pi(A)\geqslant rF\left(\min\left\{ \pi(A),\pi(\mathsf{E}\setminus A_{r})\right\} \right),
\]
from which we obtain
\[
\pi^{+}(A)=\lim\inf_{r\to0^{+}}\frac{\pi(A_{r})-\pi(A)}{r}\geqslant F\left(\min\left\{ \pi(A),\pi(A^{\complement})\right\} \right),
\]
since for closed $A$, $\lim_{r\to0^{+}}\pi(\mathsf{E}\setminus A_{r})=1-\lim_{r\to0^{+}}\pi(A_{r})=\pi(A^{\complement})$.
Hence, $\tilde{I}_{\pi}(t)=F(\min\{t,1-t\})$ is an isoperimetric
minorant, symmetric on $(0,1)$, continuous and monotone increasing
on $(0,\frac{1}{2}]$ and hence regular.

$(\Rightarrow)$ Following \citet[Theorem~4.1 and Remark~4.2]{bobkov1997some},
for any Lipschitz $f:\mathsf{E}\to\left[0,1\right]$ one may write
\[
\pi\left(\left|\nabla f\right|\right)\geqslant\int_{0}^{1}\pi^{+}\left(f>t\right)\,{\rm d}t\geqslant\int_{0}^{1}I_{\pi}\left(\pi\left(f>t\right)\right)\,{\rm d}t\geqslant\int_{0}^{1}\tilde{I}_{\pi}\left(\pi\left(f>t\right)\right)\,{\rm d}t,
\]
where we have written $\left(f>t\right)$ for the set $\left\{ x\in\mathsf{E}:f\left(x\right)>t\right\} $,
and
\[
\left|\nabla f\left(x\right)\right|:=\underset{\mathsf{d}\left(x,y\right)\to0^{+}}{\lim\sup}\frac{\left|f\left(x\right)-f\left(y\right)\right|}{\mathsf{d}\left(x,y\right)}\in\left[0,\infty\right],
\]
is the modulus of the gradient of $f$. Now let $\E=S_{1}\sqcup S_{2}\sqcup S_{3}$
with $\pi(S_{1}),\pi(S_{2})>0$. If $\mathsf{d}(S_{1},S_{2})=0$ then
(\ref{eq:3-set-iso}) holds trivially, so henceforth we assume $\mathsf{d}(S_{1},S_{2})>0$.
Following the construction in the proof of \citet[Proposition~1.7]{ledoux2001concentration},
we define $f:\E\to\left[0,1\right]$ by $f(x):=\min\left\{ 1,\frac{\mathsf{d}\left(S_{1},x\right)}{\mathsf{d}\left(S_{1},S_{2}\right)}\right\} $.
This function is $\mathsf{d}\left(S_{1},S_{2}\right)^{-1}$-Lipschitz
on $S_{3}$ and flat elsewhere. It thus holds for this $f$ that
\[
\pi\left(\left|\nabla f\right|\right)\leqslant\pi\left(S_{3}\right)\cdot\mathsf{d}\left(S_{1},S_{2}\right)^{-1}\quad\implies\quad\pi\left(S_{3}\right)\geqslant\mathsf{d}\left(S_{1},S_{2}\right)\cdot\pi\left(\left|\nabla f\right|\right).
\]
We now seek a lower bound on $\pi\left(\left|\nabla f\right|\right)$,
for which we will make use of the isoperimetric profile. Observe that
for $t\in\left[0,1\right)$, it holds that
\begin{align*}
\left\{ x\in\E:f\left(x\right)>t\right\}  & =\left\{ x\in S_{2}:f\left(x\right)>t\right\} \sqcup\left\{ x\in S_{3}:f\left(x\right)>t\right\} \\
 & =S_{2}\sqcup\left\{ x\in S_{3}:f\left(x\right)>t\right\} .
\end{align*}
and hence that $\pi\left(f>t\right)\in\left[\pi\left(S_{2}\right),\pi\left(S_{2}\sqcup S_{3}\right)\right]=\left[\pi\left(S_{2}\right),\pi\left(S_{1}^{\complement}\right)\right]$.

Suppose now that $\max\left\{ \pi\left(S_{1}\right),\pi\left(S_{2}\right)\right\} \geqslant\frac{1}{2}$,
and without loss of generality that $\pi\left(S_{1}\right)\geqslant\frac{1}{2}\geqslant\pi\left(S_{2}\right)$.
It then follows for $t\in\left(0,1\right)$ that $\pi\left(f>t\right)\in\left[\pi\left(S_{2}\right),\pi\left(S_{1}^{\complement}\right)\right]\subseteq\left[0,\frac{1}{2}\right]$.
By monotonicity of $\tilde{I}_{\pi}$ on $\left(0,\frac{1}{2}\right]$,
it then holds that $\tilde{I}_{\pi}\left(\pi\left(f>t\right)\right)\geqslant\tilde{I}_{\pi}\left(\pi\left(S_{2}\right)\right)$
and thus that $\pi\left(\left|\nabla f\right|\right)\geqslant\tilde{I}_{\pi}\left(\pi\left(S_{2}\right)\right)$,
from which it follows that
\[
\pi\left(S_{3}\right)\geqslant\mathsf{d}\left(S_{1},S_{2}\right)\cdot\tilde{I}_{\pi}\left(\pi\left(S_{2}\right)\right)=\mathsf{d}\left(S_{1},S_{2}\right)\cdot\tilde{I}_{\pi}\left(\min\left\{ \pi\left(S_{1}\right),\pi\left(S_{2}\right)\right\} \right).
\]
On the contrary, suppose that $\max\left\{ \pi\left(S_{1}\right),\pi\left(S_{2}\right)\right\} <\frac{1}{2}$.
It then holds that any median $\nu$ of $f$ under $\pi$ lies in
$\left(0,1\right)$, so one can write that
\[
t\in\left[0,\nu\right]\implies\pi\left(f>t\right)\geqslant\frac{1}{2},\qquad t\in\left(\nu,1\right]\implies\pi\left(f>t\right)\leqslant\frac{1}{2}.
\]
Letting $\nu$ be such a median, observe that $t\mapsto\tilde{I}_{\pi}\left(\pi\left(f>t\right)\right)$
is increasing on $\left[0,\nu\right]$ and decreasing on $\left[\nu,1\right]$.
In particular, 
\[
\tilde{I}_{\pi}\left(\pi\left(f>t\right)\right)\geqslant\tilde{I}_{\pi}\left(\pi\left(f>0\right)\right)=\tilde{I}_{\pi}\left(\pi\left(S_{1}^{\complement}\right)\right),\qquad t\in\left[0,\nu\right],
\]
making use of the fact that $\pi\ll{\rm Leb}$. Similarly, 
\[
\tilde{I}_{\pi}\left(\pi\left(f>t\right)\right)\geqslant\tilde{I}_{\pi}\left(\pi\left(f>u\right)\right),\qquad\nu<t\leqslant u\leqslant1,
\]
and therefore
\[
\tilde{I}_{\pi}\left(\pi\left(f>t\right)\right)\geqslant\lim_{u\to1^{-}}\tilde{I}_{\pi}\left(\pi\left(f>u\right)\right)=\tilde{I}_{\pi}\left(\lim_{u\to1^{-}}\pi\left(f>u\right)\right)\geqslant\tilde{I}_{\pi}\left(\pi\left(S_{2}\right)\right),
\]
taking limits in $u$ and applying continuity of $\tilde{I}_{\pi}$.

We thus decompose
\begin{align*}
\int_{0}^{1}\tilde{I}_{\pi}\left(\pi\left(f>t\right)\right)\,{\rm d}t & =\int_{0}^{\nu}\tilde{I}_{\pi}\left(\pi\left(f>t\right)\right)\,{\rm d}t+\int_{\nu}^{1}\tilde{I}_{\pi}\left(\pi\left(f>t\right)\right)\,{\rm d}t\\
 & \geqslant\int_{0}^{\nu}\tilde{I}_{\pi}\left(\pi\left(S_{1}^{\complement}\right)\right)\,{\rm d}t+\int_{\nu}^{1}\tilde{I}_{\pi}\left(\left(\pi\left(S_{2}\right)\right)\right)\,{\rm d}t\\
 & =\nu\cdot\tilde{I}_{\pi}\left(\pi\left(S_{1}^{\complement}\right)\right)+\left(1-\nu\right)\cdot\tilde{I}_{\pi}\left(\pi\left(S_{2}\right)\right)\\
 & \geqslant\min\left\{ \tilde{I}_{\pi}\left(\pi\left(S_{1}^{\complement}\right)\right),\tilde{I}_{\pi}\left(\pi\left(S_{2}\right)\right)\right\} \\
 & =\min\left\{ \tilde{I}_{\pi}\left(\pi\left(S_{1}\right)\right),\tilde{I}_{\pi}\left(\pi\left(S_{2}\right)\right)\right\} ,
\end{align*}
from which we may conclude.
\end{proof}
We now show that given a three-set isoperimetric inequality with a
monotone increasing $F$, together with the \textit{\emph{close coupling}}\textit{
}\textit{\emph{assumption}} on the Markov kernel, one may deduce a
lower bound on the conductance of any set for that Markov kernel.
The proof, housed in Appendix~\ref{sec:Proofs-of-auxiliary}, follows
closely that of \citet[Lemma~6]{Dwivedi-Chen-Wainwright-JMLR:v20:19-306},
which itself is based on several earlier works.
\begin{lem}
\label{lem:iso-conductance-A-close-coupling}Suppose that $\pi$ satisfies
a three-set isoperimetric inequality with metric $\mathsf{d}$ and
function $F$ monotone increasing on $\left(0,\frac{1}{2}\right]$.
Let $P$ be a $\left(\mathsf{d},\delta,\varepsilon\right)$-close
coupling, $\pi$-invariant Markov kernel. Then for any $A\in\mathscr{E}$
with $0<\pi(A)\leqslant\frac{1}{2}$,
\[
\frac{\left(\pi\otimes P\right)\left(A\times A^{\complement}\right)}{\pi\left(A\right)}\geqslant\sup_{\theta\in\left(0,1\right)}\min\left\{ \frac{1}{2}\cdot\left(1-\theta\right)\cdot\varepsilon,\frac{1}{4}\cdot\varepsilon\cdot\delta\cdot\theta\cdot\left(F/{\rm id}\right)\left(\theta\cdot\pi\left(A\right)\right)\right\} ,
\]
where $F/{\rm id}$ is the function $x\mapsto F(x)/x$.
\end{lem}

Since $\tilde{I}_{\pi}/{\rm id}$ is decreasing for a concave, regular
isoperimetric minorant $\tilde{I}_{\pi}$, we obtain the following
bounds on the conductance profile by combining Lemma~\ref{lem:iso-to-3set},
Lemma~\ref{lem:iso-conductance-A-close-coupling} and Definition~\ref{def:conductance-profile}.
We emphasize that concavity of $\tilde{I}_{\pi}$ is crucial for obtaining
non-zero lower bounds. Considering functions on $(0,\frac{1}{2}]$
of the form $p\mapsto c\cdot p^{k}$ the critical case is $k=1$ and
any $k>1$ implies only a conductance profile lower bound of $0$;
our examples give isoperimetric minorants of the form $p\mapsto c\cdot p\cdot\log\left(\frac{1}{p}\right)^{r}$
on $(0,\frac{1}{2}]$ for $r\in[0,1]$. Additionally, it is well-known
that the uniform measure on the sphere $\mathbb{S}^{d}\subset\mathbb{R}^{d+1}$
satisfies $I\left(p\right)\asymp p^{\frac{d-1}{d}}$; see e.g. Section~9
of \citet{bobkov1997some}.
\begin{cor}
\label{cor:iso-couple-conductance}Suppose that $\tilde{I}_{\pi}$
is a regular, concave isoperimetric minorant of $\pi$ w.r.t. the
metric $\mathsf{d}$. Let $P$ be a $(\mathsf{d},\delta,\varepsilon)$-close
coupling, $\pi$-invariant Markov kernel. Then for any $v\in\left(0,\frac{1}{2}\right]$,
\begin{align*}
\Phi_{P}\left(v\right) & \geqslant\sup_{\theta\in\left[0,1\right]}\min\left\{ \frac{1}{2}\cdot\left(1-\theta\right)\cdot\varepsilon,\frac{1}{4}\cdot\varepsilon\cdot\delta\cdot\theta\cdot\left(\tilde{I}_{\pi}/{\rm id}\right)\left(\theta\cdot v\right)\right\} \\
 & \geqslant\frac{1}{4}\cdot\varepsilon\cdot\min\left\{ 1,\frac{1}{2}\cdot\delta\cdot\frac{\tilde{I}_{\pi}\left(\frac{1}{2}\cdot v\right)}{\frac{1}{2}\cdot v}\right\} .
\end{align*}
\end{cor}

\begin{rem}
For obtaining tighter bounds on the conductance when $\delta\cdot\tilde{I}_{\pi}\left(\frac{1}{2}\right)$
is sufficiently small, one can take $\theta=1-\delta\cdot\tilde{I}_{\pi}\left(\frac{1}{2}\right)$
to see that
\[
\Phi_{P}^{*}\geqslant\frac{1}{2}\cdot\varepsilon\cdot\delta\cdot\tilde{I}_{\pi}\left(\frac{1}{2}\cdot\left(1-\delta\cdot\tilde{I}_{\pi}\left(\frac{1}{2}\right)\right)\right)\to\frac{1}{2}\cdot\varepsilon\cdot\delta\cdot\tilde{I}_{\pi}\left(\frac{1}{2}\right),
\]
as $\delta\cdot\tilde{I}_{\pi}\left(\frac{1}{2}\right)\to0^{+}$,
which is a slight improvement on the non-asymptotic bound $\Phi_{P}^{*}=\Phi_{P}\left(\frac{1}{2}\right)\geqslant\frac{1}{2}\cdot\varepsilon\cdot\delta\cdot\tilde{I}_{\pi}\left(\frac{1}{4}\right)$.
\end{rem}

The following theorem is the culmination of this and the preceding
section; see Figure~\ref{fig:secs2-3-map}. The proof is in Appendix~\ref{sec:Proofs-of-auxiliary}
For the mixing time bound, $v_{*}^{-1}$ will typically increase quite
rapidly as $\delta$ decreases. The bounds suggest three-stage behaviour
when $v_{*}<\frac{1}{2}$, recalling that $u_{0}=\chi^{2}(\mu,\pi)$.
The first term is active when $u_{0}>4\cdot v_{*}^{-1}>8$, either
because $u_{0}$ is extremely large or $\delta$ is large, while the
second term is active when $u_{0}>8$ and the third term is active
when $u_{0}>\varepsilon_{{\rm Mix}}$. Of course, if $u_{0}\leqslant\varepsilon_{{\rm Mix}}$,
then one may take $n=0$.
\begin{thm}
\label{thm:IP-to-mix}Let $\pi$ have a regular, concave isoperimetric
minorant $\tilde{I}_{\pi}$ w.r.t. the metric $\mathsf{d}$, and $P$
be a $\left(\mathsf{d},\delta,\varepsilon\right)$-close coupling,
$\pi$-reversible, positive Markov kernel. Then
\begin{enumerate}
\item For $v\in\left(0,\frac{1}{2}\right]$, $\Phi_{P}\left(v\right)\geqslant\frac{1}{4}\cdot\varepsilon\cdot\min\left\{ 1,\frac{1}{2}\cdot\delta\cdot\left(\tilde{I}_{\pi}/{\rm id}\right)\left(\frac{1}{2}\cdot v\right)\right\} $,
\item $\Phi_{P}^{*}\geqslant\frac{1}{4}\cdot\varepsilon\cdot\min\left\{ 1,2\cdot\delta\cdot\tilde{I}_{\pi}\left(\frac{1}{4}\right)\right\} $,
\item ${\rm \gamma}_{P}\geqslant\frac{1}{2}\cdot\left[\Phi_{P}^{*}\right]^{2}\geqslant2^{-5}\cdot\varepsilon^{2}\cdot\min\left\{ 1,4\cdot\delta^{2}\cdot\tilde{I}_{\pi}\left(\frac{1}{4}\right)^{2}\right\} .$
\end{enumerate}
Furthermore, let $\varepsilon_{{\rm Mix}}\in\left(0,8\right)$, $\mu\ll\pi$
be a probability measure and $u_{0}:=\chi^{2}(\mu,\pi)$. To ensure
that $\chi^{2}(\mu P^{n},\pi)\leqslant\varepsilon_{{\rm Mix}}$, it
suffices to take 
\begin{align*}
n & \geqslant2+2^{6}\cdot\varepsilon^{-2}\cdot\max\left\{ \log\left(\frac{u_{0}}{4\cdot v_{*}^{-1}}\right),0\right\} \\
 & +2^{8}\cdot\varepsilon^{-2}\cdot\delta^{-2}\cdot\int_{\max\left\{ \min\left\{ 2\cdot u_{0}^{-1},1/4\right\} ,v_{*}/2\right\} }^{1/4}\frac{\xi}{\tilde{I}_{\pi}\left(\xi\right)^{2}}\,{\rm d}\xi\\
 & +2^{4}\cdot\max\left\{ 1,2^{-2}\cdot\delta^{-2}\cdot\tilde{I}_{\pi}\left(\frac{1}{4}\right)^{-2}\right\} \cdot\varepsilon^{-2}\cdot\log\left(\max\left\{ \frac{\min\left\{ u_{0},8\right\} }{\varepsilon_{{\rm Mix}}},1\right\} \right).
\end{align*}
where
\begin{equation}
v_{*}:=\min\left\{ \frac{1}{2},\max\left\{ 0,\sup\left\{ v>0:1\leqslant\frac{1}{2}\cdot\delta\cdot\frac{\tilde{I}_{\pi}\left(\frac{1}{2}\cdot v\right)}{\frac{1}{2}\cdot v}\right\} \right\} \right\} .\label{eq:v*_defn}
\end{equation}

\end{thm}

We conclude our general results with a result that can be used to
establish the close-coupling condition for Metropolis kernels, which
we will use in the sequel to analyze both the RWM and pCN Markov chains.
\begin{lem}
\label{lem:met-tv-bound}Let $Q$ be a $\nu$-reversible Markov kernel
where $\nu\gg\pi$ is a $\sigma$-finite measure, $P$ be the $\pi$-reversible
Metropolis kernel with proposal $Q$ and $\alpha_{0}:=\inf_{x\in\mathsf{E}}\alpha\left(x\right)$.
Then 
\[
\left\Vert P\left(x,\cdot\right)-P\left(y,\cdot\right)\right\Vert _{{\rm TV}}\leqslant\left\Vert Q\left(x,\cdot\right)-Q\left(y,\cdot\right)\right\Vert _{{\rm TV}}+1-\alpha_{0},\qquad x,y\in\mathsf{E}.
\]
\end{lem}

\begin{proof}
Let $x,y\in\mathsf{E}$ be arbitrary. We write $\varpi=\frac{\mathrm{d}\pi}{\mathrm{d}\nu}$.
We construct a specific coupling of $\left(X^{'},Y^{'}\right)$ such
that $X^{'}\sim P\left(x,\cdot\right)$ and $Y^{'}\sim P\left(y,\cdot\right)$.
Without loss of generality, we may assume that $\varpi\left(x\right)\geqslant\varpi\left(y\right)$.
Let $\left(W_{x},W_{y}\right)$ be drawn from a maximal coupling of
$Q\left(x,\cdot\right)$ and $Q\left(y,\cdot\right)$, so that
\[
\mathbb{P}\left(W_{x}=W_{y}\right)=1-\left\Vert Q\left(x,\cdot\right)-Q\left(y,\cdot\right)\right\Vert _{{\rm TV}}.
\]
With $\mathcal{U}\sim{\rm Unif}\left(0,1\right)$, we complete the
specification of the distribution of $\left(X^{'},Y^{'}\right)$ via
\[
X^{'}\mid\left\{ W_{x}=w_{x},W_{y}=w_{y},\mathcal{U}=u\right\} =\begin{cases}
w_{x} & u\leqslant\varpi\left(w_{x}\right)/\varpi\left(x\right),\\
x & u>\varpi\left(w_{x}\right)/\varpi\left(x\right),
\end{cases}
\]
and
\[
Y^{'}\mid\left\{ W_{x}=w_{x},W_{y}=w_{y},\mathcal{U}=u\right\} =\begin{cases}
w_{y} & u\leqslant\varpi\left(w_{y}\right)/\varpi\left(y\right),\\
y & u>\varpi\left(w_{y}\right)/\varpi\left(y\right).
\end{cases}
\]
On the event $\left\{ W_{x}=W_{y}\right\} \cap\left\{ X^{'}=W_{x}\right\} $,
we have $X^{'}=Y^{'}=W_{x}$ since $\varpi\left(x\right)\geqslant\varpi\left(y\right)$.
Hence, 
\begin{align*}
\mathbb{P}\left(X^{'}=Y^{'}\right) & \geqslant\mathbb{P}\left(W_{x}=W_{y},X^{'}=W_{x}\right)\\
 & \geqslant\mathbb{P}\left(W_{x}=W_{y}\right)+\mathbb{\mathbb{P}}\left(X^{'}=W_{x}\right)-1\\
 & =\left(1-\left\Vert Q\left(x,\cdot\right)-Q\left(y,\cdot\right)\right\Vert _{{\rm TV}}\right)+\mathbb{\mathbb{P}}\left(X^{'}=W_{x}\right)-1\\
 & \geqslant\left(1-\left\Vert Q\left(x,\cdot\right)-Q\left(y,\cdot\right)\right\Vert _{{\rm TV}}\right)+\alpha_{0}-1\\
 & =\alpha_{0}-\left\Vert Q\left(x,\cdot\right)-Q\left(y,\cdot\right)\right\Vert _{{\rm TV}}.
\end{align*}
We conclude by the coupling inequality: $\left\Vert P\left(x,\cdot\right)-P\left(y,\cdot\right)\right\Vert _{{\rm TV}}\leqslant\mathbb{P}\left(X^{'}\neq Y^{'}\right)$.
\end{proof}

\subsection{\label{subsec:Examples-of-Isoperimetric}Examples of isoperimetric
profiles}

In this section, we provide some concrete examples of probability
measures for which the isoperimetric profile admits regular, concave
and tractable isoperimetric profiles or minorants. We first describe
explicitly the profiles associated to specific measures, and then
describe some general results which hold over well-behaved families
of measures. Most results in this subsection are not originally ours,
and are included in order to provide context on the breadth of applicability
of our main results.

\subsubsection{Specific examples}
\begin{example}
Let $\varphi_{\gamma}$, $\Phi_{\gamma}$ denote the PDF and CDF respectively
of the one-dimensional standard Gaussian measure. For any $d\in\mathbb{N}$,
the isoperimetric profile of $\pi=\mathcal{N}\left(0,I_{d}\right)$
w.r.t. $\left|\cdot\right|$ satisfies $I_{\pi}\left(p\right)=\left(\varphi_{\gamma}\circ\Phi_{\gamma}^{-1}\right)\left(p\right)$
\citep{borell1975brunn,sudakov1978extremal}, which satisfies
\[
\lim_{p\to0^{+}}\frac{I_{\pi}\left(p\right)}{p\cdot\left(\log\frac{1}{p}\right)^{1/2}}=\sqrt{2},
\]
see, e.g., \citet{Barthe2000}. Observing that $I_{\pi}$ is concave
and regular, one can deduce
\[
I_{\pi}\left(p\right)\geqslant2\cdot I_{\pi}\left(\frac{1}{2}\right)\cdot\min\left\{ p,1-p\right\} =\left(\frac{2}{\pi}\right)^{1/2}\cdot\min\left\{ p,1-p\right\} .
\]
\end{example}

\begin{example}
For the Laplace measure $\pi\left({\rm d}x\right)\propto\exp\left(-\left|x\right|\right)\,{\rm d}x$
in one dimension, the isoperimetric profile w.r.t. $\left|\cdot\right|$
is given by $I_{\pi}\left(p\right)=\min\left\{ p,1-p\right\} $ \citep[see, e.g.,][]{bobkov1999isoperimetric}.
\end{example}

\begin{example}
For the Subbotin measure $\pi\left({\rm d}x\right)\propto\exp\left(-\left|x\right|^{\alpha}\right)\,{\rm d}x$
in one dimension, with $\alpha\in\left(1,2\right)$, it holds that
the isoperimetric profile w.r.t. $\left|\cdot\right|$ can be bounded
from below for $p\in\left(0,\frac{1}{2}\right]$ as $I_{\pi}\left(p\right)\geqslant K\left(\alpha\right)\cdot p\cdot\left(\log\frac{1}{p}\right)^{1-1/\alpha}$
for some $K\left(\alpha\right)>0$ \citep[see, e.g.,][]{latala2000between,barthe2006interpolated}.
\end{example}

\subsubsection{Functional inequalities}

Many analyses of MCMC algorithms restrict to considering strongly
log-concave targets $\pi$, i.e. $U$ is strongly convex, to obtain
quantitative bounds. This means that for potentials with inhomogeneous
local convexity properties, but good global isoperimetric properties
(e.g. strongly convex in the tails, weakly convex in the center of
the space), complexity bounds can be somewhat pessimistic. In this
subsection, we give some examples of how to estimate isoperimetric
profiles when $\pi$ is only log-concave, given additional quantitative
information about functional inequalities which they satisfy.
\begin{example}
\label{exa:convex-poincare-to-profile}For any log-concave $\pi$,
there exists $\gamma_{\pi}>0$ such that
\[
\forall f\,\text{locally Lipschitz,}\quad\pi\left(\left|\nabla f\right|^{2}\right)\geqslant\gamma_{\pi}\cdot{\rm Var}_{\pi}\left(f\right),
\]
that is, $\pi$ satisfies a Poincaré inequality \citep[see, e.g.,][Theorem~4.6.3]{bakry2014analysis}.
It is also known that Poincaré inequalities can be translated into
${\rm L}^{1}$-Poincaré inequalities, with explicit control of the
constants \citep{cattiaux2020poincare}. Finally, ${\rm L}^{1}$-Poincaré
inequalities are equivalent to isoperimetric inequalities with respect
to $\left|\cdot\right|$ of the form $I_{\pi}\left(p\right)\geqslant c\cdot\min\left\{ p,1-p\right\} $,
with the same constant \citep[see, e.g.,][]{kolesnikov2007modified}.
Combining these facts, if $\pi$ is log-concave, one can deduce that
\[
I_{\pi}\left(p\right)\geqslant\frac{1}{6}\cdot\gamma_{\pi}^{1/2}\cdot\min\left\{ p,1-p\right\} .
\]
\end{example}

\begin{example}
Consider log-concave $\pi$ satisfying a logarithmic Sobolev inequality
with constant $\lambda_{\pi}$, namely,
\[
\forall f\,\text{locally Lipschitz,}\quad\pi\left(\left|\nabla f\right|^{2}\right)\geqslant\lambda_{\pi}\cdot{\rm Ent}_{\pi}\left(f^{2}\right),
\]
where for positive $f$, ${\rm Ent}_{\pi}\left(f\right):=\pi\left(f\cdot\log f\right)-\pi\left(f\right)\cdot\log\pi\left(f\right)$.
A result of \citet{ledoux2011concentration} then gives the following
bound for the isoperimetric profile w.r.t. $\left|\cdot\right|$,
\[
I_{\pi}\left(p\right)\geqslant\frac{1}{34}\cdot\lambda_{\pi}^{1/2}\cdot p\cdot\left(\log\frac{1}{p}\right)^{1/2},\qquad p\in\left(0,\frac{1}{2}\right].
\]
\end{example}

\begin{example}
\label{exa:q-log-sob}For log-concave $\pi$ satisfying a $q$-log-Sobolev
inequality \citep[see][]{bobkov2005entropy}, i.e.
\[
\forall f\,\text{locally Lipschitz,}\quad D\cdot{\rm Ent}_{\pi}\left(\left|f\right|^{q}\right)\leqslant\pi\left(\left|\nabla f\right|_{q}^{q}\right),
\]
a result of \citet{milman2009rolefunciso} establishes that for the
isoperimetric profile w.r.t. $\left|\cdot\right|$, 
\[
I_{\pi}\left(p\right)\geqslant c_{q}\cdot D\cdot p\cdot\left(\log\frac{1}{p}\right)^{1/q}\quad\text{for }p\in\big(0,1/2\big],
\]
where $c_{q}>0$ is universal. For $q=2$, this entails the standard
log-Sobolev inequality; for $q\in\left[1,2\right)$, the assumption
becomes stronger and corresponds to potentials which have tail behaviour
like $U\left(x\right)\sim\left|x\right|^{q_{*}}$, where $q^{-1}+q_{*}^{-1}=1$.
For $q=1$, the assumption is yet stronger and corresponds intuitively
to potentials which have tail behaviour like $U\left(x\right)\sim\exp\left(c\cdot\left|x\right|\right)$
for some $c>0$.
\end{example}

\subsubsection{Transfer principles}

Another practical aspect of working with isoperimetric profiles is
that they are often preserved under suitably regular perturbations,
some of which we detail here. These transfer principles can be used
to accommodate potentials that are not convex. The first of these
concerns the transfer of isoperimetric properties under Lipschitz
transport; related statements are made in \citet{barthe2001levels}.
\begin{lem}
\label{lem:transport-isoperimetry}For $i=1,2$, let $\mu_{i}$ be
a probability measure on the metric space $\left(\E_{i},\mathsf{d}_{i}\right)$.
Suppose that these measures are related through transport as
\[
\mu_{2}=T_{\#}\mu_{1},
\]
where $T:\E_{1}\to\E_{2}$ is a Lipschitz bijection. Then, for $\tilde{I}$
any isoperimetric minorant of $\mu_{1}$ w.r.t. $\mathsf{d}_{1}$,
it holds that $\left|T\right|_{{\rm Lip}}^{-1}\cdot\tilde{I}$ is
an isoperimetric minorant of $\mu_{2}$ w.r.t. the metric $\mathsf{d}_{2}$.
In particular, if $T$ is also an isometry, then $\mu_{1}$ and $\mu_{2}$
have identical isoperimetric profiles w.r.t. their respective metrics.
\end{lem}

\begin{proof}
Let $A\subseteq\mathsf{E}_{2}$ be measurable and $A_{r}=\left\{ x\in\E_{2}:\mathsf{d}_{2}\left(x,A\right)\leqslant r\right\} $.
Write $B:=T^{-1}\left(A\right)$, and compute that
\begin{align*}
A_{r} & =\left\{ x\in\E_{2}:\mathsf{d}_{2}\left(x,A\right)\leqslant r\right\} \\
 & =\left\{ x\in\mathcal{\E}_{2}:\mathsf{d_{2}}\left(T\left(T^{-1}\left(x\right)\right),T\left(B\right)\right)\leqslant r\right\} \\
 & \supseteq\left\{ T\left(y\right)\in\E_{2}:\mathsf{d}_{1}\left(y,B\right)\leqslant\left|T\right|_{{\rm Lip}}^{-1}\cdot r\right\} \\
 & =T\left(B_{\left|T\right|_{{\rm Lip}}^{-1}\cdot r}\right),
\end{align*}
where $B_{s}:=\left\{ y\in\E_{1}:\mathsf{d}_{1}\left(y,B\right)\leqslant s\right\} $.
Then
\[
\frac{\mu_{2}\left(A_{r}\right)-\mu_{2}\left(A\right)}{r}\geqslant\frac{\mu_{2}\left(T\left(B_{\left|T\right|_{{\rm Lip}}^{-1}\cdot r}\right)\right)-\mu_{2}\left(T\left(B\right)\right)}{r}=\frac{\mu_{1}\left(B_{\left|T\right|_{{\rm Lip}}^{-1}\cdot r}\right)-\mu_{1}\left(B\right)}{r},
\]
whence it follows that
\begin{align*}
\mu_{2,\mathsf{d}_{2}}^{+}\left(A\right) & =\lim\inf_{r\to0^{+}}\frac{\mu_{2}\left(A_{r}\right)-\mu_{2}\left(A\right)}{r}\\
 & \geqslant\lim\inf_{r\to0^{+}}\frac{\mu_{1}\left(B_{\left|T\right|_{{\rm Lip}}^{-1}\cdot r}\right)-\mu_{1}\left(B\right)}{r}\\
 & \geqslant\lim\inf_{s\to0^{+}}\frac{\mu_{1}\left(B_{s}\right)-\mu_{1}\left(B\right)}{\left|T\right|_{{\rm Lip}}\cdot s}\\
 & =\left|T\right|_{{\rm Lip}}^{-1}\cdot\mu_{1,\mathsf{d}_{1}}^{+}\left(B\right).
\end{align*}
By definition of isoperimetric minorants and $\mu_{2}$, it holds
that $\mu_{1,\mathsf{d}_{1}}^{+}\left(B\right)\geqslant\tilde{I}\left(\mu_{1}\left(B\right)\right)=\tilde{I}\left(\mu_{2}\left(A\right)\right)$,
so the first result follows. For the second result, note that bijective
isometries from $\E_{1}$ to $\E_{2}$ satisfy $\left|T\right|_{\mathrm{Lip}}=\left|T^{-1}\right|_{{\rm \mathrm{Lip}}}=1$
(noting that these Lipschitz norms are technically defined on different
spaces), and that the isoperimetric profile is always an isoperimetric
minorant. Applying the first result in both directions allows us to
conclude.
\end{proof}
The following result will be used frequently in the sequel.
\begin{lem}
\label{lem:Bounds-on-isoperimetric}For an $m$-strongly convex potential
$U$, with $\mathsf{d}=\left|\cdot\right|$, we have
\[
I_{\pi}\left(p\right)\geqslant m^{1/2}\cdot\left(\varphi_{\gamma}\circ\Phi_{\gamma}^{-1}\right)\left(p\right)\geqslant C_{\ell}\cdot m^{1/2}\cdot\min\left\{ p,1-p\right\} \cdot\left(\log\frac{1}{\min\left\{ p,1-p\right\} }\right)^{1/2},
\]
where $C_{\ell}:=\left(\frac{2}{\pi\cdot\log\left(2\right)}\right)^{1/2}\geqslant0.958357$.
Moreover, if we let $\tilde{I}_{\pi}\left(p\right):=m^{1/2}\cdot\left(\varphi_{\gamma}\circ\Phi_{\gamma}^{-1}\right)\left(p\right)$,
which is regular and concave, then 
\[
\tilde{I}_{\pi}\left(\frac{1}{4}\right)=m^{1/2}\cdot C_{{\rm \gamma}},
\]
where $C_{{\rm \gamma}}:=\left(\varphi\circ\Phi^{-1}\right)\left(\frac{1}{4}\right)\geqslant0.3177765$.
\end{lem}

\begin{proof}
By the contraction principle of \citet{caffarelli2000monotonicity},
it is known that if one takes $\mu=\mathcal{N}\left(0,m^{-1}\cdot I_{d}\right)$,
then there exists a 1-Lipschitz mapping $T$ which pushes $\mu$ onto
$\pi$, and hence a $m^{-1/2}$-Lipschitz mapping which pushes $\mathcal{N}\left(0,I_{d}\right)$
onto $\pi$. By Lemma~\ref{lem:transport-isoperimetry}, the first
inequality follows. For the second inequality, observe that the function
\[
g(p)=\frac{\left(\varphi_{\gamma}\circ\Phi_{\gamma}^{-1}\right)\left(p\right)}{p\cdot\left(\log\left(\frac{1}{p}\right)\right)^{1/2}},
\]
is minimized on $(0,\frac{1}{2}]$ at $\frac{1}{2}$; see Figure~\ref{fig:tilde_I},
and defining $C_{\ell}=g\left(\frac{1}{2}\right)$ gives the second
inequality. The final claim is a direct computation.
\end{proof}
Isoperimetric profiles also transfer under bounded changes of measure
in the style of, e.g., \citet{Holley1987}. We provide the following
result that demonstrates this.
\begin{prop}
\label{exa:holley-stroock}Let $\mu$ be a probability measure on
$\mathbb{R}^{d}$ with positive density with respect to Lebesgue.
Suppose $\tilde{I}_{\mu}$ is a regular isoperimetric minorant of
$\mu$ w.r.t. a metric $\mathsf{d}$, and $\nu$ is a probability
measure equivalent to $\mu$ with $\frac{{\rm d}\nu}{{\rm d}\mu}\in\left[c_{1},c_{2}\right]$
for some $0<c_{1}\leqslant c_{2}<\infty$. Then $p\mapsto c_{1}\cdot I_{\mu}\left(c_{2}^{-1}\cdot\min\left\{ p,1-p\right\} \right)$
for $p\in(0,1)$ is a regular isoperimetric minorant of $\nu$ w.r.t.
$\mathsf{d}$.
\end{prop}

\begin{proof}
We first deduce that $\mu$ satisfies a 3-set isoperimetric inequality
with function $F_{\mu}(t)=\tilde{I}_{\mu}(t)$ for $t\in(0,\frac{1}{2}]$
by Lemma~\ref{lem:iso-to-3set}. Then, for any measurable partition
$\mathsf{E}=S_{1}\sqcup S_{2}\sqcup S_{3}$ we have 
\begin{align*}
\nu(S_{3}) & \geqslant c_{1}\mu(S_{3})\\
 & \geqslant c_{1}\mathsf{d}(S_{1},S_{2})F_{\mu}(\min\{\mu(S_{1}),\mu(S_{2})\})\\
 & \geqslant c_{1}\mathsf{d}(S_{1},S_{2})F_{\mu}(c_{2}^{-1}\min\{\nu(S_{1}),\nu(S_{2})\}),
\end{align*}
so $\nu$ satisfies a 3-set isoperimetric inequality with metric $\mathsf{d}$
and function $F_{\nu}(t)=c_{1}F_{\mu}(c_{2}^{-1}t)$. We conclude
by Lemma~\ref{lem:iso-to-3set}.
\end{proof}
\begin{rem}
Proposition~\ref{exa:holley-stroock} provides\textcolor{red}{{} }an
improvement of \citet[Lemma 2]{belloni2009computational}, who (essentially)
work in the setting of probability measures whose density can be expressed
as
\[
\pi\left(x\right)=\exp\left(-\frac{1}{2}\left|x\right|^{2}-V\left(x\right)-\xi\left(x\right)\right),
\]
where $V$ is convex, and $\xi$ is uniformly bounded above and below.
They deduce a `pseudo'-three-set isoperimetric inequality of the form
\[
\pi\left(S_{3}\right)\geqslant c\cdot\exp\left(-{\rm Osc}\left(\xi\right)\right)\cdot t\cdot\exp\left(-\frac{1}{2}\cdot t^{2}\right)\cdot\min\left\{ \pi\left(S_{1}\right),\pi\left(S_{2}\right)\right\} ,
\]
where ${\rm Osc}\left(\xi\right):={\rm ess}\sup\xi-{\rm ess}\inf\xi$,
$t=\mathsf{d}\left(S_{1},S_{2}\right)$, $\mathsf{d}=\left|\cdot\right|$
, and $c>0$ is an explicit constant.

In fact, using Caffarelli's result, one can see that $\exp\left(-\frac{1}{2}\left|x\right|^{2}-V\left(x\right)\right)\,{\rm d}x$
will admit $\varphi_{\gamma}\circ\Phi_{\gamma}^{-1}$ as an isoperimetric
minorant, and so combining this with Proposition~\ref{exa:holley-stroock},
one obtains that $\pi$ satisfies the inequality
\[
\pi\left(S_{3}\right)\geqslant\exp\left(-{\rm Osc}\left(\xi\right)\right)\cdot\mathsf{d}\left(S_{1},S_{2}\right)\cdot\left(\varphi_{\gamma}\circ\Phi_{\gamma}^{-1}\right)\left(\min\left\{ \pi\left(S_{1}\right),\pi\left(S_{2}\right)\right\} \right),
\]
which relates to a result of \citet{bobkov2010perturbations} (which
is stronger, but only valid in dimension $1$). Combining this observation
with the spectral profile approach and the other calculations of \citet{belloni2009computational},
it seems likely that one could improve the dimension-dependence of
their results. We do not pursue this claim further in this work.
\end{rem}

\begin{example}
It is known that under suitable convexity assumptions that isoperimetric
profiles `almost' tensorize, i.e. that the isoperimetric profile of
$\pi^{\otimes n}$ satisfies $\inf_{n\geqslant1}I_{\pi^{\otimes n}}\geqslant c\cdot I_{\pi}$
for some constant $c>0$. In particular, the isoperimetric profile
of product measures can be lower-bounded independently of dimension.
We refer the reader to \citet{bobkov1997isoperimetric,roberto2010isoperimetry}
for details.
\end{example}

\begin{figure}
\begin{centering}
\includegraphics[width=0.5\textwidth]{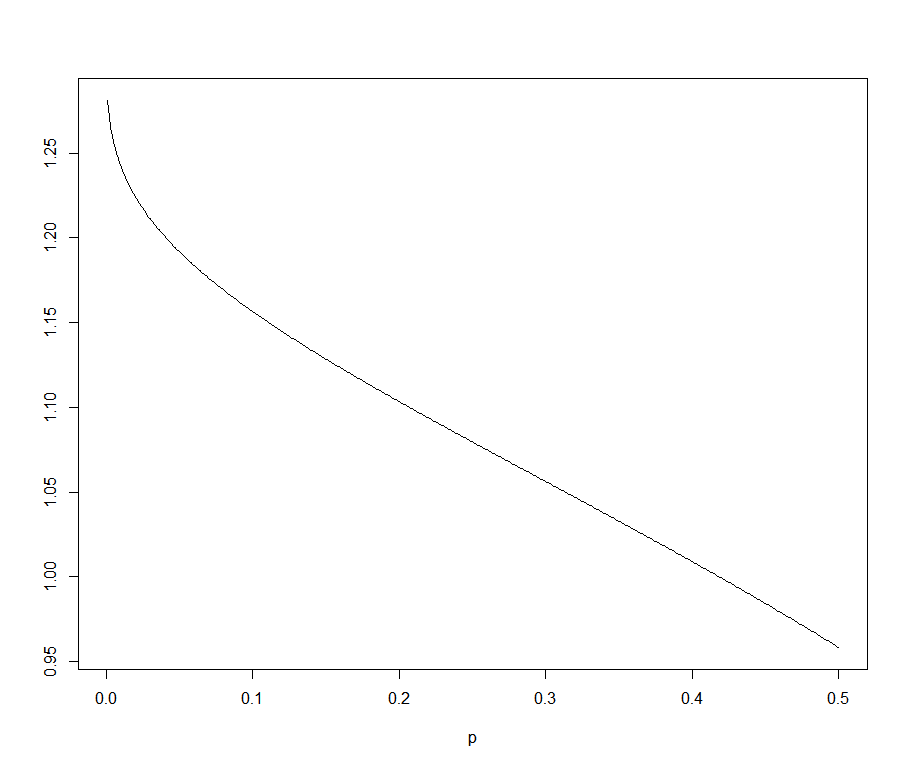}
\par\end{centering}
\caption{A plot of $p\protect\mapsto\frac{\left(\varphi_{\gamma}\circ\Phi_{\gamma}^{-1}\right)\left(p\right)}{p\cdot\left[\log\left(\frac{1}{p}\right)\right]^{1/2}}$.}
\label{fig:tilde_I}

\end{figure}

\subsection{From spectral profile to mixing times: examples}

In this section, we describe how to combine an isoperimetric profile
for $\pi$ with the close coupling condition assumption on $P$ to
estimate mixing times for the chain. Write $u_{0}=\chi^{2}(\mu,\pi)$
for the initial $\chi^{2}$-divergence. 

In all of our examples, our isoperimetric minorants take the form
$\tilde{I}_{\pi}\left(p\right)=c\cdot p\cdot\log\left(\frac{1}{p}\right)^{r}$
for $p\in\left(0,\frac{1}{2}\right]$, with $r\in\left[0,1\right]$.
We briefly recap how the $r$ parameter maps onto simple assumptions,
before providing explicit calculations:
\begin{itemize}
\item $r=0$ is `exponential-type' isoperimetry, which holds for any log-concave
measure and corresponds roughly to potentials which have a tail growth
of order or faster than $\left|x\right|$.
\item $r=\frac{1}{2}$ is `Gaussian-type' isoperimetry, which holds for
any log-concave measure with sub-Gaussian tails.
\item $r\in\left[0,\frac{1}{2}\right]$ corresponds to `intermediate' isoperimetry,
and roughly corresponds to potentials which have tail growth $U\left(x\right)\sim\left|x\right|^{\frac{1}{1-r}}\in\left(\left|x\right|^{1},\left|x\right|^{2}\right)$.
\item $r\in\left[\frac{1}{2},1\right]$ corresponds to `light-tailed' isoperimetry,
and roughly corresponds to potentials which have tail growth $U\left(x\right)\sim\left|x\right|^{\frac{1}{1-r}}\gg\left|x\right|^{2}$,
with appropriate modifications for the case $r=1$ as in Example \ref{exa:q-log-sob}.
\end{itemize}
Now, to compute: recall by Theorem~\ref{thm:IP-to-mix} that in order
to ensure that $\chi^{2}(\mu P^{n},\pi)\leqslant\varepsilon_{{\rm Mix}}\leqslant8$,
it suffices to take

\begin{align*}
n & \geqslant2+2^{6}\cdot\varepsilon^{-2}\cdot\max\left\{ \log\left(\frac{u_{0}}{4\cdot v_{*}^{-1}}\right),0\right\} \\
 & +2^{8}\cdot\varepsilon^{-2}\cdot\delta^{-2}\cdot\int_{\max\left\{ \min\left\{ 2\cdot u_{0}^{-1},1/4\right\} ,v_{*}/2\right\} }^{1/4}\frac{\xi}{\tilde{I}_{\pi}\left(\xi\right)^{2}}\,{\rm d}\xi.\\
 & +2^{4}\cdot\max\left\{ 1,2^{-2}\cdot\delta^{-2}\cdot\tilde{I}_{\pi}\left(\frac{1}{4}\right)^{-2}\right\} \cdot\varepsilon^{-2}\cdot\log\left(\max\left\{ \frac{\min\left\{ u_{0},8\right\} }{\varepsilon_{{\rm Mix}}},1\right\} \right).
\end{align*}
Assuming for now that $u_{0}\geqslant8$ is intermediate, and $v_{*}\leqslant4\cdot u_{0}^{-1}$
is small, with $v_{*}$ given in (\ref{eq:v*_defn}), we focus on
the value of the middle integral,
\[
\int_{2\cdot u_{0}^{-1}}^{1/4}\frac{\xi}{\tilde{I}_{\pi}\left(\xi\right)^{2}}\,{\rm d}\xi=c^{-2}\cdot\int_{2\cdot u_{0}^{-1}}^{1/4}\frac{1}{\xi\cdot\log\left(\frac{1}{\xi}\right)^{2r}}\,{\rm d}\xi=c^{-2}\cdot\int_{\log\left(4\right)}^{\log\left(u_{0}/2\right)}\frac{{\rm d}u}{u^{2r}}.
\]
There is now a trichotomy of behaviours based upon the relative positions
of $r$ and $\frac{1}{2}$ (recalling that in practice, the parameter
$\varepsilon$ is typically of constant order):
\begin{itemize}
\item If $r\in\left[\frac{1}{2},1\right]$, then the inner integral evaluates
to $\frac{1}{2r-1}\cdot\left(\log\left(4\right)^{-\left(2r-1\right)}-\log\left(u_{0}/2\right)^{-\left(2r-1\right)}\right)\in\mathcal{O}\left(1\right)$.
The total mixing time then scales roughly like
\[
\max\left\{ \log\left(\frac{u_{0}}{4\cdot v_{*}^{-1}}\right),0\right\} +\delta^{-2}\cdot\left\{ {\bf 1}_{u_{0}>8}+\log\frac{1}{\varepsilon_{{\rm Mix}}}\right\} .
\]
\item If $r=\frac{1}{2}$, then the inner integral evaluates to $\log\left(\frac{\log\left(u_{0}/2\right)}{\log\left(4\right)}\right)\in\mathcal{O}\left(\log\log u_{0}\right)$.
The total mixing time then scales roughly like
\[
\max\left\{ \log\left(\frac{u_{0}}{4\cdot v_{*}^{-1}}\right),0\right\} +\delta^{-2}\cdot\left\{ \log\min\left\{ \log u_{0},v_{*}^{-1}\right\} \cdot{\bf 1}_{u_{0}>8}+\log\frac{1}{\varepsilon_{{\rm Mix}}}\right\} .
\]
\item If $r\in\left[0,\frac{1}{2}\right]$, then the inner integral evaluates
to $\frac{1}{1-2r}\cdot\left(\log\left(u_{0}/2\right)^{1-2r}-\log\left(4\right)^{1-2r}\right)\in\mathcal{O}\left(\log\left(u_{0}\right)^{1-2r}\right)$.
The total mixing time then scales roughly like
\[
\max\left\{ \log\left(\frac{u_{0}}{4\cdot v_{*}^{-1}}\right),0\right\} +\delta^{-2}\cdot\left\{ \left(\min\left\{ \log u_{0},v_{*}^{-1}\right\} \right)^{1-2r}\cdot{\bf 1}_{u_{0}>8}+\log\frac{1}{\varepsilon_{{\rm Mix}}}\right\} .
\]
In particular, when $r=0$, one obtains the same $\mathcal{O}\big(\log\left(u_{0}\right)\big)$
dependence implied by a standard spectral gap approach, i.e. the spectral
profile provides no strict benefit. Indeed for $\delta$ sufficiently
small one checks that $v_{*}=0$, and so the first term vanishes entirely.
Moreover, the second and third terms are qualitatively identical,
and thus there is no `three-phase' behaviour, but instead only \textit{one}
phase.
\end{itemize}
Recalling that $v_{*}^{-1}$ grows rapidly as $\delta$ decreases
we see that the effect of $\delta^{-2}$ is amplified for large $u_{0}$
or $v_{*}^{-1}$ for $r\leqslant1/2$, roughly corresponding to distributions
with tails heavier than Gaussians, and that this is not the case for
lighter tailed distributions i.e. $r>1/2$. 

\section{Spectral gap of RWM in high dimensions\label{sec:Spectral-gap-of-RWM}}

In Sections~\ref{sec:Spectral-gap-of-RWM}--\ref{sec:Convergence-and-mixing},
we denote by $P$ the RWM kernel defined by (\ref{eq:metropolis-kernel})
with $Q$ the Gaussian proposal kernel defined for a fixed but arbitrary
$\sigma>0$ by
\[
Q\left(x,A\right)=\int{\bf 1}_{A}\left(x+\sigma\cdot z\right)\,{\cal N}\left({\rm d}z;0,I_{d}\right),\qquad x\in\mathsf{E},A\in\mathscr{E}.
\]
In Appendix~\ref{sec:Different-proposal-distributions} we discuss
how our analysis can be generalized to other proposal kernels with
independent noise increments for each of the $d$ components. Since
$Q$ is reversible w.r.t. the Lebesgue measure on $\mathbb{R}^{d}$,
we have $\varpi=\pi\propto\exp\left(-U\right)$, following Assumption~\ref{assu:pi-density-lebesgue}.
It is standard to deduce by \citet[Lemma~3.1]{baxendale2005renewal}
that $P$ is a positive Markov kernel for this particular $Q$. We
note the following useful expression
\begin{equation}
\alpha\left(x\right)=\int{\cal N}\left({\rm d}z;0,I_{d}\right)\cdot\min\left\{ 1,\exp\left(-\left(U\left(x+\sigma\cdot z\right)-U\left(x\right)\right)\right)\right\} ,\label{eq:RWM-alpha-x}
\end{equation}
and we also denote $\alpha_{0}=\inf_{x\in\mathsf{E}}\alpha\left(x\right)$.

For the purposes of obtaining explicit bounds and matching negative
results with dimension, we impose the following further assumption
about $\pi$, noting that Assumption~\ref{assu:pi-density-lebesgue}
is already in force. As will be discussed in Section~\ref{subsec:Discussion-of-the-assumptions},
both $m$-strong convexity and $L$-smoothness can be weakened to
obtain explicit bounds on the spectral gap.
\begin{assumption}
\label{assu:target-distribution}For some $0<m\leqslant L$, $U$
is $m$-strongly convex and $L$-smooth:
\[
\frac{m}{2}\cdot\left|h\right|^{2}\leqslant U\left(x+h\right)-U\left(x\right)-\left\langle \nabla U\left(x\right),h\right\rangle \leqslant\frac{L}{2}\cdot\left|h\right|^{2},\qquad x,h\in\E.
\]
We write $\kappa:=L/m$ for the condition number of the target measure.
\end{assumption}

\begin{example}
\label{exa:gaussian-mL}Let $\pi=\mathcal{N}\left(0,\sigma_{0}^{2}\cdot I_{d}\right)$,
so that $U\left(x\right)=\frac{1}{2\cdot\sigma_{0}^{2}}\cdot\left|x\right|^{2}$.
Then
\[
U\left(x+h\right)-U\left(x\right)-\left\langle \nabla U\left(x\right),h\right\rangle =\frac{1}{2\cdot\sigma_{0}^{2}}\cdot\left|h\right|^{2},\qquad x,h\in\E,
\]
so $U$ is $m$-strongly convex and $L$-smooth with $L=m=1/\sigma_{0}^{2}$
and $\kappa=1$.
\end{example}

Another natural class of examples with strongly convex and smooth
potentials comes from considering Bayesian posterior measures for
which the prior is normal, and the log-likelihood is concave with
bounded Hessian.
\begin{example}
\label{exa:bayesian-logistic-regression}Consider the task of Bayesian
logistic regression, taking as prior $\pi_{0}=\mathcal{N}\left(0,\sigma_{0}^{2}\cdot I_{d}\right)$,
and observing covariate-response pairs $\left\{ \left(a_{i},y_{i}\right)\right\} _{i=1}^{N}\subset\mathbb{R}^{d}\times\left\{ 0,1\right\} $.
The potential corresponding to the posterior measure is then given
up to an additive constant by
\[
U\left(x\right)=\frac{1}{2\cdot\sigma_{0}^{2}}\cdot\left|x\right|^{2}+\sum_{i=1}^{N}\left\{ \log\left(1+\exp\left(-\left\langle a_{i},x\right\rangle \right)\right)-y_{i}\cdot\left\langle a_{i},x\right\rangle \right\} .
\]
Writing $A$ for the $N\times d$ matrix with rows given by the $\left\{ a_{i}\right\} $,
one can check that $U$ is $m$-strongly convex and $L$-smooth with
$m\geqslant\frac{1}{\sigma_{0}^{2}}$ and $L\leqslant\frac{1}{\sigma_{0}^{2}}+\frac{1}{4}\cdot\lambda_{{\rm Max}}\left(AA^{\top}\right)$,
$\lambda_{{\rm Max}}$ denoting the largest eigenvalue of a symmetric
matrix. This gives $\kappa\leqslant1+\frac{1}{4}\cdot\sigma_{0}^{2}\cdot\lambda_{{\rm Max}}\left(AA^{\top}\right)$.
\end{example}

\subsection{Lower bound on the conductance and spectral gap}

The following is a general result, depending \textit{only} on macroscopic
properties of the target measure: if the measure has good isoperimetry,
and the acceptance rate is lower bounded, then the chain will equilibrate
at a rate commensurate with these properties (roughly in accordance
with the convergence of the overdamped Langevin diffusion with the
same target measure), attenuated only by the choice of $\sigma$.
\begin{thm}
\label{thm:gaussian-rwm-conductance-general}Let $\pi$ admit a regular,
concave isoperimetric minorant $\tilde{I}_{\pi}$ w.r.t. $\left|\cdot\right|$.
Then
\[
\Phi_{P}^{*}\geqslant2^{-3}\cdot\alpha_{0}\cdot\min\left\{ 1,2\cdot\alpha_{0}\cdot\sigma\cdot\tilde{I}_{\pi}\left(\frac{1}{4}\right)\right\} ,
\]
and
\[
\gamma_{P}\geqslant2^{-7}\cdot\alpha_{0}^{2}\cdot\min\left\{ 1,4\cdot\alpha_{0}^{2}\cdot\sigma^{2}\cdot\tilde{I}_{\pi}\left(\frac{1}{4}\right)^{2}\right\} .
\]
\end{thm}

\begin{proof}
By Lemma~\ref{lem:close-coupling}, $P$ is $\left(\left|\cdot\right|,\alpha_{0}\cdot\sigma,\frac{1}{2}\cdot\alpha_{0}\right)$-close-coupling,
and we conclude from Theorem~\ref{thm:IP-to-mix}.
\end{proof}
\begin{cor}
\label{cor:rwm-conductance-lower-Lm}Under Assumption~\ref{assu:target-distribution},
let $\sigma=\varsigma\cdot L^{-1/2}\cdot d^{-1/2}$ for any $\varsigma>0$,
and $C_{\gamma}$ as in Lemma~\ref{lem:Bounds-on-isoperimetric}.
Then
\[
\Phi_{P}^{*}\geqslant2^{-4}\cdot C_{\gamma}\cdot\varsigma\cdot\exp\left(-\varsigma^{2}\right)\cdot d^{-1/2}\cdot\left(\frac{m}{L}\right)^{1/2},
\]
and 
\[
\gamma_{P}\geqslant2^{-9}\cdot C_{\gamma}^{2}\cdot\varsigma^{2}\cdot\exp\left(-2\cdot\varsigma^{2}\right)\cdot d^{-1}\cdot\frac{m}{L}.
\]
\end{cor}

\begin{proof}
By Corollary~\ref{cor:normal-accept-x}, $\alpha_{0}\geqslant\frac{1}{2}\cdot\exp\left(-\frac{1}{2}\cdot\varsigma^{2}\right)$
and we can take $\tilde{I}_{\pi}\left(\frac{1}{4}\right)\geqslant m^{1/2}\cdot C_{\gamma}$
by Lemma~\ref{lem:Bounds-on-isoperimetric}. Since 
\[
C_{\gamma}\cdot\exp\left(-\frac{1}{2}\cdot\varsigma^{2}\right)\cdot\varsigma\cdot L^{-1/2}\cdot d^{-1/2}\cdot m^{1/2}\leqslant C_{\gamma}\cdot\mathrm{e}^{-1/2}\cdot1<1,
\]
for any $d\in\mathbb{N}$ and $\varsigma>0$ (noting that $m\leqslant L$),
we deduce that
\begin{align*}
\Phi_{P}^{*} & \geqslant2^{-3}\cdot\alpha_{0}\cdot\min\left\{ 1,2\cdot\alpha_{0}\cdot\sigma\cdot\tilde{I}_{\pi}\left(\frac{1}{4}\right)\right\} \\
 & \geqslant2^{-4}\cdot\exp\left(-\frac{1}{2}\cdot\varsigma^{2}\right)\cdot\min\left\{ 1,\exp\left(-\frac{1}{2}\cdot\varsigma^{2}\right)\cdot\varsigma\cdot L^{-1/2}\cdot d^{-1/2}\cdot m^{1/2}\cdot C_{\gamma}\right\} \\
 & =2^{-4}\cdot C_{\gamma}\cdot\varsigma\cdot\exp\left(-\varsigma^{2}\right)\cdot d^{-1/2}\cdot\left(\frac{m}{L}\right)^{1/2},
\end{align*}
and the lower bound on $\gamma_{P}$ follows similarly.
\end{proof}
\begin{rem}
\label{rem:optimal-varsigma}We find that $\Phi_{P}^{*}\in\Omega\left(d^{-1/2}\right)$,
and $\gamma_{P}\in\Omega\left(d^{-1}\right)$. Fixing $\varsigma$,
we obtain (non-asymptotically) that
\begin{align*}
\Phi_{P}^{*}\cdot d^{1/2} & \geqslant2^{-4}\cdot C_{\gamma}\cdot\varsigma\cdot\exp\left(-\varsigma^{2}\right)\cdot\left(\frac{m}{L}\right)^{1/2}\\
 & \geqslant0.019861\cdot\varsigma\cdot\exp\left(-\varsigma^{2}\right)\cdot\left(\frac{m}{L}\right)^{1/2}.
\end{align*}
This lower bound is maximized in $\varsigma$ by taking $\varsigma^{2}=\frac{1}{2}$,
which yields
\[
\Phi_{P}^{*}\cdot d^{1/2}\geqslant0.008518\cdot\left(\frac{m}{L}\right)^{1/2}
\]
This particular bound-maximizing value of $\varsigma$ is likely an
artifact of the proof technique; optimal scaling results of \citet{roberts2001optimal}
suggest that $\varsigma\approx2.38$ is optimal in high dimensions
when $\pi=\mathcal{N}(0,I_{d})$ , although they do not provide a
bound on the conductance or spectral gap of the associated Markov
kernel. Similarly, taking $\varsigma^{2}=\frac{1}{2}$ leads to the
following bound for the spectral gap:
\[
\gamma_{P}\cdot d\geqslant3.62784\times10^{-5}\cdot\frac{m}{L}.
\]
\end{rem}

The following lemma gives a useful bound on the total variation distance
between the proposals; the proof can be found in Appendix~\ref{sec:Proofs-of-auxiliary}.
\begin{lem}
\label{lem:RWM-continuity-proposal}If $v>0$ and $x,y\in\mathsf{E}$
satisfy $\left|x-y\right|\leqslant v\cdot\sigma$,
\[
\left\Vert Q\left(x,\cdot\right)-Q\left(y,\cdot\right)\right\Vert _{{\rm TV}}\leqslant\frac{1}{2}\cdot v.
\]
\end{lem}

The following is a key lemma establishing the close coupling condition
for $P$.
\begin{lem}
\label{lem:close-coupling}$P$ is $\left(\left|\cdot\right|,\alpha_{0}\cdot\sigma,\frac{1}{2}\cdot\alpha_{0}\right)$-close-coupling.
\end{lem}

\begin{proof}
Assume $x,y\in\mathsf{E}$ are such that $\left|x-y\right|\leqslant\alpha_{0}\cdot\sigma$.
Then $\left\Vert Q\left(x,\cdot\right)-Q\left(y,\cdot\right)\right\Vert _{{\rm TV}}\leqslant\frac{1}{2}\cdot\alpha_{0}$
by Lemma~\ref{lem:RWM-continuity-proposal}. We conclude by applying
Lemma~\ref{lem:met-tv-bound}.
\end{proof}
To lower bound the acceptance rate for the RWM kernel, we first prove
a result which holds under a general smoothness condition on the potential
$U$, and then obtain the case of $L$-smoothness as a corollary.
\begin{lem}
\label{lem:general-accept-x}Suppose that for some nonnegative, nondecreasing
$\psi$, the potential $U$ satisfies the smoothness bound 
\[
U\left(x+h\right)-U\left(x\right)-\left\langle \nabla U\left(x\right),h\right\rangle \leqslant\psi\left(\left|h\right|\right),\qquad x,h\in\mathsf{E}.
\]
\begin{enumerate}
\item Then for any $\sigma\geqslant0$,
\[
\alpha_{0}\geqslant\frac{1}{2}\cdot\exp\left(-\int{\cal N}\left({\rm d}z;0,I_{d}\right)\cdot\psi\left(\sigma\cdot\left|z\right|\right)\right).
\]
\item Let $\sigma=\varsigma\cdot d^{-1/2}$ where $\varsigma>0$ is arbitrary,
$\psi$ be twice-differentiable, such that for some $c_{0},c_{1}>0$,
\[
\left(\varsigma\cdot\psi^{''}-\psi^{'}\right)\left(t\right)\leqslant c_{0}\cdot\exp\left(c_{1}\cdot t\right).
\]
Then
\[
\alpha_{0}\geqslant\frac{1}{2}\cdot\exp\left(-\psi\left(\varsigma\right)+\mathcal{O}\left(d^{-1}\right)\right)\in\Omega\left(1\right).
\]
\end{enumerate}
\end{lem}

\begin{proof}
The proof proceeds by direct calculation. Let $x\in\mathsf{E}$ be
arbitrary. Applying the growth bound assumption, it holds that
\[
U\left(x+\sigma\cdot z\right)-U\left(x\right)\leqslant\left\langle \nabla U\left(x\right),\sigma\cdot z\right\rangle +\psi\left(\sigma\cdot\left|z\right|\right),
\]
and substituting this into (\ref{eq:RWM-alpha-x}) gives
\begin{align*}
\alpha\left(x\right) & \geqslant\int{\cal N}\left({\rm d}z;0,I_{d}\right)\cdot\min\left\{ 1,\exp\left(-\left\langle \nabla U\left(x\right),\sigma\cdot z\right\rangle -\psi\left(\sigma\cdot\left|z\right|\right)\right)\right\} .
\end{align*}
Applying the inequality $\min\left\{ 1,a\cdot b\right\} \geqslant\min\left\{ 1,a\right\} \cdot\min\left\{ 1,b\right\} $
establishes that
\begin{align*}
\alpha\left(x\right) & \geqslant\int{\cal N}\left({\rm d}z;0,I_{d}\right)\cdot\min\left\{ 1,\exp\left(-\left\langle \nabla U\left(x\right),\sigma\cdot z\right\rangle \right)\right\} \cdot\min\left\{ 1,\exp\left(-\psi\left(\sigma\cdot\left|z\right|\right)\right)\right\} \\
 & \geqslant\int{\cal N}\left({\rm d}z;0,I_{d}\right)\cdot\exp\left(-\psi\left(\sigma\cdot\left|z\right|\right)\right)\cdot\min\left\{ 1,\exp\left(-\left\langle \nabla U\left(x\right),\sigma\cdot z\right\rangle \right)\right\} \\
 & =\int{\cal N}\left({\rm d}z;0,I_{d}\right)\cdot\exp\left(-\psi\left(\sigma\cdot\left|z\right|\right)\right)\cdot\min\left\{ 1,\exp\left(+\left\langle \nabla U\left(x\right),\sigma\cdot z\right\rangle \right)\right\} ,
\end{align*}
where the last equality follows from the change of variables $z\mapsto-z$.
Averaging these final two expressions and noting that $\min\left\{ 1,a\right\} +\min\left\{ 1,a^{-1}\right\} \geqslant1$,
it follows that
\[
\alpha\left(x\right)\geqslant\frac{1}{2}\cdot\int{\cal N}\left({\rm d}z;0,I_{d}\right)\cdot\exp\left(-\psi\left(\sigma\cdot\left|z\right|\right)\right).
\]
By the convexity of $\psi\mapsto\exp\left(-\psi\right)$, we may apply
Jensen's inequality to bound this integral from below, and so the
first part follows. 

Finally, take $\sigma=\varsigma\cdot d^{-1/2}$, and Taylor expand
$s\mapsto\psi\left(s^{1/2}\right)$ with a second-order remainder
term around $s=\varsigma^{2}$. Applying the hypothesis on the derivatives
of $\psi$ then yields that for sufficiently large $d$, there exists
$F\left(\varsigma\right)>0$ such that
\[
\int{\cal N}\left({\rm d}z;0,I_{d}\right)\cdot\psi\left(\varsigma\cdot d^{-1/2}\cdot\left|z\right|\right)\leqslant\psi\left(\varsigma\right)+F\left(\varsigma\right)\cdot d^{-1},
\]
and we conclude.
\end{proof}
\begin{cor}
\label{cor:normal-accept-x}Under Assumption~\ref{assu:target-distribution},
let $\sigma=\varsigma\cdot d^{-1/2}\cdot L^{-1/2}$ where $\varsigma>0$
is arbitrary. Then
\[
\alpha_{0}\geqslant\frac{1}{2}\cdot\exp\left(-\frac{1}{2}\cdot\varsigma^{2}\right).
\]
\end{cor}

\begin{proof}
Assumption~\ref{assu:target-distribution} implies that $U$ satisfies
the growth bound
\[
U\left(x+h\right)-U\left(x\right)-\left\langle \nabla U\left(x\right),h\right\rangle \leqslant\psi\left(\left|h\right|\right)
\]
with $\psi:r\mapsto\frac{1}{2}\cdot L\cdot r^{2}$. Applying the first
part of Lemma~\ref{lem:general-accept-x}, we see that 
\begin{align*}
\alpha_{0} & \geqslant\frac{1}{2}\cdot\exp\left(-\int{\cal N}\left({\rm d}z;0,I_{d}\right)\cdot\frac{1}{2}\cdot L\cdot\sigma^{2}\cdot\left|z\right|^{2}\right)\\
 & =\frac{1}{2}\cdot\exp\left(-\frac{1}{2}\cdot L\cdot\sigma^{2}\cdot d\right)\\
 & =\frac{1}{2}\cdot\exp\left(-\frac{1}{2}\cdot\varsigma^{2}\right).\qedhere
\end{align*}
\end{proof}

\subsection{Conductance and spectral gap upper bounds}

To complement our lower bounds, we can show matching upper bounds
with respect to dimension under Assumption~\ref{assu:target-distribution}.
This shows that the the conductance and spectral gap must decrease
at least as fast as $\mathcal{O}\left(d^{-1/2}\right)$ and $\mathcal{O}\left(d^{-1}\right)$
respectively, and that these are the slowest polynomial decays possible.
Hence, we may infer that in terms of optimizing conductance and spectral
gap, $d^{-1}$ is the correct polynomial scaling of $\sigma^{2}$.

We emphasize that the upper bounds are uniform over the class of $m$-strongly
convex and $L$-smooth potentials, indicating that the dimension-dependence
of this particular class of target distributions is well-characterized
by the analysis. This is in contrast to bounds which rely only on
specific examples exhibiting poor performance, as in minimax complexity
analysis: for example, \citet{wu2022minimax} show that the optimal
scaling of step-size with dimension in the Metropolis-adjusted Langevin
algorithm is not uniform in this class.

\subsubsection{Conductance upper bounds}
\begin{thm}
Under Assumption~\ref{assu:target-distribution} and twice continuous-differentiability
of $U$,

\[
\Phi_{P}^{*}\leqslant\min\left\{ 4\cdot L^{1/2}\cdot\sigma,\left(1+m\cdot\sigma^{2}\right)^{-d/2}\right\} .
\]
Hence, among polynomial scalings of $\sigma$, the scaling $\sigma\sim d^{-1/2}$
is optimal with $\Phi_{P}^{*}\sim d^{-1/2}$.
\end{thm}

\begin{proof}
The bounds follow from Propositions~\ref{prop:kappa-upper-small-sigma}
and~\ref{prop:kappa-upper-big-sigma}. Now let $\sigma=\varsigma\cdot L^{-1/2}\cdot d^{-\beta}$.
If $\beta\geqslant1/2$, then we obtain $\Phi_{P}^{*}\in\mathcal{O}\left(d^{-\beta}\right)$
and this is maximized by taking $\beta=\frac{1}{2}$. Combined with
Theorem~\ref{thm:gaussian-rwm-conductance-general} we may conclude
that $\Phi_{P}^{*}$ decays as $d^{-1/2}$ as $d\to\infty$ when $\beta=\frac{1}{2}$.
For $\beta<\frac{1}{2}$, we recall that by Proposition~\ref{prop:kappa-upper-big-sigma},
the conductance decays faster than any polynomial in $1/d$, and in
particular, faster than $d^{-1/2}$, from which the claim follows. 
\end{proof}
\begin{lem}
\label{lem:gauss-bd}Under Assumption~\ref{assu:target-distribution},
let $x_{*}$ be the (unique) minimizer of $U$. Then
\[
\left(\frac{m}{L}\right)^{d/2}\cdot\mathcal{N}\left(x;x_{*},L^{-1}\cdot I_{d}\right)\leqslant\pi\left(x\right)\leqslant\left(\frac{L}{m}\right)^{d/2}\cdot\mathcal{N}\left(x;x_{*},m^{-1}\cdot I_{d}\right).
\]
\end{lem}

\begin{proof}
Applying the definitions of $m$-strong convexity and $L$-smoothness,
one sees that
\[
\frac{m}{2}\cdot\left|x-x_{*}\right|^{2}\leqslant U\left(x\right)-U\left(x_{*}\right)-\left\langle \nabla U\left(x_{*}\right),x-x_{*}\right\rangle \leqslant\frac{L}{2}\cdot\left|x-x_{*}\right|^{2},
\]
and since $\nabla U\left(x_{*}\right)=0$, this can be simplified
to 
\[
\frac{m}{2}\cdot\left|x-x_{*}\right|^{2}\leqslant U\left(x\right)-U\left(x_{*}\right)\leqslant\frac{L}{2}\cdot\left|x-x_{*}\right|^{2}.
\]
Recalling that $\pi\left(x\right)=\frac{1}{Z}\cdot\exp\left(-U\left(x\right)\right)$
for some normalizing constant $Z$, this implies that
\[
\left(\frac{2\pi}{L}\right)^{d/2}\cdot\mathcal{N}\left(x;x_{*},L^{-1}\cdot I_{d}\right)\leqslant\pi\left(x\right)\cdot Z\cdot\exp\left(U\left(x_{*}\right)\right)\leqslant\left(\frac{2\pi}{m}\right)^{d/2}\cdot\mathcal{N}\left(x;x_{*},m^{-1}\cdot I_{d}\right).
\]
By integrating the above inequalities over space, one sees that 
\[
\left(\frac{2\pi}{L}\right)^{d/2}\leqslant Z\cdot\exp\left(U\left(x_{*}\right)\right)\leqslant\left(\frac{2\pi}{m}\right)^{d/2},
\]
and substituting these bounds into the preceding display completes
the proof.
\end{proof}
\begin{lem}
\label{lem:smooth_cvx_pres}Under Assumption~\ref{assu:target-distribution}
and twice continuous-differentiability of $U$, Assumption~\ref{assu:target-distribution}
holds for any finite-dimensional marginal of $\pi$.
\end{lem}

\begin{proof}
Preservation of $m$-strong log-concavity of the marginals of $\pi$
is shown by \citet[Theorem~3.8]{Saumard2014}. For preservation of
$L$-smoothness of the potential of any marginal of $\pi$, write
the state as $x=\left(x_{1},x_{2}\right)$, such that $\pi\left(x_{1},x_{2}\right)=\exp\left(-U\left(x_{1},x_{2}\right)\right)$,
and define the marginal $\pi\left(x_{1}\right)=\int\exp\left(-U\left(x_{1},x_{2}\right)\right)\dif x_{2}=\exp\left(-V\left(x_{1}\right)\right).$
Formal computations give that
\begin{align*}
\nabla V\left(x_{1}\right) & ={\bf \Ebb}\left[\nabla_{x_{1}}U\left(x_{1},X_{2}\right)\vert X_{1}=x_{1}\right]\\
\nabla^{2}V\left(x_{1}\right) & ={\bf \Ebb}\left[\nabla_{x_{1}}^{2}U\left(x_{1},X_{2}\right)\vert X_{1}=x_{1}\right]-{\bf Cov}\left[\nabla_{x_{1}}U\left(x_{1},X_{2}\right)\vert X_{1}=x_{1}\right].
\end{align*}
By the smoothness of $U$, for any fixed $x_{1}$ one can bound $\left|\nabla_{x_{1}}U\left(x_{1},x_{2}\right)\right|\lesssim1+\left|x_{2}\right|$,
$\left|\nabla_{x_{1}}^{2}U\left(x_{1},x_{2}\right)\right|\lesssim1$,
and since log-concave measures admit moments of all orders, it is
guaranteed that the integrals above do indeed exist and are finite.
Combining this with the twice continuous-differentiability of $U$,
we may validly interchange differentiation and integration, so that
the formal identities described above hold.

Recalling that covariance matrices are positive-semidefinite and that
$U$ is $L$-smooth, compute that
\begin{align*}
\nabla^{2}V\left(x_{1}\right) & ={\bf \Ebb}\left[\nabla_{x_{1}}^{2}U\left(x_{1},X_{2}\right)\vert X_{1}=x_{1}\right]-{\bf Cov}\left[\nabla_{x_{1}}U\left(x_{1},X_{2}\right)\vert X_{1}=x_{1}\right]\\
 & \preceq{\bf \Ebb}\left[\nabla_{x_{1}}^{2}U\left(x_{1},X_{2}\right)\vert X_{1}=x_{1}\right]\\
 & \preceq L\cdot I_{d},
\end{align*}
from which the $L$-smoothness of $V$ follows.
\end{proof}
The proofs the following two propositions involve the identification
of appropriate sets whose conductance can be bounded in terms of $\sigma$;
the resulting calculations are somewhat involved so the proofs are
housed in Appendix~\ref{sec:Proofs-of-auxiliary}.
\begin{prop}
\label{prop:kappa-upper-small-sigma}Under Assumption~\ref{assu:target-distribution}
and twice continuous-differentiability of $U$,
\[
\Phi_{P}^{*}\leqslant2\cdot L^{1/2}\cdot\sigma.
\]
\end{prop}

\begin{prop}
\label{prop:kappa-upper-big-sigma}Under Assumption~\ref{assu:target-distribution},
\[
\Phi_{P}^{*}\leqslant\left(1+m\cdot\sigma^{2}\right)^{-d/2}.
\]
Furthermore, if $\sigma=\varsigma\cdot L^{-1/2}\cdot d^{-\beta}$
with $\beta<1/2$, then $\Phi_{P}^{*}={\cal O}\left(\exp\left(-c\cdot d^{1-2\cdot\beta}\right)\right)$
for any $c\in\left(0,\frac{1}{2}\cdot\frac{m}{L}\cdot\varsigma^{2}\right)$,
and in particular, decays faster than any polynomial in $1/d$.
\end{prop}

\subsubsection{Spectral gap upper bounds}

A natural question is whether the lower bound for the spectral gap
is of the correct order when $\sigma\sim d^{-1/2}$, i.e. whether
indeed $\gamma_{P}={\rm Gap}_{R}\left(P\right)$ scales as $d^{-1}$.
Under Assumption~\ref{assu:target-distribution}, we verify this
directly and also show that this is the optimal polynomial scaling.
\begin{thm}
\label{thm:upper-bound-gap}Under Assumption~\ref{assu:target-distribution}
and twice continuous-differentiability of $U$,
\[
\gamma_{P}\leqslant\min\left\{ \frac{1}{2}\cdot L\cdot\sigma^{2},\left(1+m\cdot\sigma^{2}\right)^{-d/2}\right\} .
\]
Hence, among polynomial scalings of the $\sigma$, the scaling $\sigma\sim d^{-1/2}$
is optimal with $\gamma_{P}\sim d^{-1}$.
\end{thm}

\begin{proof}
The bounds follow from Lemma~\ref{lem:gap-upper-small-sigma} and
Proposition~\ref{prop:kappa-upper-big-sigma} combined with Lemma~\ref{lem:lawlersokal}.
Let $\sigma=\varsigma\cdot L^{-1/2}\cdot d^{-\beta}$. If $\beta\geqslant\frac{1}{2}$
then we obtain $\gamma_{P}\in\mathcal{O}\left(d^{-2\cdot\beta}\right)$
and this rate is maximized by taking $\beta=\frac{1}{2}$. Combined
with Corollary~\ref{cor:rwm-conductance-lower-Lm}, we may conclude
that $\gamma_{P}$ decays as $d^{-1}$ as $d\to\infty$ when $\beta=\frac{1}{2}$.
On the other hand, if $\beta<\frac{1}{2}$ then $\gamma_{P}\cdot d$
converges to $0$ and hence $\beta<\frac{1}{2}$ leads to a faster
decay of $\gamma_{P}$ than $\beta=\frac{1}{2}$.
\end{proof}
\begin{lem}
\label{lem:gap-upper-small-sigma}Let $U$ be $L$-smooth and twice
continuously differentiable. Then
\[
\gamma_{P}\leqslant\frac{1}{2}\cdot L\cdot\sigma^{2}.
\]
\end{lem}

\begin{proof}
By the Cramér--Rao inequality \citep[see, e.g.,][Eq.~10.25]{Saumard2014},
it holds for any $v\in\mathbb{R}^{d}$ that
\[
{\rm Var}_{\pi}\left(\left\langle v,X\right\rangle \right)\geqslant v^{\top}\cdot{\bf \Ebb}_{\pi}\left[\nabla^{2}\left(-\log\pi\left(X\right)\right)\right]^{-1}\cdot v\geqslant L^{-1}\cdot\left|v\right|^{2},
\]
by $L$-smoothness. For $v\neq0$, define $g_{v}\left(x\right):=\left\langle v,x-{\bf \Ebb}_{\pi}\left[X\right]\right\rangle $,
and compute
\begin{align*}
\mathcal{E}\left(P,g_{v}\right) & =\frac{1}{2}\int\pi\left({\rm d}x\right)P\left(x,{\rm d}y\right)\left(g_{v}\left(y\right)-g_{v}\left(x\right)\right)^{2}\\
 & =\frac{1}{2}\int\pi\left({\rm d}x\right)P\left(x,{\rm d}y\right)\left\langle v,y-x\right\rangle ^{2}\\
 & \leqslant\frac{1}{2}\int\pi\left({\rm d}x\right)Q\left(x,{\rm d}y\right)\left\langle v,y-x\right\rangle ^{2}\\
 & =\frac{1}{2}\int\pi\left({\rm d}x\right){\cal N}\left({\rm d}z;0,I_{d}\right)\left\langle v,\sigma\cdot z\right\rangle ^{2}\\
 & =\frac{1}{2}\cdot\sigma^{2}\cdot\left|v\right|^{2}.
\end{align*}
We obtain
\[
\gamma_{P}={\rm Gap}_{{\rm R}}\left(P\right)=\inf_{f\in\ELL_{0}\left(\pi\right)}\frac{{\cal E}\left(P,f\right)}{\left\Vert f\right\Vert _{2}^{2}}\leqslant\frac{{\cal E}\left(P,g_{v}\right)}{\left\Vert g_{v}\right\Vert _{2}^{2}}\leqslant\frac{\frac{1}{2}\cdot\sigma^{2}\cdot\left|v\right|^{2}}{L^{-1}\cdot\left|v\right|^{2}}=\frac{1}{2}\cdot L\cdot\sigma^{2}.\qedhere
\]
\end{proof}

\subsection{Implications for the asymptotic variance}

In this sub-section, we address the asymptotic variance of MCMC estimators
computed from RWM chains. We will show that when using appropriately-tuned
RWM chains to compute expectations of functions under $\pi$, the
asymptotic variance of these estimators is an inflation of the ideal
variance by a factor which scales linearly with the dimension of the
problem. Furthermore, we will exhibit that for a specific class of
functions (in particular, affine functions) that this bound is tight
in terms of its dimension-dependence.
\begin{prop}
\label{prop:result-avar-RWM}Let Assumption~\ref{assu:target-distribution}
hold, and $\sigma=\varsigma\cdot L^{-1/2}\cdot d^{-1/2}$ for any
$\varsigma>0$. Then, for any $f\in\ELL_{0}\left(\pi\right)$, the
asymptotic variance of $f$ can be bounded as
\begin{align*}
{\rm var}\left(P,f\right) & \leqslant2^{10}\cdot C_{\gamma}^{-2}\cdot\varsigma^{-2}\cdot\exp\left(2\cdot\varsigma^{2}\right)\cdot\kappa\cdot d\cdot\left\Vert f\right\Vert _{2}^{2}.
\end{align*}
Additionally, for any linear $f\in\ELL_{0}\left(\pi\right)$, ${\rm var}\left(P,f\right)\geqslant2\cdot\varsigma^{-2}\cdot d\cdot\left\Vert f\right\Vert _{2}^{2}$.
\end{prop}

\begin{proof}
Since $P$ is reversible and ${\rm Gap}_{{\rm R}}\left(P\right)>0$,
${\rm Id}-P$ is invertible on $\ELL_{0}\left(\pi\right)$. We have
\begin{align*}
{\rm var}\left(P,f\right) & =\left\langle f,\left(\Id+P\right)\cdot\left(\Id-P\right)^{-1}\cdot f\right\rangle ,
\end{align*}
Moreover, by considering the spectral resolution of $f$ with respect
to $P$, it is classical that 
\begin{align*}
\left\langle f,\left(\Id+P\right)\cdot\left(\Id-P\right)^{-1}\cdot f\right\rangle  & \leqslant\frac{2}{{\rm Gap}_{\mathrm{R}}\left(P\right)}\cdot\left\Vert f\right\Vert _{2}^{2},
\end{align*}
where ${\rm Gap}_{\mathrm{R}}\left(P\right)$ is the right spectral
gap of $P$. Recalling that ${\rm Gap}_{\mathrm{R}}\left(P\right)\geqslant\gamma_{P}$
(for any reversible $P$), we apply Corollary~\ref{cor:rwm-conductance-lower-Lm}
to bound
\begin{align*}
{\rm Gap}_{\mathrm{R}}\left(P\right) & \geqslant2^{-9}\cdot C_{\gamma}^{2}\cdot\varsigma^{2}\cdot\exp\left(-2\cdot\varsigma^{2}\right)\cdot d^{-1}\cdot\frac{m}{L},
\end{align*}
and deduce the first bound. For the second bound, by positivity of
$P$ we may write
\[
{\rm var}\left(P,f\right)=\left\langle f,\left(\Id+P\right)\cdot\left(\Id-P\right)^{-1}\cdot f\right\rangle \geqslant\left\langle f,\left(\Id-P\right)^{-1}\cdot f\right\rangle .
\]
Recall the variational representation \citep[see, e.g.,][]{caracciolo1990nonlocal},

\[
\left\langle f,\left(\Id-P\right)^{-1}\cdot f\right\rangle =\sup_{g\in\ELL_{0}\left(\pi\right)}\left\{ 2\cdot\left\langle f,g\right\rangle -\mathcal{E}\left(P,g\right)\right\} .
\]
We will consider taking $f:x\mapsto\left\langle v,x-\Ebb_{\pi}\left[X\right]\right\rangle $,
and $g=c\cdot f$ for $c\in\R$. By the argument in the proof of Lemma~\ref{lem:gap-upper-small-sigma},
it holds that
\[
\left\Vert f\right\Vert _{2}^{2}\geqslant L^{-1}\cdot\left|v\right|^{2},\qquad\mathcal{E}\left(P,f\right)\leqslant\frac{1}{2}\cdot\sigma^{2}\cdot\left|v\right|^{2}.
\]
We thus see that
\begin{align*}
{\rm var}\left(P,f\right) & \geqslant\left\langle f,\left(\Id-P\right)^{-1}\cdot f\right\rangle \\
 & =\sup_{g\in\ELL_{0}\left(\pi\right)}\left\{ 2\cdot\left\langle f,g\right\rangle -\mathcal{E}\left(P,g\right)\right\} \\
 & \geqslant\sup_{c\in\R}\left\{ 2\cdot c\cdot\left\Vert f\right\Vert _{2}^{2}-c^{2}\cdot\mathcal{E}\left(P,f\right)\right\} \\
 & =\frac{\left\Vert f\right\Vert _{2}^{2}}{\mathcal{E}\left(P,f\right)}\cdot\left\Vert f\right\Vert _{2}^{2}\\
 & \geqslant\frac{L^{-1}\cdot\left|v\right|^{2}}{\frac{1}{2}\cdot\sigma^{2}\cdot\left|v\right|^{2}}\cdot\left\Vert f\right\Vert _{2}^{2}\\
 & =2\cdot\varsigma^{-2}\cdot d\cdot\left\Vert f\right\Vert _{2}^{2}.\qedhere
\end{align*}
\end{proof}

\subsection{\label{subsec:Discussion-of-the-assumptions}Discussion of the assumptions}

Under Assumption~\ref{assu:target-distribution}, Corollary~\ref{cor:rwm-conductance-lower-Lm}
provides explicit bounds on the spectral gap of the RWM kernel for
any $\sigma>0$. Theorem~\ref{thm:upper-bound-gap} further shows
that the optimal dependence on dimension is $d^{-1}$ for this class
of target measures. On the other hand, it is clear from the proof
that the lower bounds on the spectral gap and conductance profile
have $m$-strong convexity and $L$-smoothness of $U$ as sufficient
but not necessary conditions.

For showing lower bounds, $m$-strong convexity of the potential is
only used to verify an appropriate isoperimetric profile inequality.
Hence, any other method for establishing such inequalities could be
used instead. For instance, one may follow Example~\ref{exa:convex-poincare-to-profile}
to deduce an isoperimetric profile inequality for any log-concave
probability measure; an explicit constant could be obtained by study
of the corresponding overdamped Langevin diffusion. Example~\ref{exa:holley-stroock}
could also be used to extend this reasoning to perturbations of log-concave
probability measures. Similarly, for showing lower bounds, $L$-smoothness
is only assumed to establish that the acceptance probability is uniformly
lower bounded above $0$, i.e. $\alpha_{0}=\inf_{x\in\mathsf{E}}\alpha\left(x\right)>0$;
our Lemma~\ref{lem:general-accept-x} shows that this is possible
under far less restrictive smoothness conditions on $U$, covering
both light-tailed targets and potentials whose gradients are only,
e.g., Hölder-continuous. As such, $L$-smoothness should not be viewed
as strictly necessary for the RWM to be performant; this stands in
contrast to various gradient-based MCMC algorithms, whose performance
can deteriorate in the absence of $L$-smoothness \citep[see, e.g.,][]{Livingstone2022}.

On the other hand, if $U$ has particularly poor regularity (e.g.
non-Lipschitz-continuous gradients), then it can be necessary to scale
$\sigma$ differently in order to stabilize the acceptance probability
away from $0$, (see also \citet{vogrinc2021counterexamples}, who
study this phenomenon in the optimal scaling framework). For a concrete
example, consider the target with potential given by $U\left(x\right)=\left|x\right|_{p}^{p}$
for $p\in[1,2)$. Observing that $U\left(\sigma\cdot z\right)-U\left(0\right)=\sigma^{p}\cdot\left(\sum_{i=1}^{d}\left|z_{i}\right|^{p}\right)$,
one sees that unless $\sigma^{p}\in\mathcal{O}\left(d^{-1}\right)$,
then $\alpha\left(0\right)$ will tend to $0$ as $d\to\infty$. Similar
cautionary examples which necessitate anomalous scalings of $\sigma$
can be constructed by designing potentials with `sharp' growth around
a local minimum, and following the strategy of Proposition~\ref{prop:kappa-upper-big-sigma}.

\section{\label{sec:Convergence-and-mixing}Convergence and mixing time for
RWM}

In this section, we analyze the mixing time for the RWM kernel $P$
defined in Section~\ref{sec:Spectral-gap-of-RWM} under Assumption~\ref{assu:target-distribution}.

\subsection{Three phase mixing}

The mixing time bound essentially consists of up to three phases.
If the initial chi-squared divergence is very large, then there is
an initial phase during which the convergence takes place at an exponential
rate depending only on $\alpha_{0}$ (and in particular, \textit{not}
directly on $\sigma^{2}$). Subsequently, there is a secondary phase
during which convergence takes place at a faster-than-exponential
rate, depending now on both $\sigma^{2}$ and $\alpha_{0}$. Finally,
once the chi-squared divergence drops below a universal constant (e.g.
$8$), the convergence is again exponential, again with a rate depending
on both $\sigma^{2}$ and $\alpha_{0}$. Qualitatively similar behaviour
was also observed (for a different algorithm) in the work of \citet{mou2019efficient}.
In the proof of Theorem~\ref{thm:RWM_mixing} in Appendix~\ref{sec:Proofs-of-auxiliary}
we provide two additional different bounds, valid for any $\varsigma>0$,
which are stronger than the stated bound. In particular, the stated
bound may be conservative when $u_{0}$ is sufficiently small , since
the first phase can be non-existent and the second phase is bounded
fairly crudely.
\begin{thm}
\label{thm:RWM_mixing}Under Assumption~\ref{assu:target-distribution},
let $\mu\ll\pi$ be a probability measure and $u_{0}=\chi^{2}\left(\mu,\pi\right)$.
Let $\sigma=\varsigma\cdot L^{-1/2}\cdot d^{-1/2}$ and $\kappa=L/m$,
with $\varsigma>0$ arbitrary and let the universal constants $C_{\ell},C_{\gamma}$
be as defined in Lemma~\ref{lem:Bounds-on-isoperimetric}. Then,
to guarantee that $\chi^{2}\left(\mu P^{n},\pi\right)\leqslant\varepsilon_{{\rm Mix}}\in\left(0,8\right)$
we may take
\begin{align*}
n & \geqslant2+2^{10}\cdot\exp\left(\varsigma^{2}\right)\cdot\log\left(\max\left\{ u_{0},1\right\} \right)\\
 & +2^{14}\cdot C_{\ell}^{-2}\cdot\exp\left(2\cdot\varsigma^{2}\right)\cdot\varsigma^{-2}\cdot\kappa\cdot d\cdot\left\{ \log\left(16\cdot C_{\ell}^{-2}\varsigma^{-2}\right)+\log\left(\kappa\cdot d\right)+\varsigma^{2}\right\} \\
 & +2^{6}\cdot C_{\gamma}^{-2}\cdot\exp\left(2\cdot\varsigma^{2}\right)\cdot\varsigma^{-2}\cdot\kappa\cdot d\cdot\log\left(\frac{8}{\varepsilon_{{\rm Mix}}}\right),
\end{align*}
i.e., of order $\mathcal{O}\left(\log u_{0}+\kappa\cdot d\cdot\log\left(\kappa\cdot d\right)+\kappa\cdot d\cdot\log\left(\varepsilon_{{\rm Mix}}^{-1}\right)\right)$.
\end{thm}

\subsection{Two feasible ``warm starts''}

We provide two complexity bounds based on Theorem~\ref{thm:RWM_mixing},
corresponding to two different algorithmically feasible initial distributions
$\mu$. In Remark~\ref{rem:normal-L-initialization-RWM}, we consider
a Gaussian initial distribution centered at the mode of the density
$\pi$ with covariance $L^{-1}\cdot I_{d}$; when $U$ is convex,
identifying the mode numerically is feasible and hence the strategy
relies on explicit knowledge of $L$. In Remark~\ref{rem:bc-acc-warm}
we consider the approach suggested by \citet{belloni2009computational},
for which the initial distribution is given by the distribution of
the first `accepted' proposal from $Q\left(x_{0},\cdot\right)$ where
$x_{0}$ is arbitrary. This does not require explicit knowledge of
$L$, only e.g. a bound that allows one to tune $\sigma$ such that
$\varsigma\in\Theta\left(1\right)$, and corresponds directly to how
RWM chains are often initialized in practice. In particular, we observe
that the RWM mixing time is fairly robust to the choice of initial
point $x_{0}$ as long as it is not very far from the mode, and this
is due to the fast phases of convergence identified in Theorem~\ref{thm:RWM_mixing}.

\begin{rem}
\label{rem:normal-L-initialization-RWM}If $\mu=\mathcal{N}\left(x_{*},L^{-1}\cdot I_{d}\right)$,
with $x_{*}$ the mode of $\pi$, then by Lemma~\ref{lem:gauss-bd}
we may obtain the bound $u_{0}=\chi^{2}\left(\mu,\pi\right)\leqslant\kappa^{d/2}$.
It then holds that $\log u_{0}=\mathcal{O}\left(d\cdot\log\kappa\right)$,
and one obtains a mixing time bound which scales as $\mathcal{O}\left(\kappa\cdot d\cdot\log\left(\kappa\cdot d\right)+\kappa\cdot d\cdot\log\left(\varepsilon_{{\rm Mix}}^{-1}\right)\right)$.
In relation to the comments preceding Theorem~\ref{thm:RWM_mixing},
we note that by taking $\varsigma$ sufficiently small, one may consider
(\ref{eq:rwm-mix-bound-2})--(\ref{eq:rwm-mix-bound-2-vo}) in the
proof of Theorem~\ref{thm:RWM_mixing}, ensure $u_{0}\leqslant4\cdot v_{\circ}^{-1}$
and thereby reduce the complexity to $\mathcal{O}\left(\kappa\cdot d\cdot\log\left(d\cdot\log\kappa\right)+\kappa\cdot d\cdot\log\left(\varepsilon_{{\rm Mix}}^{-1}\right)\right)$.
For example, taking $\varsigma^{2}=\frac{1}{2}$ as in Remark~\ref{rem:optimal-varsigma},
one can deduce that
\[
n\geqslant2+96982\cdot\kappa\cdot d\cdot\log\left(\frac{d}{2}\cdot\log\kappa\right)+3446\cdot\kappa\cdot d\cdot\log\left(\frac{8}{\varepsilon_{{\rm Mix}}}\right),
\]
is sufficient for $\chi^{2}\left(\mu P^{n},\pi\right)\leqslant\varepsilon_{{\rm Mix}}\in(0,8)$
when $\kappa^{d/2}>8$.
\end{rem}

In Example~\ref{exa:bayesian-logistic-regression}, we have $\kappa\leqslant1+\frac{1}{4}\cdot\sigma_{0}^{2}\cdot\lambda_{{\rm Max}}\left(AA^{\top}\right)$
and if $A$ is a random matrix with i.i.d. entries from a distribution
with mean $0$, variance $1$ and finite $4$th moment, then \citet[Theorem~3.1]{yin1988limit}
gives $\frac{1}{N}\lambda_{{\rm Max}}\left(AA^{\top}\right)\to\frac{1}{\alpha}(1+\sqrt{\alpha})^{2}$
a.s., where $d/N\to\alpha$. Hence, if $\sigma_{0}^{2}$ is $\mathcal{O}(d^{-1})$
then it is reasonable to expect $\kappa$ independent of $d$ in this
regime and hence the $\mathcal{O}(d\log d)$ scaling of $n$ given
above. The precise bounds are likely quite loose, e.g. $d=100$, $\kappa=10$,
$\varepsilon_{{\rm Mix}}=\frac{1}{2}$ give $n\sim0.47\times10^{9}$
, but yet are not astronomical. The analysis also does not take into
account explicitly any concentration of the posterior distribution,
and indeed the bounds hold irrespective of the values the data take.

\begin{rem}
\label{rem:bc-acc-warm}There is another approach to obtaining a warm
start which neatly sidesteps the need for any preliminary optimization,
suggested by \citet{belloni2009computational}. The idea is to initialize
the chain by a single accepted move of the Metropolis kernel from
some arbitrary point $x_{0}$, i.e.
\begin{align*}
\mu\left(\mathrm{d}x\right) & =P^{\alpha}\left(x_{0},\mathrm{d}x\right)\\
 & :=Q\left(x_{0},\mathrm{d}x\right)\cdot\alpha\left(x_{0},x\right)\cdot\alpha\left(x_{0}\right)^{-1}.
\end{align*}
To this end, one can compute directly (in fact, for \textit{any} Metropolis--Hastings
chain) that
\begin{align*}
\pi\left(\mathrm{d}x\right)\cdot P^{\alpha}\left(x,\mathrm{d}y\right) & =\pi\left(\mathrm{d}x\right)\cdot Q\left(x,\mathrm{d}y\right)\cdot\alpha\left(x,y\right)\cdot\alpha\left(x\right)^{-1}\\
 & =\pi\left(\mathrm{d}y\right)\cdot Q\left(y,\mathrm{d}x\right)\cdot\alpha\left(y,x\right)\cdot\alpha\left(x\right)^{-1}.
\end{align*}
Disintegrating this joint measure appropriately, one sees for the
RWM that
\begin{align*}
\frac{\mathrm{d}P^{\alpha}\left(x_{0},\cdot\right)}{\mathrm{d}\pi}\left(x\right) & =\alpha\left(x,x_{0}\right)\cdot\alpha\left(x_{0}\right)^{-1}\cdot\frac{\mathrm{d}Q\left(x,\cdot\right)}{\mathrm{d}\pi}\left(x_{0}\right)\\
 & \leqslant1\cdot\alpha_{0}^{-1}\cdot\left(\sup_{y\in\E}\frac{\mathrm{d}Q\left(x,\cdot\right)}{\mathrm{d}{\rm Leb}}\left(y\right)\right)\cdot\left(\frac{\mathrm{d}\pi}{\mathrm{d}{\rm Leb}}\left(x_{0}\right)\right)^{-1}.
\end{align*}
Computing that $\alpha_{0}^{-1}\leqslant2\cdot\exp\left(\frac{1}{2}\cdot\varsigma^{2}\right)$,
$\sup_{y\in\E}\frac{\mathrm{d}Q\left(x,\cdot\right)}{\mathrm{d}{\rm Leb}}\left(y\right)=\left(2\cdot\pi\cdot\sigma^{2}\right)^{-d/2}$,
$\left(\frac{\mathrm{d}\pi}{\mathrm{d}{\rm Leb}}\left(x_{0}\right)\right)^{-1}\leqslant\kappa^{d/2}\cdot\left(2\cdot\pi\cdot L^{-1}\right)^{d/2}\cdot\exp\left(\frac{1}{2}\cdot L\cdot\left|x_{0}-x_{*}\right|^{2}\right)$,
one obtains that
\begin{align*}
\sup_{x\in\E}\frac{\mathrm{d}P^{\alpha}\left(x_{0},\cdot\right)}{\mathrm{d}\pi}\left(x\right) & \leqslant2\cdot\exp\left(\frac{1}{2}\cdot\varsigma^{2}\right)\cdot\left(\frac{\kappa\cdot d}{\varsigma^{2}}\right)^{d/2}\cdot\exp\left(\frac{1}{2}\cdot L\cdot\left|x_{0}-x_{*}\right|^{2}\right)\\
 & \in\exp\left(\mathcal{O}\left(d\cdot\log\left(d\cdot\kappa\right)+L\cdot\left|x_{0}-x_{*}\right|^{2}\right)\right),
\end{align*}
when $\varsigma$ is of order $1$. As such, provided that $L\cdot\left|x_{0}-x_{*}\right|^{2}\in\mathcal{O}\left(\kappa\cdot d\cdot\log\left(\kappa\cdot d\right)\right)$,
it holds that $\log u_{0}\in\mathcal{O}\left(\kappa\cdot d\cdot\log\left(\kappa\cdot d\right)\right)$,
and one obtains a mixing time bound which again scales as $\mathcal{O}\left(\kappa\cdot d\cdot\log\left(\kappa\cdot d\right)+\kappa\cdot d\cdot\log\left(\varepsilon_{{\rm Mix}}^{-1}\right)\right)$.
From the perspective of complexity analysis, it is thus essentially
sufficient to initialize the chain by identifying some point within
a reasonable distance, i.e. $\mathcal{O}\left(\left(\frac{d}{m}\cdot\log\left(\kappa\cdot d\right)\right)^{1/2}\right)$
of the mode, and waiting until a single proposed move is accepted.
\end{rem}

\subsection{Comparisons to existing results}

The prior work of \citet{belloni2009computational} is closely related;
the authors study the application of MCMC techniques to frequentist
estimation tasks for which direct optimization of the objective function
is challenging. For their application, a different isoperimetric inequality
is used, which allows them to handle potentials which are non-convex,
non-smooth, or even both. However they assume that their target distributions
are supported on a (large) closed ball. Their complexity bounds are
essentially obtained by estimating the spectral gap, and constructing
an appropriate `warm start', i.e. an initial distribution $\mu$ which
satisfies $\frac{{\rm d}\mu}{{\rm d}\pi}\leqslant W<\infty$, which
implies a bound on the initial $\ELL$ distance to equilibrium, recalled
in Remark~\ref{rem:bc-acc-warm}. Their result holds for a particular
choice of $\sigma^{2}$, which depends on theoretical quantities that
are typically unknown. In contrast, our bounds are valid for any $\sigma^{2}$.
Their conductance bound is of the form $\Phi_{P}^{*}\in\Omega\left(d^{-1}\right)$,
as opposed to the $\Phi_{P}^{*}\in\Omega\left(d^{-1/2}\right)$ that
we find. This may appear suboptimal, but it may also be the case that
their class of target distributions is sufficiently harder than those
which we consider, to the point that the conductance is genuinely
worse in this way.

Closest to our work are the twin papers \citet{Dwivedi-Chen-Wainwright-JMLR:v20:19-306}
and \citet{chenJMLR:v21:19-441}, which give complexity bounds for
the RWM under Assumption~\ref{assu:target-distribution}. Under a
`feasible' Gaussian initial law (roughly as in Remark~\ref{rem:normal-L-initialization-RWM}),
\citet{Dwivedi-Chen-Wainwright-JMLR:v20:19-306} obtain complexity
bounds of $\mathcal{O}\left(d^{2}\cdot\kappa^{2}\cdot\log^{1.5}\left(\kappa\cdot\varepsilon_{{\rm Mix}}^{-1}\right)\right)$
which \citet{chenJMLR:v21:19-441} refined to $\mathcal{O}\left(d\cdot\kappa^{2}\cdot\log\left(d\cdot\varepsilon_{{\rm Mix}}^{-1}\right)\right)$
by making use of the conductance profile framework; we observe that
our complexity analysis implies a weaker dependence on $\kappa$.
Their results also assume precise values of $\sigma^{2}$ that are
typically unknown in practical applications, in contrast with ours.

\section{\label{sec:Cvg-mix-pcn}Convergence and mixing time for pCN}

In this section, we apply the same technique to analyze the convergence
to equilibrium of the preconditioned Crank--Nicolson algorithm. We
consider the following class of distributions, consistent with both
Assumptions~\ref{assu:pi-density-lebesgue} and~\ref{assu:target-distribution}:
\begin{assumption}
\label{assu:pCN_target}The target measure $\pi$ on $\E=\R^{d}$
can be written
\[
\pi\left({\rm d}x\right)\propto\mathcal{N}\left({\rm d}x;0,\mathsf{C}\right)\cdot\exp\left(-\Psi\left(x\right)\right),
\]
where $\Psi$ is convex, $L$-smooth, and minimized at $x=0$, and
$\mathsf{C}$ is a positive definite covariance matrix.
\end{assumption}

Letting $\nu=\mathcal{N}\left(0,\mathsf{C}\right)$, throughout this
section we denote by $P$ the pCN kernel defined by (\ref{eq:metropolis-kernel}),
with $Q$ the $\nu$-reversible Gaussian kernel defined for a fixed
but arbitrary $\rho\in(0,1)$ by
\[
Q\left(x,A\right)=\int\mathbf{1}_{A}\left(\rho\cdot x+\eta\cdot z\right)\,\mathcal{N}\left(\dif z;0,\mathsf{C}\right),\qquad x\in\E,A\in\mathscr{E},
\]
where $\rho^{2}+\eta^{2}=1$. We will refer to $\eta$ as the step-size
of this kernel, and denote $\alpha_{0}:=\inf_{x\in\E}\alpha\left(x\right)$.
In particular, we have that $\varpi\propto\exp\left(-\Psi\right)$
is the density of $\pi$ w.r.t. the reference measure $\nu$. By \citet[Proposition 3(i)]{doucet2015efficient},
we may deduce that $P$ is positive, taking in their notation $\chi=\nu$
and $r\left(u,v\right)=\mathcal{N}\left(u;\rho^{1/2}\cdot v,\left(1-\rho\right)\cdot\mathsf{C}\right)$.

The pCN algorithm is particularly popular in Bayesian Inverse Problems
\citep[see, e.g.,][Example~5.3]{Stuart2010}, where $\mathsf{C}$
is typically a finite section of some infinite-dimensional trace-class
covariance operator. As with RWM, one can handle relaxations of $L$-smoothness
of $\Psi$ by a suitable adaptation of Lemma~\ref{lem:general-accept-x}.

\begin{rem}
\label{rem:mL-pcn-conversion}Assume $\pi$ has an $m$-strongly log-concave
density w.r.t. Lebesgue, with associated potential $U$ minimized
at $0$. Then in Assumption~\ref{assu:pCN_target} one may take $\mathsf{C}=m^{-1}\cdot I_{d}$
and $\Psi\left(x\right)=U\left(x\right)-\frac{1}{2}\cdot m\cdot\left|x\right|^{2}$.
We observe that if $U$ is $L'$-smooth, then $\Psi$ is $(L'-m)$-smooth,
and in particular we obtain the pCN condition number $\kappa=L/m=\kappa'-1$,
where $\kappa'=L'/m$ is the condition number one would associate
with the RWM kernel. Our results for pCN with more general $\mathsf{C}$
can therefore be interpreted as applying to the class of densities
with respect to Lebesgue which possess $m$-strongly convex and $L$-smooth
potentials, where we see an improvement over RWM in terms of the condition
number, at least if one is willing to use the minimizer of $U$ to
define an appropriate parameterization.
\end{rem}

Throughout this section, we write ${\rm Tr}\left(\mathsf{C}\right)$
for the trace of a matrix $\mathsf{C}$, and define the norm
\begin{equation}
\left|x\right|_{\mathsf{C}^{-1}}:=\left|\mathsf{C}^{-1/2}\cdot x\right|,\qquad x\in\E.\label{eq:norm_C}
\end{equation}
As with RWM, one can derive a spectral gap estimate for pCN by applying
Theorem~\ref{thm:IP-to-mix}.
\begin{thm}
\label{thm:pcn-gap}Under Assumption~\ref{assu:pCN_target}, $\pi$
admits a regular, concave isoperimetric minorant $\tilde{I}_{\pi}=\varphi_{\gamma}\circ\Phi_{\gamma}^{-1}$
w.r.t. $\left|\cdot\right|_{\mathsf{C}^{-1}}$ and
\[
\Phi_{P}^{*}\geqslant\frac{1}{4}\cdot C_{\gamma}\cdot\alpha_{0}^{2}\cdot\eta,\qquad\gamma_{P}\geqslant2^{-5}\cdot C_{\gamma}^{2}\cdot\alpha_{0}^{4}\cdot\eta^{2},
\]
where the constant $C_{\gamma}$ is defined in Lemma~\ref{lem:Bounds-on-isoperimetric}.
Writing $\eta=\varsigma\cdot\left(L\cdot{\rm Tr}\left(\mathsf{C}\right)\right)^{-1/2}\in\left(0,1\right)$,
we have $\alpha_{0}\geqslant\frac{1}{2}\cdot\exp\left(-\frac{1}{2}\cdot\varsigma^{2}\right)$,
and hence
\begin{equation}
\gamma_{P}\geqslant2^{-9}\cdot C_{\gamma}^{2}\cdot\exp\left(-2\cdot\varsigma^{2}\right)\cdot\varsigma^{2}\cdot\left(L\cdot{\rm Tr}\left(\mathsf{C}\right)\right)^{-1}.\label{eq:pcn-gap-bound-varsigma}
\end{equation}
Optimizing over $\varsigma$ gives
\[
\gamma_{P}\geqslant2^{-10}\cdot C_{\gamma}^{2}\cdot{\rm e}^{-1}\cdot\left(L\cdot{\rm Tr}\left(\mathsf{C}\right)\right)^{-1}\geqslant3.62784\times10^{-5}\cdot\left(L\cdot{\rm Tr}\left(\mathsf{C}\right)\right)^{-1}.
\]
\end{thm}

\begin{proof}
This is proven in an analogous fashion to Theorem~\ref{thm:gaussian-rwm-conductance-general}.
The fact that $\tilde{I}_{\pi}=\varphi_{\gamma}\circ\Phi_{\gamma}^{-1}$
is an isoperimetric minorant for $\pi$ is established in Lemma~\ref{lem:pcn-threeset},
and the close-coupling condition for pCN is established with $\varepsilon=\frac{1}{2}\cdot\alpha_{0}$
and $\delta=\alpha_{0}\cdot\frac{\eta}{\rho}$ in Lemma~\ref{lem:close-coupling-pcn}.
Thus we can apply Theorem~\ref{thm:IP-to-mix} to deduce the conductance
bound $\Phi_{P}^{*}\geqslant\frac{1}{8}\cdot\alpha_{0}\cdot\min\left\{ 1,2\cdot C_{\gamma}\cdot\alpha_{0}\cdot\frac{\eta}{\rho}\right\} $,
which we then simplify by recalling that $\rho\leqslant1$ and $2\cdot C_{\gamma}\cdot\alpha_{0}\cdot\eta\leqslant2\cdot C_{\gamma}\cdot1\cdot1<1$.
Finally, the lower bound on the acceptance rate is established in
Lemma~\ref{lem:pCN-accept-x}.
\end{proof}

\subsection{Lower bounds for pCN}

We now give appropriate results related to isoperimetry for the pCN
algorithm. The key subtlety to establishing an appropriate isoperimetric
inequality for the pCN kernel is the fact that we want to change metric
from the flat Euclidean metric $|\cdot|$ to the metric $|\cdot|_{\mathsf{C}^{-1}}$.
\begin{lem}
\label{lem:pcn-threeset}Under Assumption~\ref{assu:pCN_target},
$\pi$ admits $\varphi_{\gamma}\circ\Phi_{\gamma}^{-1}$ as an isoperimetric
minorant w.r.t. the metric $\mathsf{d}_{\mathsf{C}}\left(x,y\right)=\left|x-y\right|_{\mathsf{C}^{-1}}$.
\end{lem}

\begin{proof}
First, define $\pi_{W}=\left(x\mapsto\mathsf{C}^{-1/2}\cdot x\right)_{\#}\pi$.
One checks that the density of $\pi_{W}$ is given by $\mathcal{N}\left(x;0,I_{d}\right)\cdot\exp\left(-\Psi\left(\mathsf{C}^{1/2}\cdot x\right)\right)$,
for which the potential $\frac{1}{2}\cdot\left|x\right|^{2}+\Psi\left(\mathsf{C}^{1/2}\cdot x\right)$
is 1-strongly convex. By Lemma~\ref{lem:Bounds-on-isoperimetric},
one sees that $\pi_{W}$ admits $\varphi_{\gamma}\circ\Phi_{\gamma}^{-1}$
as an isoperimetric minorant w.r.t. $\left|\cdot\right|$. Defining
$\mathsf{d}_{\mathsf{C}}\left(x,y\right):=\left|x-y\right|_{\mathsf{C}^{-1}}$,
one can apply Lemma~\ref{lem:transport-isoperimetry} with $\left(\mu_{1},\E_{1},\mathsf{d}_{1}\right)=\left(\pi_{W},\E,\left|\cdot\right|\right)$,
$\left(\mu_{2},\E_{2},\mathsf{d}_{2}\right)=\left(\pi,\E,\mathsf{d}_{\mathsf{C}}\right)$,
noting that $x\mapsto\mathsf{C}^{-1/2}\cdot x$ is an \textit{isometry}
between these two metric spaces, to conclude that $\pi$ admits $\varphi_{\gamma}\circ\Phi_{\gamma}^{-1}$
as an isoperimetric minorant with respect to $\mathsf{d}_{\mathsf{C}}$.
\end{proof}
The following lemma gives a useful bound on the total variation distance
between the proposals, analogous to Lemma~\ref{lem:RWM-continuity-proposal};
the proof can be found in Appendix~\ref{sec:Proofs-of-auxiliary}.
\begin{lem}
\label{lem:PCN-continuity-proposal}If $v>0$ and $x,y\in\mathsf{E}$
satisfy $\left|x-y\right|_{\mathsf{C}^{-1}}\leqslant v\cdot\frac{\eta}{\rho}$,
\[
\left\Vert Q\left(x,\cdot\right)-Q\left(y,\cdot\right)\right\Vert _{{\rm TV}}\leqslant\frac{1}{2}\cdot v.
\]
\end{lem}

\begin{lem}
\label{lem:close-coupling-pcn}$P$ is $\left(\left|\cdot\right|_{\mathsf{C}^{-1}},\alpha_{0}\cdot\frac{\eta}{\rho},\frac{1}{2}\cdot\alpha_{0}\right)$-close-coupling.
\end{lem}

\begin{proof}
Assume $x,y\in\mathsf{E}$ are such that $\left|x-y\right|_{\mathsf{C}^{-1}}\leqslant\alpha_{0}\cdot\frac{\eta}{\rho}$.
Then $\left\Vert Q\left(x,\cdot\right)-Q\left(y,\cdot\right)\right\Vert _{{\rm TV}}\leqslant\frac{1}{2}\cdot\alpha_{0}$
by Lemma~\ref{lem:PCN-continuity-proposal}. Since $Q$ is $\nu$-reversible,
we may apply Lemma~\ref{lem:met-tv-bound} to conclude. 
\end{proof}
\begin{lem}
\label{lem:pCN-accept-x}Let $\tilde{\kappa}=L\cdot{\rm Tr}\left(\mathsf{C}\right)$.
The pCN chain satisfies
\[
\alpha\left(x\right)\geqslant\frac{1}{2}\cdot\exp\left(-\frac{1}{2}\cdot\eta^{2}\cdot\tilde{\kappa}\right)>0,\qquad x\in\mathsf{E}.
\]
In particular, take $\eta:=\varsigma\cdot\tilde{\kappa}^{-1/2}$ where
$\varsigma\in\left(0,\tilde{\kappa}^{1/2}\right)$ is arbitrary. Then
\[
\alpha_{0}:=\inf_{x\in\mathsf{E}}\alpha\left(x\right)\geqslant\frac{1}{2}\cdot\exp\left(-\frac{1}{2}\cdot\varsigma^{2}\right).
\]
\end{lem}

\begin{proof}
For $x\in\mathsf{E}$, let $\tilde{x}:=\rho\cdot x$ and $w:=\tilde{x}+\eta\cdot z$,
with $z\sim\mathcal{N}\left(0,\mathsf{C}\right)$. Recalling that
$\rho\in\left(0,1\right)$, we can apply convexity of $\Psi$ to see
that
\begin{align*}
\Psi\left(\tilde{x}\right) & =\Psi\left(\left(1-\rho\right)\cdot0+\rho\cdot x\right)\\
 & \leqslant\left(1-\rho\right)\cdot\Psi\left(0\right)+\rho\cdot\Psi\left(x\right)\\
\implies\Psi\left(\tilde{x}\right)-\Psi\left(0\right) & \leqslant\rho\cdot\left(\Psi\left(x\right)-\Psi\left(0\right)\right)\\
 & \leqslant\Psi\left(x\right)-\Psi\left(0\right)\\
\implies\Psi\left(\tilde{x}\right) & \leqslant\Psi\left(x\right).
\end{align*}
Applying $L$-smoothness shows that
\begin{align*}
\Psi\left(w\right) & \leqslant\Psi\left(\tilde{x}\right)+\left\langle \nabla\Psi\left(x\right),\eta\cdot z\right\rangle +\frac{1}{2}\cdot L\cdot\eta^{2}\cdot\left|z\right|^{2}\\
 & \leqslant\Psi\left(x\right)+\left\langle \nabla\Psi\left(x\right),\eta\cdot z\right\rangle +\frac{1}{2}\cdot L\cdot\eta^{2}\cdot\left|z\right|^{2}.
\end{align*}
From here, we imitate the proof of Lemma~\ref{lem:general-accept-x},
writing
\begin{align*}
\alpha\left(x\right) & =\int\mathcal{N}\left({\rm d}z;0,\mathsf{C}\right)\cdot\min\left\{ 1,\exp\left(-\left[\Psi\left(w\right)-\Psi\left(x\right)\right]\right)\right\} \\
 & \geqslant\int\mathcal{N}\left({\rm d}z;0,\mathsf{C}\right)\cdot\min\left\{ 1,\exp\left(-\left[\left\langle \nabla\Psi\left(x\right),\eta\cdot z\right\rangle +\frac{1}{2}\cdot L\cdot\eta^{2}\cdot\left|z\right|^{2}\right]\right)\right\} \\
 & \geqslant\int\mathcal{N}\left({\rm d}z;0,\mathsf{C}\right)\cdot\min\left\{ 1,\exp\left(-\left[\left\langle \nabla\Psi\left(x\right),\eta\cdot z\right\rangle \right]\right)\right\} \cdot\exp\left(-\frac{1}{2}\cdot L\cdot\eta^{2}\cdot\left|z\right|^{2}\right)\\
 & \geqslant\frac{1}{2}\cdot\int\mathcal{N}\left({\rm d}z;0,\mathsf{C}\right)\cdot\exp\left(-\frac{1}{2}\cdot L\cdot\eta^{2}\cdot\left|z\right|^{2}\right)\\
 & \geqslant\frac{1}{2}\cdot\exp\left(-\int\mathcal{N}\left({\rm d}z;0,\mathsf{C}\right)\cdot\frac{1}{2}\cdot L\cdot\eta^{2}\cdot\left|z\right|^{2}\right)\\
 & =\frac{1}{2}\cdot\exp\left(-\frac{1}{2}\cdot L\cdot\eta^{2}\cdot\mathrm{Tr}\left(\mathsf{C}\right)\right).
\end{align*}
The second inequality follows by algebraic substitution.
\end{proof}

\subsection{Mixing time for pCN}

We now give mixing time results for the pCN algorithm; the proof of
Theorem~\ref{thm:pcn-mixing} is in Appendix~\ref{sec:Proofs-of-auxiliary}.
\begin{thm}
\label{thm:pcn-mixing}Under Assumption~\ref{assu:pCN_target}, let
$\mu\ll\pi$ be a probability measure and $u_{0}=\chi^{2}\left(\mu,\pi\right)$.
Let $\tilde{\kappa}=L\cdot{\rm Tr}\left(\mathsf{C}\right)$ and $\eta=\varsigma\cdot\tilde{\kappa}^{-1/2}$
with $\varsigma\in\left(0,\tilde{\kappa}^{1/2}\right)$ arbitrary
and let the universal constants $C_{\ell},C_{\gamma}$ be as defined
in Lemma~\ref{lem:Bounds-on-isoperimetric}. Then, to guarantee that
$\chi^{2}\left(\mu P^{n},\pi\right)\leqslant\varepsilon_{{\rm Mix}}\in\left(0,8\right)$
we may take
\begin{align*}
n & \geqslant2+2^{10}\cdot\exp\left(\varsigma^{2}\right)\cdot\log\left(\max\left\{ u_{0},1\right\} \right)\\
 & +2^{14}\cdot C_{\ell}^{-2}\cdot\exp\left(2\cdot\varsigma^{2}\right)\cdot\varsigma^{-2}\cdot\tilde{\kappa}\cdot\left\{ \log\left(16\cdot C_{\ell}^{-2}\cdot\varsigma^{-2}\cdot\tilde{\kappa}\right)+\varsigma^{2}\right\} \\
 & +2^{6}\cdot C_{\gamma}^{-2}\cdot\exp\left(2\cdot\varsigma^{2}\right)\cdot\varsigma^{-2}\cdot\tilde{\kappa}\cdot\log\left(\frac{8}{\varepsilon_{{\rm Mix}}}\right),
\end{align*}
i.e. of order $\mathcal{O}\left(\log u_{0}+\tilde{\kappa}\log(\tilde{\kappa})+\tilde{\kappa}\log(\varepsilon_{{\rm Mix}}^{-1})\right)$.
\end{thm}

\begin{rem}
We note that in contrast to the RWM, the assumptions made on the target
for pCN allow for a dimension-independent control of the mixing behaviour.
This phenomenon has been observed since at least \citet{hairer2014spectral},
who establish the dimension-robustness of the spectral gap under similar
assumptions, though with much less explicit quantitative results.
On the other hand, one can still quantify the difficulty of navigating
the target measure through the roughness of the potential $\Psi$,
as summarized by $L$, and the effective dimension of the prior, as
summarized by ${\rm Tr}\left(\mathsf{C}\right)$; see \citet{agapiou2017importance}
for related notions.
\end{rem}

\begin{rem}
An analogous initial distribution to that in Remark~\ref{rem:normal-L-initialization-RWM}
is\linebreak{}
 $\mu=\mathcal{N}\left(0,\mathsf{C}\cdot\left(I_{d}+L\cdot\mathsf{C}\right)^{-1}\right)$,
for which
\[
\frac{\mathrm{d}\mu}{\mathrm{d}\pi}\left(x\right)\leqslant\det\left(I_{d}+L\cdot\mathsf{C}\right)^{1/2}\leqslant\exp\left(\frac{1}{2}\cdot L\cdot\mathrm{Tr}\left(\mathsf{C}\right)\right).
\]
Hence, $\log u_{0}\in\mathcal{O}\left(\tilde{\kappa}\right)$, from
which one concludes that the mixing time is bounded as $\mathcal{O}\left(\tilde{\kappa}\cdot\log\tilde{\kappa}+\tilde{\kappa}\cdot\log\varepsilon_{\mathrm{Mix}}^{-1}\right)$.
\end{rem}

\subsection{Comparison with independent Metropolis--Hastings}

A non-local pCN chain may be defined by taking $\rho=0$, and hence
$\eta=1$. This corresponds to an independent Metropolis--Hastings
(IMH) kernel with proposal distribution $q(A)=\mathcal{N}(A;0,\mathsf{C})$.
Theorem~\ref{thm:pcn-gap}, strictly speaking, does not apply but
does allow consideration of $\rho$ arbitrarily close to $0$ by taking
$\varsigma^{2}$ arbitrarily close to $L\cdot{\rm Tr}\left(\mathsf{C}\right)$,
resulting in a spectral gap bound of order $\alpha_{0}^{4}\sim\exp\left(-2\cdot L\cdot{\rm Tr}(C)\right)$.
This is somewhat crude, perhaps because the analysis here is more
suitable for Markov chains with local behaviour. 

On the other hand, we may deduce by \citet[Theorem~2]{gaasemyr2003spectrum}
that the spectral gap of the IMH is precisely 
\[
\gamma_{P}=\alpha_{0}=\inf_{x\in\mathsf{E}}\frac{{\rm d}q}{{\rm d}\pi}(x)=\exp(\Psi(0))\int\mathcal{N}(x;0,\mathsf{C})\exp(-\Psi(x))\,{\rm d}x,
\]
and we will see that when $\rho=0$, the bound on $\alpha_{0}$ deteriorates
rapidly with $d$. By $L$-smoothness and $\nabla\Psi(0)=0$, we have
$\Psi(x)-\Psi(0)\leqslant\frac{L}{2}\left|x\right|^{2}$, so that
we have the estimate
\[
\gamma_{P}\geqslant\int\mathcal{N}(x;0,\mathsf{C})\exp\left(-\frac{L}{2}\left|x\right|^{2}\right)\,{\rm d}x={\rm det}\left(I_{d}+L\mathsf{C}\right)^{-1/2}.
\]
In the context of Remark~\ref{rem:mL-pcn-conversion}, we obtain
\[
\gamma_{P}\geqslant\left(1+\frac{L}{m}\right)^{-d/2}=(1+\kappa)^{-d/2},
\]
which decreases much faster than $d^{-1}$; this also implies poor
scaling of independent Metropolis--Hastings in comparison with RWM.
In general, we only obtain the bound
\[
\gamma_{P}\geqslant\exp\left(-\frac{1}{2}\cdot L\cdot{\rm Tr}(\mathsf{C})\right),
\]
where the exponential dependence on $L\cdot{\rm Tr}(\mathsf{C})$
is much worse than the linear dependence in (\ref{eq:pcn-gap-bound-varsigma})
when $\varsigma$ is chosen appropriately. Hence, significant improvements
in the spectral gap bound are obtained by using appropriately tuned
``local'' Markov chains.

\appendix

\section{\label{sec:Notation}Notation}
\begin{itemize}
\item The Euclidean norm on $\E=\R^{d}$ is denoted $\left|\cdot\right|$,
which we will also use to denote the associated metric.
\item We write ${\rm Leb}$, and plainly $\dif x$, for the Lebesgue measure
on $\R^{d}$.
\item We write $\ELL\left(\pi\right)$ for the Hilbert space of (equivalence
classes of) real-valued $\pi$-square-integrable measurable functions
with inner product 
\[
\left\langle f,g\right\rangle =\int_{\E}f\left(x\right)g\left(x\right)\,\dif\pi\left(x\right)\,,
\]
and corresponding norm $\left\Vert \cdot\right\Vert _{2}$. For $g\in\ELL\left(\pi\right)$,
$\Var_{\pi}\left(g\right):=\left\Vert g-\pi\left(g\right)\right\Vert _{2}^{2}$.
We write $\ELL_{0}\left(\pi\right)$ for the set of functions $f\in\ELL\left(\pi\right)$
which also satisfy $\pi\left(f\right)=0$. 
\item Given a set $A\in\mathcal{E}$ with $\pi\left(A\right)>0$, we define
the probability measure $\pi_{A}$ on $\left(\E,\mathscr{E}\right)$
via $\pi_{A}\left(\cdot\right):=\pi\left(\cdot\cap A\right)/\pi\left(A\right)$.
\item Given a probability measure $\pi$ and a function $T$ on $\E$, we
define the pushforward measure of $\pi$ under the action of $T$
by $\left(T_{\#}\pi\right)\left(A\right):=\pi\left(T^{-1}\left(A\right)\right)$.
\item For a set $A\in\mathscr{E}$, its complement in $\E$ is denoted by
$A^{\complement}$. We denote the corresponding indicator function
by $\mathbf{1}_{A}:\E\to\left\{ 0,1\right\} $.
\item For two sets $A,B\in\mathscr{E}$ and a metric $\mathsf{d}$ on $\E$,
the distance between the two sets is given by
\[
\mathsf{d}\left(A,B\right):=\inf\left\{ \mathsf{d}\left(x,y\right):x\in A,y\in B\right\} .
\]
When one of the sets is a singleton, we will simply write $\mathsf{d}\left(x,B\right)$
for $\mathsf{d}\left(\left\{ x\right\} ,B\right)$, say.
\item For two measures $\mu$ and $\nu$, we write $\nu\ll\mu$ to mean
that $\nu$ is absolutely continuous with respect to $\mu$.
\item For two probability measures $\mu$ and $\nu$ on $\left(\E,\mathscr{E}\right)$,
we let $\mu\otimes\nu\left(A\times B\right):=\mu\left(A\right)\cdot\nu\left(B\right)$
for $A,B\in\mathscr{E}$. For a Markov kernel $P\left(x,\dif y\right)$
on $\E\times\mathscr{E}$, we write for $\bar{A}\in\mathscr{E}\otimes\mathscr{E}$,
the minimal product $\sigma$-algebra, $\mu\otimes P\left(\bar{A}\right):=\int_{\bar{A}}\mu\left(\dif x\right)P\left(x,\dif y\right)$.
\item For a probability measure $\mu\ll\pi$, the chi-squared divergence
between $\mu$ and $\pi$ is given by $\chi^{2}\left(\mu,\pi\right):=\left\Vert \frac{\dif\mu}{\dif\pi}-1\right\Vert _{2}^{2}$.
\item A point mass distribution at $x$ is denoted by $\delta_{x}$.
\item We associate, to a $\pi$-invariant Markov kernel $P$, the bounded
linear operator also denoted $P:\ELL\left(\pi\right)\to\ELL\left(\pi\right)$,
given by $Pf\left(x\right)=\int_{\mathsf{E}}P\left(x,{\rm d}y\right)f\left(y\right)$.
We may refer to $P$ as a kernel or as an operator, the meaning being
clear from the context.
\item We write $\mathrm{Id}$ for the identity mapping on $\ELL\left(\pi\right)$,
and $\mathrm{id}$ for the identity mapping on $\R$, and $I_{d}$
for the $d\times d$ identity matrix.
\item Given a bounded linear operator $P:\ELL\left(\pi\right)\to\ELL\left(\pi\right)$,
we define the Dirichlet form $\calE\left(P,f\right):=\left\langle \left(\Id-P\right)f,f\right\rangle $
for any $f\in\ELL\left(\pi\right)$.
\item For a mapping $T:\left(\E,\mathsf{d}\right)\to\left(\E',\mathsf{d}'\right)$
between metric spaces, the Lipschitz norm is defined as $\left|T\right|_{{\rm Lip}}:=\sup_{x\neq y}\frac{\mathsf{d}'\left(T\left(x\right),T\left(y\right)\right)}{\mathsf{d}\left(x,y\right)}$.
\item The spectrum of a bounded linear operator $P$ is the set 
\[
\mathcal{S}\left(P\right):=\left\{ \lambda\in\mathbb{C}:P-\lambda\cdot{\rm Id}\text{ is not invertible}\right\} .
\]
\item For a $\pi$-invariant Markov kernel $P$, we denote by $\mathcal{S}_{0}\left(P\right)$
the spectrum of the restriction of $P$ to $\ELL_{0}\left(\pi\right)$.
We define the spectral gap of $P$ to be $\gamma_{P}:=1-\sup\left|\mathcal{S}_{0}\left(P\right)\right|$.
\item If $P$ is a $\pi$-reversible Markov kernel, then $\mathcal{S}_{0}\left(P\right)\subseteq\left[-1,1\right]$
and we may define the right spectral gap as ${\rm Gap_{{\rm R}}}\left(P\right):=1-\sup\mathcal{S}_{0}\left(P\right)$,
which satisfies \citep[see,  e.g., ][Theorem~22.A.19]{douc2018markov},
\[
{\rm Gap}_{\mathrm{R}}\left(P\right)=\inf_{g\in\ELL_{0}\left(\pi\right),g\neq0}\frac{\calE\left(P,g\right)}{\left\Vert g\right\Vert _{2}^{2}}\,.
\]
\item We say that a $\pi$-reversible Markov kernel $P$ is positive if
$\left\langle f,Pf\right\rangle \geqslant0$ for all $f\in\ELL\left(\pi\right)$.
In this case, $\gamma_{P}=\GapR\left(P\right)$.
\item We write $\mathcal{N}(m,\Sigma)$ where $m\in\R^{d}$ and $\Sigma$
is a $d\times d$ covariance matrix, for the corresponding Gaussian
distribution on $\R^{d}$, $\mathcal{N}\left(x;m,\Sigma\right)$ for
its density with respect to Lebesgue at $x\in\R^{d}$ and $\mathcal{N}\left(A;m,\Sigma\right)$
for the measure it assigns to $A\in\mathscr{E}$.
\item We adopt the following $\mathcal{O}$ (resp. $\Omega$) notation to
indicate when functions grow no faster than (resp. no slower than)
other functions. For $a\in\mathbb{R}\cup\{\infty\}$,
\begin{itemize}
\item If $f\left(x\right)\in\mathcal{O}\left(g\left(x\right)\right)$ as
$x\to a$, this means that $\underset{x\to a}{\lim\sup}\left|\frac{f\left(x\right)}{g\left(x\right)}\right|<\infty$.
When $a=+\infty$, then we may drop explicit mention of $a$. 
\item If $f\left(x\right)\in\Omega\left(g\left(x\right)\right)$ as $x\to a$,
this means that $\underset{x\to a}{\lim\inf}\left|\frac{f\left(x\right)}{g\left(x\right)}\right|>0$.
In particular $f\in\mathcal{O}\left(g\right)\iff g\in\Omega\left(f\right)$. 
\item We will write $f\left(x\right)\in\Theta\left(g(x)\right)$ as $x\to a$
if both $f\left(x\right)\in\mathcal{O}\left(g\left(x\right)\right)$
and $f\left(x\right)\in\Omega\left(g\left(x\right)\right)$ as $x\to a$.
\end{itemize}
\end{itemize}

\section{\label{sec:Proofs-of-auxiliary}Additional proofs}
\begin{proof}[Proof of Lemma~\ref{lem:cond-spec-profile-gap}]
For the case $v>\frac{1}{2}$, we know that $\Lambda_{P}\left(v\right)\geqslant\GapR\left(P\right)$,
since we are taking an infimum over a smaller set of functions for
any $v$. By Lemma~\ref{lem:lawlersokal}, we have $\GapR\left(P\right)\geqslant\frac{1}{2}\cdot\left[\Phi_{P}^{*}\right]^{2}$.
Now consider the case $v\in\left(0,\frac{1}{2}\right]$, let $A$
be a measurable set such that $0<\pi\left(A\right)\leqslant v$, and
let $h\in C_{0}^{+}\left(A\right)$. Consider the quantity
\begin{align*}
\mathcal{E}_{1}\left(P,h\right) & :=\frac{1}{2}\cdot\int\pi\left({\rm d}x\right)\cdot P\left(x,{\rm d}y\right)\cdot\left|h\left(y\right)-h\left(x\right)\right|
\end{align*}
and observe that, by symmetry of $\pi\otimes P$ , one can write
\begin{align*}
\mathcal{E}_{1}\left(P,h\right) & =\int\pi\left({\rm d}x\right)\cdot P\left(x,{\rm d}y\right)\cdot\left|h\left(y\right)-h\left(x\right)\right|\cdot{\bf 1}\left[h\left(x\right)<h\left(y\right)\right]\\
 & =\int\pi\left({\rm d}x\right)\cdot P\left(x,{\rm d}y\right)\cdot{\bf 1}\left[h\left(x\right)<h\left(y\right)\right]\cdot\left(h\left(y\right)-h\left(x\right)\right)\\
 & =\int\pi\left({\rm d}x\right)\cdot P\left(x,{\rm d}y\right)\cdot\int_{t\geqslant0}{\bf 1}\left[h\left(x\right)\leqslant t<h\left(y\right)\right]\,{\rm d}t\\
 & =\int_{t\geqslant0}\left(\int\pi\left({\rm d}x\right)\cdot P\left(x,{\rm d}y\right)\cdot{\bf 1}\left[h\left(x\right)\leqslant t<h\left(y\right)\right]\right)\,{\rm d}t.
\end{align*}
Now, observe that if one defines $H_{t}:=\left\{ x\in A:h\left(x\right)>t\right\} $,
then
\begin{align*}
\int\pi\left({\rm d}x\right)\cdot P\left(x,{\rm d}y\right)\cdot{\bf 1}\left[h\left(x\right)\leqslant t<h\left(y\right)\right] & =\left(\pi\otimes P\right)\left(H_{t}^{\complement}\times H_{t}\right)\\
 & =\left(\pi\otimes P\right)\left(H_{t}\times H_{t}^{\complement}\right).
\end{align*}
Recalling that $H_{t}\subseteq A$ and hence that $\pi\left(H_{t}\right)\leqslant\pi\left(A\right)$,
one sees that if $\pi(H_{t})>0$,
\begin{align*}
\left(\pi\otimes P\right)\left(H_{t}\times H_{t}^{\complement}\right) & =\pi\left(H_{t}\right)\cdot\frac{\left(\pi\otimes P\right)\left(H_{t}\times H_{t}^{\complement}\right)}{\pi\left(H_{t}\right)}\\
 & \geqslant\pi\left(H_{t}\right)\cdot\inf\left\{ \frac{\left(\pi\otimes P\right)\left(S\times S^{\complement}\right)}{\pi\left(S\right)}:0<\pi\left(S\right)\leqslant\pi\left(A\right)\right\} \\
 & =\pi\left(H_{t}\right)\cdot\Phi_{P}\left(\pi\left(A\right)\right),
\end{align*}
while the inequality holds trivially if $\pi(H_{t})=0$. It thus holds
that
\begin{align*}
\mathcal{E}_{1}\left(P,h\right) & =\int_{t\geqslant0}\left(\int\pi\left({\rm d}x\right)\cdot P\left(x,{\rm d}y\right)\cdot{\bf 1}\left[h\left(x\right)\leqslant t<h\left(y\right)\right]\right)\,{\rm d}t\\
 & =\int_{t\geqslant0}\left(\pi\otimes P\right)\left(H_{t}\times H_{t}^{\complement}\right)\,{\rm d}t\\
 & \geqslant\int_{t\geqslant0}\pi\left(H_{t}\right)\cdot\Phi_{P}\left(\pi\left(A\right)\right)\,{\rm d}t\\
 & =\left(\int_{t\geqslant0}\int\pi\left({\rm d}x\right)\cdot{\bf 1}\left[h\left(x\right)>t\right]\,{\rm d}t\right)\cdot\Phi_{P}\left(\pi\left(A\right)\right)\\
 & =\left(\int\pi\left({\rm d}x\right)\cdot\int_{t\geqslant0}{\bf 1}\left[h\left(x\right)>t\right]\,{\rm d}t\right)\cdot\Phi_{P}\left(\pi\left(A\right)\right)\\
 & =\left(\int\pi\left({\rm d}x\right)\cdot h\left(x\right)\,{\rm d}t\right)\cdot\Phi_{P}\left(\pi\left(A\right)\right)\\
 & =\pi\left(h\right)\cdot\Phi_{P}\left(\pi\left(A\right)\right).
\end{align*}
 Now, let $g\in C_{0}^{+}\left(A\right)$, and take $h=g^{2}$ in
the above to see that
\begin{align*}
\pi\left(g^{2}\right)\cdot\Phi_{P}\left(\pi\left(A\right)\right) & \leqslant\frac{1}{2}\cdot\int\pi\left({\rm d}x\right)P\left(x,{\rm d}y\right)\left|g\left(x\right)^{2}-g\left(y\right)^{2}\right|\\
 & \leqslant\frac{1}{2}\cdot\left(\int\pi\left({\rm d}x\right)P\left(x,{\rm d}y\right)\left|g\left(x\right)-g\left(y\right)\right|^{2}\right)^{1/2}\\
 & \hspace{3cm}\cdot\left(\int\pi\left({\rm d}x\right)P\left(x,{\rm d}y\right)\left|g\left(x\right)+g\left(y\right)\right|^{2}\right)^{1/2}\\
 & \leqslant\frac{1}{2}\cdot\left(2\cdot\mathcal{E}\left(P,g\right)\right)^{1/2}\cdot\left(4\cdot\pi\left(g^{2}\right)\right)^{1/2}\\
 & =2^{1/2}\cdot\mathcal{E}\left(P,g\right)^{1/2}\cdot\pi\left(g^{2}\right)^{1/2},
\end{align*}
from which we may deduce that 
\[
\frac{1}{2}\cdot\Phi_{P}\left(\pi\left(A\right)\right)^{2}\leqslant\frac{\mathcal{E}\left(P,g\right)}{\pi\left(g^{2}\right)}\leqslant\frac{\mathcal{E}\left(P,g\right)}{{\rm Var}_{\pi}\left(g\right)}.
\]
Taking an infimum over $g$ shows that $\lambda_{P}\left(A\right)\geqslant\frac{1}{2}\cdot\Phi_{P}\left(\pi\left(A\right)\right)^{2}$,
and taking an infimum over $A$ shows that $\Lambda_{P}\left(v\right)\geqslant\frac{1}{2}\cdot\Phi_{P}\left(v\right)^{2}$.
\end{proof}
\begin{proof}[Proof of Lemma~\ref{lem:iso-conductance-A-close-coupling}]
For $A\in\mathscr{E}$, define the sets
\begin{align*}
S_{1} & :=\left\{ z\in A\colon P\left(z,A^{\complement}\right)<\frac{1}{2}\cdot\varepsilon\right\} \\
S_{2} & :=\left\{ z\in A^{\complement}\colon P\left(z,A\right)<\frac{1}{2}\cdot\varepsilon\right\} 
\end{align*}
and $S_{3}:=\big(S_{1}\cup S_{2}\big)^{\complement}$, and let $\theta\in\left(0,1\right)$.
We consider two cases. First, we establish that when either $\pi\left(S_{1}\right)\leqslant\theta\cdot\pi\left(A\right)$
or $\pi\left(S_{2}\right)\leqslant\theta\cdot\pi\left(A^{\complement}\right)$,
then 
\begin{equation}
\left(\pi\otimes P\right)\big(A\times A^{\complement}\big)\geqslant\frac{1}{2}\cdot\left(1-\theta\right)\cdot\varepsilon\cdot\min\left\{ \pi\left(A\right),\pi\big(A^{\complement}\big)\right\} \,.\label{eq:3set-cond-S1-S2-small}
\end{equation}
If $\pi\left(S_{1}\right)\leqslant\theta\cdot\pi\left(A\right)$ then
\begin{align*}
\pi\left(A\right) & =\pi\left(S_{1}\right)+\pi\left(A\setminus S_{1}\right)\\
 & \leqslant\theta\cdot\pi\left(A\right)+\pi\left(A\setminus S_{1}\right)\\
\implies\quad\pi\left(A\setminus S_{1}\right) & \geqslant\left(1-\theta\right)\cdot\pi\left(A\right)
\end{align*}
Now, 
\begin{align*}
\left(\pi\otimes P\right)\left(A\times A^{\complement}\right) & \geqslant\left(\pi\otimes P\right)\left(\left(A\setminus S_{1}\right)\times A^{\complement}\right)\\
 & \geqslant\frac{1}{2}\cdot\varepsilon\cdot\pi\left(A\setminus S_{1}\right)\\
 & \geqslant\frac{1}{2}\cdot\left(1-\theta\right)\cdot\varepsilon\cdot\pi\left(A\right).
\end{align*}
Similarly, if $\pi\left(S_{2}\right)\leqslant\theta\cdot\pi\left(A^{\complement}\right)$
then
\begin{align*}
\pi\left(A^{\complement}\right) & =\pi\left(S_{2}\right)+\pi\left(A^{\complement}\setminus S_{2}\right).\\
 & \leqslant\theta\cdot\pi\left(A^{\complement}\right)+\pi\left(A^{\complement}\setminus S_{2}\right)\\
\implies\quad\pi\left(A^{\complement}\setminus S_{2}\right) & \geqslant\left(1-\theta\right)\cdot\pi\left(A^{\complement}\right),
\end{align*}
and arguing as before:
\begin{align*}
\left(\pi\otimes P\right)\left(A^{\complement}\times A\right) & \geqslant\left(\pi\otimes P\right)\left(\left(A^{\complement}\setminus S_{2}\right)\times A\right)\\
 & \geqslant\frac{1}{2}\cdot\varepsilon\cdot\pi\left(A^{\complement}\setminus S_{2}\right)\\
 & \geqslant\frac{1}{2}\cdot\left(1-\theta\right)\cdot\varepsilon\cdot\pi\left(A^{\complement}\right).
\end{align*}
The first claim thus follows. In the second case, $\pi\left(S_{1}\right)>\theta\cdot\pi\left(A\right)$
and $\pi\left(S_{2}\right)>\theta\cdot\pi\left(A^{\complement}\right)$.
As noticed by \citet{Dwivedi-Chen-Wainwright-JMLR:v20:19-306}, reversibility
is not required to establish the following
\begin{align*}
\left(\pi\otimes P\right)\left(A\times A^{\complement}\right) & =\left(\pi\otimes P\right)\left(\mathsf{E}\times A^{\complement}\right)-\left[\left(\pi\otimes P\right)\left(A^{\complement}\times\mathsf{E}\right)-\left(\pi\otimes P\right)\left(A^{\complement}\times A\right)\right]\\
 & =\pi\left(A^{\complement}\right)-\pi\left(A^{\complement}\right)+\left(\pi\otimes P\right)\left(A^{\complement}\times A\right)\\
 & =\left(\pi\otimes P\right)\left(A^{\complement}\times A\right)\,,
\end{align*}
We then compute
\begin{align*}
\left(\pi\otimes P\right)\left(A\times A^{\complement}\right) & =\frac{1}{2}\cdot\left(\pi\otimes P\right)\left(A\times A^{\complement}\right)+\frac{1}{2}\cdot\left(\pi\otimes P\right)\left(A^{\complement}\times A\right)\\
 & \geqslant\frac{1}{2}\cdot\left(\pi\otimes P\right)\left(\left(A\setminus S_{1}\right)\times A^{\complement}\right)+\frac{1}{2}\cdot\left(\pi\otimes P\right)\left(\left(A^{\complement}\setminus S_{2}\right)\times A\right)\\
 & \geqslant\frac{1}{4}\cdot\varepsilon\cdot\pi\left(A\setminus S_{1}\right)+\frac{1}{4}\cdot\varepsilon\cdot\pi\left(A^{\complement}\setminus S_{2}\right)\\
 & =\frac{1}{4}\cdot\varepsilon\cdot\pi\left(S_{3}\right)
\end{align*}
Now for $\left(z,z'\right)\in S_{1}\times S_{2}$ we have
\begin{align*}
\left\Vert P\left(z,\cdot\right)-P\left(z',\cdot\right)\right\Vert _{\mathrm{TV}} & \geqslant P\left(z,A\right)-P\left(z',A\right)\\
 & =1-P\left(z,A^{\complement}\right)-P\left(z',A\right)\\
 & >1-\varepsilon.
\end{align*}
This implies that $\mathsf{d}\left(S_{1},S_{2}\right)=\inf\left\{ \left|z-z'\right|\colon\left(z,z'\right)\in S_{1}\times S_{2}\right\} \geqslant\delta$,
since $P$ is $(\mathsf{d},\delta,\varepsilon)$-close coupling. Hence,
using Definition~\ref{def:three-set-iso-ineq} and monotonicity of
$F$, 
\begin{align*}
\left(\pi\otimes P\right)\left(A\times A^{\complement}\right) & \geqslant\frac{1}{4}\cdot\varepsilon\cdot\pi\left(S_{3}\right)\\
 & \geqslant\frac{1}{4}\cdot\varepsilon\cdot\mathsf{d}\left(S_{1},S_{2}\right)\cdot F\left(\min\left\{ \pi\left(S_{1}\right),\pi\left(S_{2}\right)\right\} \right)\\
 & \geqslant\frac{1}{4}\cdot\varepsilon\cdot\delta\cdot F\left(\min\left\{ \theta\cdot\pi\left(A\right),\theta\cdot\pi\left(A^{\complement}\right)\right\} \right).
\end{align*}
We conclude by combining this inequality with (\ref{eq:3set-cond-S1-S2-small})
and considering $A$ with $\pi(A)\leqslant\frac{1}{2}$.
\end{proof}
\begin{proof}[Proof of Theorem~\ref{thm:IP-to-mix}]
By Corollary~\ref{cor:iso-couple-conductance}, we have 
\begin{align*}
\Phi_{P}\left(v\right) & \geqslant\frac{1}{4}\cdot\varepsilon\cdot\min\left\{ 1,\frac{1}{2}\cdot\delta\cdot\frac{\tilde{I}_{\pi}\left(\frac{1}{2}\cdot v\right)}{\frac{1}{2}\cdot v}\right\} \qquad v\in\big(0,1/2\big],
\end{align*}
and the bounds on $\Phi_{P}^{*}$ and $\gamma_{P}$ follow then from
Definition~\ref{def:conductance-profile}, Lemma~\ref{lem:lawlersokal}
and positivity of $P$. Writing $h=\frac{{\rm d}\mu}{{\rm d}\pi}$
and $u_{n}:={\rm Var}_{\pi}\left(P^{n}h\right)=\chi^{2}\left(\mu P^{n},\pi\right)$,
we recall by Theorem~\ref{thm:CP-to-mix} that in order to ensure
that $u_{n}\leqslant\varepsilon_{{\rm Mix}}$, it suffices to take
\[
n\geqslant2+4\cdot\int_{\min\left\{ 4\cdot u_{0}^{-1},1/2\right\} }^{1/2}\frac{{\rm d}v}{v\cdot\Phi_{P}\left(v\right)^{2}}+\left[\Phi_{P}^{*}\right]^{-2}\cdot\log\left(\max\left\{ \frac{\min\left\{ u_{0},8\right\} }{\varepsilon_{{\rm Mix}}},1\right\} \right).
\]
Then for $v\in\left(0,v_{*}\right)$, as defined in (\ref{eq:v*_defn}),
it holds that $\Phi_{P}\left(v\right)\geqslant\frac{1}{4}\cdot\varepsilon$,
and for $v\in\left(v_{*},\frac{1}{2}\right)$, it holds that $\Phi_{P}\left(v\right)\geqslant\frac{1}{8}\cdot\varepsilon\cdot\delta\cdot\frac{\tilde{I}_{\pi}\left(\frac{1}{2}\cdot v\right)}{\frac{1}{2}\cdot v}$.
One thus writes
\begin{align*}
\int_{\min\left\{ 4\cdot u_{0}^{-1},1/2\right\} }^{1/2}\frac{1}{v\cdot\Phi_{P}\left(v\right)^{2}}\,{\rm d}v & =\int_{\min\left\{ 4\cdot u_{0}^{-1},1/2\right\} }^{\max\left\{ \min\left\{ 4\cdot u_{0}^{-1},1/2\right\} ,v_{*}\right\} }\frac{1}{v\cdot\Phi_{P}\left(v\right)^{2}}\,{\rm d}v\\
 & +\int_{\max\left\{ \min\left\{ 4\cdot u_{0}^{-1},1/2\right\} ,v_{*}\right\} }^{1/2}\frac{1}{v\cdot\Phi_{P}\left(v\right)^{2}}\,{\rm d}v.
\end{align*}
We treat the two integrals separately. For the first, write
\begin{align*}
\int_{\min\left\{ 4\cdot u_{0}^{-1},1/2\right\} }^{\max\left\{ \min\left\{ 4\cdot u_{0}^{-1},1/2\right\} ,v_{*}\right\} }\frac{{\rm d}v}{v\cdot\Phi_{P}\left(v\right)^{2}} & \leqslant\int_{\min\left\{ 4\cdot u_{0}^{-1},1/2\right\} }^{\max\left\{ \min\left\{ 4\cdot u_{0}^{-1},1/2\right\} ,v_{*}\right\} }\frac{1}{v\cdot\left(\frac{1}{4}\cdot\varepsilon\right)^{2}}\,{\rm d}v\\
 & =2^{4}\cdot\varepsilon^{-2}\cdot\log\left(\frac{\max\left\{ \min\left\{ 4\cdot u_{0}^{-1},1/2\right\} ,v_{*}\right\} }{\min\left\{ 4\cdot u_{0}^{-1},\frac{1}{2}\right\} }\right)\\
 & =2^{4}\cdot\varepsilon^{-2}\cdot\max\left\{ \log\left(\frac{u_{0}}{4\cdot v_{*}^{-1}}\right),0\right\} ,
\end{align*}
where the final equality follows from a case-by-case analysis. For
the second, write
\begin{align*}
\int_{\max\left\{ \min\left\{ 4\cdot u_{0}^{-1},1/2\right\} ,v_{*}\right\} }^{1/2} & \frac{{\rm d}v}{v\cdot\Phi_{P}\left(v\right)^{2}}\\
 & \leqslant\int_{\max\left\{ \min\left\{ 4\cdot u_{0}^{-1},1/2\right\} ,v_{*}\right\} }^{1/2}\frac{1}{v\cdot\left(\frac{1}{8}\cdot\varepsilon\cdot\delta\cdot\frac{\tilde{I}_{\pi}\left(\frac{1}{2}\cdot v\right)}{\frac{1}{2}\cdot v}\right)^{2}}\,{\rm d}v\\
 & =2^{6}\cdot\varepsilon^{-2}\cdot\delta^{-2}\cdot\int_{\max\left\{ \min\left\{ 4\cdot u_{0}^{-1},1/2\right\} ,v_{*}\right\} }^{1/2}\frac{\left(\frac{1}{2}\cdot v\right)^{2}}{v\cdot\tilde{I}_{\pi}\left(\frac{1}{2}\cdot v\right)^{2}}\,{\rm d}v\\
 & =2^{6}\cdot\varepsilon^{-2}\cdot\delta^{-2}\cdot\int_{\max\left\{ \min\left\{ 4\cdot u_{0}^{-1},1/2\right\} ,v_{*}\right\} }^{1/2}\frac{\left(\frac{1}{2}\cdot v\right)}{\tilde{I}_{\pi}\left(\frac{1}{2}\cdot v\right)^{2}}\,{\rm d}\left(\frac{1}{2}\cdot v\right)\\
 & =2^{6}\cdot\varepsilon^{-2}\cdot\delta^{-2}\cdot\int_{\max\left\{ \min\left\{ 2\cdot u_{0}^{-1},1/4\right\} ,v_{*}/2\right\} }^{1/4}\frac{\xi}{\tilde{I}_{\pi}\left(\xi\right)^{2}}\,{\rm d}\xi.
\end{align*}
For the final term, we have $\Phi_{P}^{*}=\Phi_{P}(\frac{1}{2})\geqslant2^{-2}\cdot\varepsilon\cdot\min\left\{ 1,2\cdot\delta\cdot\tilde{I}_{\pi}\left(\frac{1}{4}\right)\right\} $.
We conclude by combining these bounds.
\end{proof}
\begin{proof}[Proof of Lemma~\ref{lem:RWM-continuity-proposal}]
This is obtained via Pinsker's inequality. First compute directly
that
\[
{\rm KL}\left(Q\left(x,\cdot\right),Q\left(y,\cdot\right)\right)=\frac{1}{2\cdot\sigma^{2}}\cdot\left|x-y\right|^{2}.
\]
Recalling Pinsker's inequality, we deduce that
\[
\left\Vert Q\left(x,\cdot\right)-Q\left(y,\cdot\right)\right\Vert _{\mathrm{TV}}\leqslant\left(\frac{1}{2}\cdot\mathrm{KL}\left(Q\left(x,\cdot\right),Q\left(y,\cdot\right)\right)\right)^{1/2}=\frac{1}{2\cdot\sigma}\cdot\left|x-y\right|.\qedhere
\]
\end{proof}
\begin{proof}[Proof of Proposition~\ref{prop:kappa-upper-small-sigma}]
First, let $\nu_{1}$ be a median of $x_{1}$, the first coordinate
of $x$, under $\pi$, and let $x_{*,1}$ be the mode of the marginal
law of $x_{1}$ under $\pi$, which exists and is unique as a consequence
of Lemma~\ref{lem:smooth_cvx_pres}. Now, define $A=\left\{ x\in\mathsf{E}:x_{1}\geqslant\nu_{1}\right\} $,
so that $\pi\left(A\right)=\frac{1}{2}$. We let $Z\sim\mathcal{N}\left(0,I_{d}\right)$,
and by neglecting the acceptance probability, we obtain the bounds

\begin{align*}
\left(\pi\otimes P\right)\left(A\times A^{\complement}\right) & =\int_{A}\pi\left(\mathrm{d}x\right)P\left(x,A^{\complement}\right)\\
 & =\int_{A}\pi\left(\mathrm{d}x\right)\int_{\mathsf{E}}{\cal N}\left(\mathrm{d}z;0,I_{d}\right)\cdot\min\left\{ 1,\frac{\pi\left(x+\sigma\cdot z\right)}{\pi\left(x\right)}\right\} \cdot{\bf 1}_{A^{\complement}}\left(x+\sigma\cdot z\right)\\
 & \leqslant\int_{A}\pi\left(\mathrm{d}x\right)\mathbb{P}\left(x+\sigma\cdot Z\in A^{\complement}\right)\\
 & =\int_{A}\pi\left(\mathrm{d}x\right)\mathbb{P}\left(x_{1}+\sigma\cdot Z_{1}<\nu_{1}\right)\\
 & =\int_{x_{1}\geqslant\nu_{1}}\pi_{1}\left(\mathrm{d}x_{1}\right)\mathbb{P}\left(x_{1}+\sigma\cdot Z_{1}<\nu_{1}\right),
\end{align*}
where $\pi_{1}$ is the marginal law of $x_{1}$ under $\pi$. Recall
that by Lemma~\ref{lem:smooth_cvx_pres}, $\pi_{1}$ will also be
$m$-strongly log-concave and admit an $L$-smooth potential, and
so we may apply Lemma~\ref{lem:gauss-bd} to control the density
of $\pi_{1}$ as $\pi_{1}\left({\rm d}x_{1}\right)\leqslant\left(\frac{L}{m}\right)^{1/2}\cdot{\cal N}\left({\rm d}x_{1};x_{*,1},m^{-1}\right)$.
Substituting this, we may bound
\begin{align*}
\left(\pi\otimes P\right)\left(A\times A^{\complement}\right) & \leqslant\left(\frac{L}{m}\right)^{1/2}\cdot\int_{x_{1}\geqslant\nu_{1}}{\cal N}\left(\mathrm{d}x_{1};x_{*,1},m^{-1}\right)\mathbb{P}\left(x_{1}+\sigma\cdot Z_{1}<\nu_{1}\right)\\
 & =\left(\frac{L}{m}\right)^{1/2}\cdot\int_{x_{1}\geqslant\nu_{1}}{\cal N}\left(\mathrm{d}x_{1};x_{*,1},m^{-1}\right)\mathbb{P}\left(Z_{1}<\frac{\nu_{1}-x_{1}}{\sigma}\right)\\
 & \leqslant\left(\frac{L}{m}\right)^{1/2}\cdot\int_{x_{1}\geqslant\nu_{1}}{\cal N}\left(\mathrm{d}x_{1};x_{*,1},m^{-1}\right)\exp\left(-\frac{1}{2\cdot\sigma^{2}}\cdot\left(\nu_{1}-x_{1}\right)^{2}\right)\\
 & =\left(\frac{L}{2\pi}\right)^{1/2}\cdot\int_{x_{1}\geqslant\nu_{1}}\exp\left(-\frac{m}{2}\cdot\left(x_{1}-x_{*,1}\right)^{2}-\frac{1}{2\cdot\sigma^{2}}\cdot\left(\nu_{1}-x_{1}\right)^{2}\right){\rm d}x_{1},
\end{align*}
where we have used the Chernoff bound $\mathbb{P}\left(Z_{1}\leqslant-z\right)\leqslant\exp\left(-\frac{1}{2}z^{2}\right)$
for $z>0$, to move from the second to the third line.

Computing directly that
\begin{align*}
\frac{m}{2}\left(x_{1}-x_{*,1}\right)^{2}+\frac{1}{2\cdot\sigma^{2}}\cdot\left(\nu_{1}-x_{1}\right)^{2} & =\frac{1}{2}\cdot\frac{1+m\cdot\sigma^{2}}{\sigma^{2}}\cdot\left(x_{1}-\frac{\nu_{1}+m\cdot\sigma^{2}\cdot x_{*,1}}{1+m\cdot\sigma^{2}}\right)^{2}\\
 & \hspace{3cm}+\frac{m}{2\cdot\left(1+m\cdot\sigma^{2}\right)}\cdot\left(\nu_{1}-x_{*,1}\right)^{2},
\end{align*}
one sees that
\begin{align*}
\, & \pi\otimes P\left(A\times A^{\complement}\right)\leqslant\left(\frac{L}{2\pi}\right)^{1/2}\cdot\int_{x_{1}\geqslant\nu_{1}}\exp\left(-\frac{1}{2}\cdot\frac{1+m\cdot\sigma^{2}}{\sigma^{2}}\cdot\left(x_{1}-\frac{\nu_{1}+m\cdot\sigma^{2}\cdot x_{*,1}}{1+m\cdot\sigma^{2}}\right)^{2}\right.\\
 & \hspace{6cm}\cdot\exp\left(\left.-\frac{m}{2\cdot\left(1+m\cdot\sigma^{2}\right)}\cdot\left(\nu_{1}-x_{*,1}\right)^{2}\right)\right)\,\mathrm{d}x_{1}\\
= & \left(\frac{L\cdot\sigma^{2}}{1+m\cdot\sigma^{2}}\right)^{1/2}\cdot\exp\left(-\frac{m}{2\cdot\left(1+m\cdot\sigma^{2}\right)}\cdot\left(\nu_{1}-x_{*,1}\right)^{2}\right)\\
 & \hspace{6cm}\cdot\int_{x_{1}\geqslant\nu_{1}}{\cal N}\left(\mathrm{d}x_{1};\frac{\nu_{1}+m\cdot\sigma^{2}\cdot x_{*,1}}{1+m\cdot\sigma^{2}},\frac{\sigma^{2}}{1+m\cdot\sigma^{2}}\right)\\
\leqslant & \left(\frac{L\cdot\sigma^{2}}{1+m\cdot\sigma^{2}}\right)^{1/2}\cdot1\cdot1\\
\leqslant & \:L^{1/2}\cdot\sigma.
\end{align*}
Recalling that $\pi\left(A\right)=\pi\left(A^{\complement}\right)=\frac{1}{2}$,
the result follows.
\end{proof}
\begin{proof}[Proof of Proposition~\ref{prop:kappa-upper-big-sigma}]
By $m$-strong convexity of the potential, it holds that
\[
U\left(x+\sigma\cdot z\right)-U\left(x\right)\geqslant\left\langle \nabla U\left(x\right),\sigma\cdot z\right\rangle +\frac{1}{2}\cdot m\cdot\sigma^{2}\cdot\left|z\right|^{2},
\]
and substituting this into (\ref{eq:RWM-alpha-x}) we obtain
\begin{align*}
\alpha\left(x\right) & \leqslant\int{\cal N}\left(\mathrm{d}z;0,I_{d}\right)\cdot\min\left\{ 1,\exp\left(-\left\langle \nabla U\left(x\right),\sigma\cdot z\right\rangle -\frac{1}{2}\cdot m\cdot\sigma^{2}\cdot\left|z\right|^{2}\right)\right\} .
\end{align*}
Applying the inequality $\min\left\{ 1,c\right\} \leqslant c$ establishes
that
\[
\alpha\left(x\right)\leqslant\int{\cal N}\left(\mathrm{d}z;0,I_{d}\right)\cdot\exp\left(-\left\langle \nabla U\left(x\right),\sigma\cdot z\right\rangle -\frac{1}{2}\cdot m\cdot\sigma^{2}\cdot\left|z\right|^{2}\right).
\]
Straightforward computations show that
\begin{align*}
 & {\cal N}\left(\mathrm{d}z;0,I_{d}\right)\cdot\exp\left(-\left\langle \nabla U\left(x\right),\sigma\cdot z\right\rangle -\frac{1}{2}\cdot m\cdot\sigma^{2}\cdot\left|z\right|^{2}\right)=\left(1+m\cdot\sigma^{2}\right)^{-d/2}\\
 & \hspace{1cm}\cdot\exp\left(\frac{1}{2}\cdot\frac{\sigma^{2}\cdot\left|\nabla U\left(x\right)\right|^{2}}{1+m\cdot\sigma^{2}}\right)\cdot{\cal N}\left(\mathrm{d}z;-\frac{\sigma}{1+m\cdot\sigma^{2}}\cdot\nabla U\left(x\right),\frac{1}{1+m\cdot\sigma^{2}}\cdot I_{d}\right),
\end{align*}
which allows us to write
\[
\alpha\left(x\right)\leqslant\left(1+m\cdot\sigma^{2}\right)^{-d/2}\cdot\exp\left(\frac{1}{2}\cdot\frac{\sigma^{2}}{1+m\cdot\sigma^{2}}\cdot\left|\nabla U\left(x\right)\right|^{2}\right).
\]
Now, for $\rho>0$, define the set
\[
B_{\rho}=\left\{ x:\frac{\sigma^{2}}{1+m\cdot\sigma^{2}}\cdot\left|\nabla U\left(x\right)\right|^{2}\leqslant\rho^{2}\right\} ,
\]
which for $\rho$ small enough will have $0<\pi\left(B_{\rho}\right)<\pi\left(B_{\rho}^{\complement}\right)$.
It then follows that
\begin{align*}
\Phi_{P}^{*} & \leqslant\frac{\pi\otimes P\left(B_{\rho}\times B_{\rho}^{\complement}\right)}{\pi\left(B_{\rho}\right)}\\
 & =\int\pi_{B_{\rho}}\left(\mathrm{d}x\right)\cdot P\left(x,B_{\rho}^{\complement}\right)\\
 & \leqslant\int\pi_{B_{\rho}}\left(\mathrm{d}x\right)\cdot P\left(x,\left\{ x\right\} ^{\complement}\right)\\
 & =\int\pi_{B_{\rho}}\left(\mathrm{d}x\right)\cdot\alpha\left(x\right)\\
 & \leqslant\left(1+m\cdot\sigma^{2}\right)^{-d/2}\cdot\exp\left(\frac{1}{2}\cdot\rho^{2}\right),
\end{align*}
and taking an infimum as $\rho\to0^{+}$ gives that $\Phi_{P}\leqslant\left(1+m\cdot\sigma^{2}\right)^{-d/2}$,
as claimed.

For the subsequent remark, compute that
\begin{align*}
\log\Phi_{P}^{*} & \leqslant-\frac{1}{2}\cdot d\cdot\log\left(1+m\cdot\sigma^{2}\right)\\
 & =-\frac{1}{2}\cdot d\cdot\log\left(1+\frac{m}{L}\cdot\frac{\varsigma^{2}}{d^{2\cdot\beta}}\right)\\
 & =-\frac{1}{2}\cdot d^{1-2\cdot\beta}\cdot\left(d^{2\cdot\beta}\cdot\log\left(1+\frac{m}{L}\cdot\frac{\varsigma^{2}}{d^{2\cdot\beta}}\right)\right)\\
 & =-\frac{1}{2}\cdot d^{1-2\cdot\beta}\cdot\left(\frac{m}{L}\cdot\varsigma^{2}+\mathcal{O}\left(d^{-2\cdot\beta}\right)\right)\\
 & =-\frac{1}{2}\cdot d^{1-2\cdot\beta}\cdot\frac{m}{L}\cdot\varsigma^{2}\cdot\left(1+\mathcal{O}\left(d^{-2\cdot\beta}\right)\right)\\
 & \leqslant-\frac{1}{2}\cdot\frac{m}{L}\cdot\varsigma^{2}\cdot d^{1-2\cdot\beta}\cdot\left(1+o\left(1\right)\right).
\end{align*}
Fixing $c\in\left(0,\frac{1}{2}\cdot\frac{m}{L}\cdot\varsigma^{2}\right)$,
it holds for sufficiently large $d$ that $\log\Phi_{P}^{*}\leqslant-c\cdot d^{1-2\cdot\beta}$.
\end{proof}
\begin{proof}[Proof of Theorem~\ref{thm:RWM_mixing}]
From Lemma~\ref{lem:Bounds-on-isoperimetric}, we can write $I_{\pi}\left(p\right)\geqslant\tilde{I}_{\pi}\left(p\right):=m^{1/2}\cdot\left(\varphi_{\gamma}\circ\Phi_{\gamma}^{-1}\right)\left(p\right)$,
which admits the bounds
\[
\tilde{I}_{\pi}\left(p\right)\geqslant C_{\ell}\cdot m^{1/2}\cdot p\cdot\left(\log\frac{1}{p}\right)^{1/2}\qquad p\in\big(0,1/2\big],
\]
and $\tilde{I}_{\pi}\left(\frac{1}{4}\right)=m^{1/2}\cdot C_{{\rm \gamma}}$.
From Theorem~\ref{thm:IP-to-mix},
\begin{align*}
n & \geqslant2+2^{6}\cdot\varepsilon^{-2}\cdot\max\left\{ \log\left(\frac{u_{0}}{4\cdot v_{*}^{-1}}\right),0\right\} \\
 & +2^{8}\cdot\varepsilon^{-2}\cdot\delta^{-2}\cdot\int_{\max\left\{ \min\left\{ 2\cdot u_{0}^{-1},1/4\right\} ,v_{*}/2\right\} }^{1/4}\frac{\xi}{\tilde{I}_{\pi}\left(\xi\right)^{2}}\,{\rm d}\xi\\
 & +2^{4}\cdot\max\left\{ 1,2^{-2}\cdot\delta^{-2}\cdot\tilde{I}_{\pi}\left(\frac{1}{4}\right)^{-2}\right\} \cdot\varepsilon^{-2}\cdot\log\left(\max\left\{ \frac{\min\left\{ u_{0},8\right\} }{\varepsilon_{{\rm Mix}}},1\right\} \right)
\end{align*}
is sufficient, where $v_{*}$ is defined in (\ref{eq:v*_defn}). Additionally,
from Lemma~\ref{lem:close-coupling}, $P$ is $\left(\left|\cdot\right|,\alpha_{0}\cdot\sigma,\frac{1}{2}\cdot\alpha_{0}\right)$-close
coupling. Substituting these values and using the lower bound on $\tilde{I}_{\pi}\left(p\right)$
to upper bound the integrand and lower bound $v_{*}\geqslant v_{\circ}:=\min\left\{ \frac{1}{2},2\cdot\exp\left(-4\cdot C_{\ell}^{-2}\cdot\sigma^{-2}\cdot\alpha_{0}^{-2}\cdot m^{-1}\right)\right\} $
it suffices to take 
\begin{align*}
n & \geqslant2+2^{8}\cdot\alpha_{0}^{-2}\cdot\max\left\{ \log\left(\frac{u_{0}}{4\cdot v_{\circ}^{-1}}\right),0\right\} \\
 & +2^{10}\cdot C_{\ell}^{-2}\cdot\alpha_{0}^{-4}\cdot\sigma^{-2}\cdot m^{-1}\cdot\int_{\max\left\{ \min\left\{ 2\cdot u_{0}^{-1},1/4\right\} ,v_{\circ}/2\right\} }^{1/4}\frac{1}{\xi\cdot\log\left(\frac{1}{\xi}\right)}\,{\rm d}\xi\\
 & +2^{4}\cdot\max\left\{ 1,2^{-2}\cdot C_{\gamma}^{-2}\cdot\alpha_{0}^{-2}\cdot\sigma^{-2}\cdot m^{-1}\right\} \cdot\alpha_{0}^{-2}\cdot\sigma^{-2}\cdot\log\left(\max\left\{ \frac{\min\left\{ u_{0},8\right\} }{\varepsilon_{{\rm Mix}}},1\right\} \right),
\end{align*}
and computing that 
\[
0<a<b<1\implies\quad\int_{a}^{b}\frac{1}{\xi\cdot\log\left(\frac{1}{\xi}\right)}\,{\rm d}\xi=\log\left(\frac{\log\left(1/a\right)}{\log\left(1/b\right)}\right)
\]
provides the bound 
\begin{align*}
n & \geqslant2+2^{8}\cdot\alpha_{0}^{-2}\cdot\max\left\{ \log\left(\frac{u_{0}}{4\cdot v_{\circ}^{-1}}\right),0\right\} \\
 & +2^{10}\cdot C_{\ell}^{-2}\cdot\alpha_{0}^{-4}\cdot\sigma^{-2}\cdot m^{-1}\cdot\log\left(\frac{\log\left(\min\left\{ \max\left\{ \frac{1}{2}\cdot u_{0},4\right\} ,2\cdot v_{\circ}^{-1}\right\} \right)}{\log4}\right)\\
 & +2^{4}\cdot\max\left\{ 1,2^{-2}\cdot C_{\gamma}^{-2}\cdot\alpha_{0}^{-2}\cdot\sigma^{-2}\cdot m^{-1}\right\} \cdot\alpha_{0}^{-2}\cdot\log\left(\max\left\{ \frac{\min\left\{ u_{0},8\right\} }{\varepsilon_{{\rm Mix}}},1\right\} \right).
\end{align*}
By Corollary~\ref{cor:normal-accept-x} and $\exp\left(\varsigma^{2}\right)\cdot\varsigma^{-2}\geqslant\exp\left(1\right)$,
we obtain
\begin{align*}
\max\left\{ 1,2^{-2}\cdot C_{\gamma}^{-2}\cdot\alpha_{0}^{-2}\cdot\sigma^{-2}\cdot m^{-1}\right\}  & \leqslant\max\left\{ 1,C_{\gamma}^{-2}\cdot\exp\left(\varsigma^{2}\right)\cdot\varsigma^{-2}\cdot\kappa\cdot d\right\} \\
 & =C_{\gamma}^{-2}\cdot\exp\left(\varsigma^{2}\right)\cdot\varsigma^{-2}\cdot\kappa\cdot d,
\end{align*}
and simplifying the other terms provides the bound
\begin{align}
n & \geqslant2+2^{10}\cdot\exp\left(\varsigma^{2}\right)\cdot\max\left\{ \log\left(\frac{u_{0}}{4\cdot v_{\circ}^{-1}}\right),0\right\} \nonumber \\
 & +2^{14}\cdot C_{\ell}^{-2}\cdot\exp\left(2\cdot\varsigma^{2}\right)\cdot\varsigma^{-2}\cdot\kappa\cdot d\cdot\log\left(\frac{\log\left(\min\left\{ \max\left\{ \frac{1}{2}\cdot u_{0},4\right\} ,2\cdot v_{\circ}^{-1}\right\} \right)}{\log4}\right)\nonumber \\
 & +2^{6}\cdot C_{\gamma}^{-2}\cdot\exp\left(2\cdot\varsigma^{2}\right)\cdot\varsigma^{-2}\cdot\kappa\cdot d\cdot\log\left(\max\left\{ \frac{\min\left\{ u_{0},8\right\} }{\varepsilon_{{\rm Mix}}},1\right\} \right),\label{eq:rwm-mix-bound-2}
\end{align}
where we may also bound
\begin{align}
v_{\circ}^{-1} & =\max\left\{ 2,\frac{1}{2}\exp\left(4\cdot C_{\ell}^{-2}\cdot\sigma^{-2}\cdot\alpha_{0}^{-2}\cdot m^{-1}\right)\right\} \nonumber \\
 & \leqslant\max\left\{ 2,\frac{1}{2}\cdot\exp\left(16\cdot C_{\ell}^{-2}\cdot\exp\left(\varsigma^{2}\right)\cdot\varsigma^{-2}\cdot\kappa\cdot d\right)\right\} \nonumber \\
 & \leqslant\frac{1}{2}\cdot\exp\left(16\cdot C_{\ell}^{-2}\cdot\exp\left(\varsigma^{2}\right)\cdot\varsigma^{-2}\cdot\kappa\cdot d\right).\label{eq:rwm-mix-bound-2-vo}
\end{align}
For the stated bound, since $v_{\circ}^{-1}\geqslant2$, we have $\frac{u_{0}}{4\cdot v_{\circ}^{-1}}\leqslant u_{0}$,
leading to the bound on the first term. The second term follows from
\begin{align*}
\log\left(\frac{\log\left(\min\left\{ \max\left\{ \frac{1}{2}\cdot u_{0},4\right\} ,2\cdot v_{\circ}^{-1}\right\} \right)}{\log4}\right) & \leqslant\log\left(\frac{\log\left(2\cdot v_{\circ}^{-1}\right)}{\log4}\right)\\
 & \leqslant\log\left(\log\left(2\cdot v_{\circ}^{-1}\right)\right)\\
 & \leqslant\log\left(16\cdot C_{\ell}^{-2}\cdot\varsigma^{-2}\cdot\kappa\cdot d\right)+\varsigma^{2},
\end{align*}
where we have used the bound on $v_{\circ}^{-1}$ to arrive at the
final inequality.
\end{proof}
\begin{proof}[Proof of Lemma~\ref{lem:PCN-continuity-proposal}]
Compute directly that
\[
{\rm KL}\left(Q\left(x,\cdot\right),Q\left(y,\cdot\right)\right)=\frac{1}{2}\cdot\frac{\rho^{2}}{\eta^{2}}\cdot\left|x-y\right|_{\mathsf{C}^{-1}}^{2}.
\]
By Pinsker's inequality, it thus holds that
\[
\left\Vert Q\left(x,\cdot\right)-Q\left(y,\cdot\right)\right\Vert _{{\rm TV}}\leqslant\frac{1}{2}\cdot\frac{\rho}{\eta}\cdot\left|x-y\right|_{\mathsf{C}^{-1}}.\qedhere
\]
\end{proof}
\begin{proof}[Proof of Theorem~\ref{thm:pcn-mixing}]
The proof structure broadly follows that of Theorem~\ref{thm:RWM_mixing};
certain details which are omitted here are spelt out more clearly
in that proof.

By Lemma~\ref{lem:pcn-threeset}, $\pi$ admits the regular, concave
$\left|\cdot\right|_{\mathsf{C}^{-1}}$-isoperimetric minorant $\varphi_{\gamma}\circ\Phi_{\gamma}^{-1}$,
and by Lemma~\ref{lem:close-coupling-pcn}, $P$ satisfies a close
coupling inequality with $\delta=\alpha_{0}\cdot\frac{\eta}{\rho}$,
$\varepsilon=\frac{1}{2}\cdot\alpha_{0}$. By Theorem~\ref{thm:IP-to-mix},
it suffices to take 
\begin{align*}
n & \geqslant2+2^{6}\cdot\varepsilon^{-2}\cdot\max\left\{ \log\left(\frac{u_{0}}{4\cdot v_{\circ}^{-1}}\right),0\right\} \\
 & +2^{8}\cdot\varepsilon^{-2}\cdot\delta^{-2}\cdot\int_{\max\left\{ \min\left\{ 2\cdot u_{0}^{-1},1/4\right\} ,v_{\circ}/2\right\} }^{1/4}\frac{\xi}{\tilde{I}_{\pi}\left(\xi\right)^{2}}\,{\rm d}\xi\\
 & +2^{4}\cdot\max\left\{ 1,2^{-2}\cdot\delta^{-2}\cdot\tilde{I}_{\pi}\left(\frac{1}{4}\right)^{-2}\right\} \cdot\varepsilon^{-2}\cdot\log\left(\max\left\{ \frac{\min\left\{ u_{0},8\right\} }{\varepsilon_{{\rm Mix}}},1\right\} \right).
\end{align*}
Recalling that $\varphi_{\gamma}\circ\Phi_{\gamma}^{-1}\geqslant C_{\ell}\cdot p\cdot\left(\log\frac{1}{p}\right)^{1/2}$
for $p\in\left[0,\frac{1}{2}\right]$, we thus simplify to
\begin{align*}
n & \geqslant2+2^{8}\cdot\alpha_{0}^{-2}\cdot\max\left\{ \log\left(\frac{u_{0}}{4\cdot v_{\circ}^{-1}}\right),0\right\} \\
 & +2^{10}\cdot C_{\ell}^{-2}\cdot\alpha_{0}^{-4}\cdot\frac{\rho^{2}}{\eta^{2}}\cdot\int_{\max\left\{ \min\left\{ 2\cdot u_{0}^{-1},1/4\right\} ,v_{\circ}/2\right\} }^{1/4}\frac{1}{\xi\cdot\log\left(\frac{1}{\xi}\right)}\,{\rm d}\xi\\
 & +2^{6}\cdot\max\left\{ 1,2^{-2}\cdot C_{\gamma}^{-2}\cdot\alpha_{0}^{-2}\cdot\frac{\rho^{2}}{\eta^{2}}\right\} \cdot\alpha_{0}^{-2}\cdot\log\left(\max\left\{ \frac{\min\left\{ u_{0},8\right\} }{\varepsilon_{{\rm Mix}}},1\right\} \right).
\end{align*}
and repeating earlier calculations with the inner integral gives the
bound
\begin{align*}
n & \geqslant2+2^{8}\cdot\alpha_{0}^{-2}\cdot\max\left\{ \log\left(\frac{u_{0}}{4\cdot v_{\circ}^{-1}}\right),0\right\} \\
 & +2^{10}\cdot C_{\ell}^{-2}\cdot\alpha_{0}^{-4}\cdot\frac{\rho^{2}}{\eta^{2}}\cdot\log\left(\frac{\log\left(\min\left\{ \max\left\{ \frac{1}{2}\cdot u_{0},4\right\} ,2\cdot v_{\circ}^{-1}\right\} \right)}{\log4}\right)\\
 & +2^{6}\cdot\max\left\{ 1,2^{-2}\cdot C_{\gamma}^{-2}\cdot\alpha_{0}^{-2}\cdot\frac{\rho^{2}}{\eta^{2}}\right\} \cdot\alpha_{0}^{-2}\cdot\log\left(\max\left\{ \frac{\min\left\{ u_{0},8\right\} }{\varepsilon_{{\rm Mix}}},1\right\} \right).
\end{align*}
where $v_{\circ}:=\min\left\{ \frac{1}{2},2\cdot\exp\left(-4\cdot C_{\ell}^{-2}\cdot\alpha_{0}^{-2}\cdot\frac{\rho^{2}}{\eta^{2}}\right)\right\} $.

We observe that
\begin{align*}
\inf_{\varsigma\in\left(0,L^{1/2}\cdot{\rm Tr}\left(\mathsf{C}\right)^{1/2}\right)}\exp\left(\varsigma^{2}\right)\cdot\varsigma^{-2}\cdot\tilde{\kappa} & =\inf_{\eta\in\left(0,1\right)}\exp\left(\eta^{2}\cdot\tilde{\kappa}\right)\cdot\eta^{-2}\\
 & =\begin{cases}
\exp\left(\tilde{\kappa}\right) & \tilde{\kappa}\leqslant1,\\
\tilde{\kappa}\cdot\exp\left(1\right) & \tilde{\kappa}>1,
\end{cases}\\
 & \geqslant1.
\end{align*}
Using Lemma~\ref{lem:pCN-accept-x}, the bound above and $\rho^{2}\leqslant1$,
we obtain
\begin{align*}
\max\left\{ 1,2^{-2}\cdot C_{\gamma}^{-2}\cdot\alpha_{0}^{-2}\cdot\frac{\rho^{2}}{\eta^{2}}\right\}  & \leqslant\max\left\{ 1,C_{\gamma}^{-2}\cdot\exp\left(\varsigma^{2}\right)\cdot\varsigma^{-2}\cdot\tilde{\kappa}\right\} \\
 & =C_{\gamma}^{-2}\cdot\exp\left(\varsigma^{2}\right)\cdot\varsigma^{-2}\cdot\tilde{\kappa},
\end{align*}
providing the bound
\begin{align*}
n & \geqslant2+2^{10}\cdot\exp\left(\varsigma^{2}\right)\cdot\max\left\{ \log\left(\frac{u_{0}}{4\cdot v_{\circ}^{-1}}\right),0\right\} \\
 & +2^{14}\cdot C_{\ell}^{-2}\cdot\exp\left(2\cdot\varsigma^{2}\right)\cdot\varsigma^{-2}\cdot\tilde{\kappa}\cdot\log\left(\frac{\log\left(\min\left\{ \max\left\{ \frac{1}{2}\cdot u_{0},4\right\} ,2\cdot v_{\circ}^{-1}\right\} \right)}{\log4}\right)\\
 & +2^{6}\cdot C_{\gamma}^{-2}\cdot\exp\left(2\cdot\varsigma^{2}\right)\cdot\varsigma^{-2}\cdot\tilde{\kappa}\cdot\log\left(\max\left\{ \frac{\min\left\{ u_{0},8\right\} }{\varepsilon_{{\rm Mix}}},1\right\} \right),
\end{align*}
where
\begin{align*}
v_{\circ}^{-1} & =\max\left\{ 2,\frac{1}{2}\exp\left(4\cdot C_{\ell}^{-2}\cdot\sigma^{-2}\cdot\alpha_{0}^{-2}\cdot\frac{\rho^{2}}{\eta^{2}}\right)\right\} \\
 & \leqslant\max\left\{ 2,\frac{1}{2}\cdot\exp\left(16\cdot C_{\ell}^{-2}\cdot\exp\left(\varsigma^{2}\right)\cdot\varsigma^{-2}\cdot\tilde{\kappa}\right)\right\} \\
 & \leqslant\frac{1}{2}\cdot\exp\left(16\cdot C_{\ell}^{-2}\cdot\exp\left(\varsigma^{2}\right)\cdot\varsigma^{-2}\cdot\tilde{\kappa}\right).
\end{align*}
The stated bound is obtained by the following observations. Since
$v_{\circ}^{-1}\geqslant2$, $\frac{u_{0}}{4\cdot v_{\circ}^{-1}}\leqslant u_{0}$,
leading to the bound on the first term. The bound on the second term
follows from
\begin{align*}
\log\left(\frac{\log\left(\min\left\{ \max\left\{ \frac{1}{2}\cdot u_{0},4\right\} ,2\cdot v_{\circ}^{-1}\right\} \right)}{\log4}\right) & \leqslant\log\left(\frac{\log\left(2\cdot v_{\circ}^{-1}\right)}{\log4}\right)\\
 & \leqslant\log\left(\log\left(2\cdot v_{\circ}^{-1}\right)\right)\\
 & \leqslant\log\left(16\cdot C_{\ell}^{-2}\cdot\varsigma^{-2}\cdot\tilde{\kappa}\right)+\varsigma^{2},
\end{align*}
where we have used the bound on $v_{\circ}^{-1}$ to arrive at the
final inequality.
\end{proof}
\begin{prop}
\label{prop:spec-to-super}Suppose that the $\pi$-invariant Markov
kernel $P$ satisfies the optimized spectral profile inequality
\[
\mathcal{E}\left(P,g\right)\geqslant\mathrm{Var}_{\pi}\left(g\right)\cdot F\left(\frac{\pi\left(g\right)^{2}}{\mathrm{Var}_{\pi}\left(g\right)}\right)
\]
for all nonnegative, non-constant $\pi$-a.s. $g\in\ELL\left(\pi\right)$,
where $F$ is positive, decreasing and $\lim\sup_{v\to0^{+}}F\left(v\right)>0$
(and may be infinite). Then $P$ also satisfies a super-Poincaré inequality
of the form
\[
\mathrm{Var}_{\pi}\left(f\right)\leqslant s\cdot\mathcal{E}\left(P,f\right)+\beta\left(s\right)\cdot\pi\left(\left|f\right|\right)^{2},
\]
where $\beta$ is positive, decreasing, and can be written explicitly.
\end{prop}

\begin{proof}
For $w\geqslant0$, define 
\[
F^{*}\left(w\right)=\inf_{v\geqslant0}\left\{ F\left(v\right)+w\cdot v\right\} .
\]
Note that $F^{*}$ is positive, increasing, and $\mathrm{id}/F^{*}$
is increasing. By our assumption that $\lim\sup_{v\to0^{+}}F\left(v\right)>0$,
it also holds for $w>0$ that $F^{*}\left(w\right)>0$. One can then
write
\[
F\left(v\right)\geqslant\sup_{w\geqslant0}\left\{ F^{*}\left(w\right)-w\cdot v\right\} ,
\]
and hence for any $w>0$ and nonnegative, non-constant $g\in\ELL\left(\pi\right)$
that
\begin{align*}
\mathcal{E}\left(P,g\right) & \geqslant\mathrm{Var}_{\pi}\left(g\right)\cdot\left\{ F^{*}\left(w\right)-w\cdot\frac{\pi\left(g\right)^{2}}{\mathrm{Var}_{\pi}\left(g\right)}\right\} \\
\implies\quad\mathrm{Var}_{\pi}\left(g\right) & \leqslant\frac{1}{F^{*}\left(w\right)}\cdot\mathcal{E}\left(P,g\right)+\frac{w}{F^{*}\left(w\right)}\cdot\pi\left(g\right)^{2}.
\end{align*}
Writing $s=\frac{1}{F^{*}\left(w\right)}$ and $\beta_{0}=\left(\mathrm{Id}/F^{*}\right)\circ\left(1/F^{*}\right)$
(which is decreasing), we see for nonnegative $g\in\ELL\left(\pi\right)$
that
\[
\mathrm{Var}_{\pi}\left(g\right)\leqslant s\cdot\mathcal{E}\left(P,g\right)+\beta_{0}\left(s\right)\cdot\pi\left(g\right)^{2}.
\]
Now, let $f\in\ELL\left(\pi\right)$ and write $f_{\pm}=\max\left(\pm f,0\right)\geqslant0$,
so that
\[
\mathrm{Var}_{\pi}\left(f_{\pm}\right)\leqslant s\cdot\mathcal{E}\left(P,f_{\pm}\right)+\beta_{0}\left(s\right)\cdot\pi\left(f_{\pm}\right)^{2}.
\]
Standard calculations give that $\mathrm{Var}_{\pi}\left(f\right)\leqslant2\cdot\left(\mathrm{Var}_{\pi}\left(f_{+}\right)+\mathrm{Var}_{\pi}\left(f_{-}\right)\right)$,
$\mathcal{E}\left(P,f_{+}\right)+\mathcal{E}\left(P,f_{-}\right)\leqslant\mathcal{E}\left(P,f\right)$
\citep[Lemma~2.3]{goel2006mixing}, and $\pi\left(f_{+}\right)^{2}+\pi\left(f_{-}\right)^{2}\leqslant\pi\left(\left|f\right|\right)^{2}$.
Assembling these inequalities yields that
\begin{align*}
\mathrm{Var}_{\pi}\left(f\right) & \leqslant2\cdot\left(\mathrm{Var}_{\pi}\left(f_{+}\right)+\mathrm{Var}_{\pi}\left(f_{-}\right)\right)\\
 & \leqslant2\cdot s\cdot\left(\mathcal{E}\left(P,f_{+}\right)+\mathcal{E}\left(P,f_{-}\right)\right)+2\cdot\beta_{0}\left(s\right)\cdot\left(\pi\left(f_{+}\right)^{2}+\pi\left(f_{-}\right)^{2}\right)\\
 & \leqslant2\cdot s\cdot\mathcal{E}\left(P,f\right)+2\cdot\beta_{0}\left(s\right)\cdot\pi\left(\left|f\right|\right)^{2},
\end{align*}
i.e. that 
\[
\mathrm{Var}_{\pi}\left(f\right)\leqslant s\cdot\mathcal{E}\left(P,f\right)+\beta\left(s\right)\cdot\pi\left(\left|f\right|\right)^{2},
\]
where $\beta:s\mapsto2\cdot\beta_{0}\left(\frac{1}{2}\cdot s\right)$.
The result follows.
\end{proof}

\section{Different proposal distributions for RWM\label{sec:Different-proposal-distributions}}

For simplicity, the analysis of the RWM in Sections~\ref{sec:Spectral-gap-of-RWM}--\ref{sec:Convergence-and-mixing}
involved specifically Gaussian proposal distributions. We extend the
results here to a wide class of proposal distributions with independent
noise increments for each of the $d$ components. That is, we define
\[
Q\left(x,A\right)=\int\mathbf{1}_{A}\left(x+\sigma\cdot z\right)q^{\otimes d}\left(\mathrm{d}z\right),\qquad x\in\mathsf{E},A\in\mathscr{E},
\]
where $q$ is a probability measure on $(\mathbb{R},\mathcal{B}(\mathbb{R}))$
that is symmetric, i.e. $q(A)=q(-A)$ for all $A\in\mathcal{B}(\mathbb{R})$
where $\mathcal{B}(\mathbb{R})$ is the Borel $\sigma$-algebra of
$\mathbb{R}$. We define for $t\in\mathbb{R}$, $q_{t}$ to be the
distribution of $Z+t$ where $Z\sim q$. We define the squared Hellinger
distance between $P$ and $Q$ as
\[
d_{{\rm H}}(P,Q)^{2}=\int\left(\left(\frac{\mathrm{d}P}{\mathrm{d}\lambda}\left(x\right)\right)^{1/2}-\left(\frac{\mathrm{d}Q}{\mathrm{d}\lambda}\left(x\right)\right)^{1/2}\right)^{2}\,\lambda\left(\mathrm{d}x\right),
\]
where $\lambda$ is any common dominating reference measure. This
is a suitable metric to use because of its tensorization properties
and because by Le Cam's inequalities \citep[see, e.g.,][Section~2.4]{Tsybakov:1315296},
it holds that 
\[
\left\Vert P-Q\right\Vert _{{\rm TV}}\in\left[\frac{1}{2}\cdot d_{{\rm H}}(P,Q)^{2},d_{{\rm H}}(P,Q)\right].
\]
Hence, in order to control $\left\Vert P-Q\right\Vert _{{\rm TV}}\in\Theta\left(1\right)$,
it is necessary and sufficient to control $d_{{\rm H}}(P,Q)\in\Theta\left(1\right)$.
\begin{prop}
\label{prop:KL-smooth-xi-bound-close-coupling-1}Let $U$ be a potential
such that 
\begin{equation}
U\left(x+h\right)-U\left(x\right)-\left\langle \nabla U\left(x\right),h\right\rangle \leqslant\sum_{i=1}^{d}\psi\left(\left|h_{i}\right|\right),\label{eq:U-reg-1}
\end{equation}
for some $\psi:\mathbb{R}_{+}\to\mathbb{R}_{+}$ which is nondecreasing
and satisfies $\psi\left(0\right)=0$. Assume there exists $L_{H}>0$
such that
\[
d_{{\rm H}}(q,q_{t})^{2}\leqslant\frac{1}{2}\cdot L_{H}\cdot t^{2}\quad\text{for }t\in\mathbb{R}.
\]
Assume also for $\sigma>0$ that
\[
\int_{\mathbb{R}}q\left(\mathrm{d}z_{1}\right)\cdot\exp\left(-\psi\left(\sigma\cdot\left|z_{1}\right|\right)\right)\geqslant1-\xi\left(\sigma\right)
\]
for some $\xi:\mathbb{R}_{+}\to\mathbb{R}_{+}$ which is continuous,
nondecreasing and satisfies $\xi\left(0\right)=0$. Let $\eta>0$
and define $P_{\eta}$ to be the RWM kernel with proposal scaling
$\sigma$ satisfying $\xi\left(\sigma\right)=d^{-1}\cdot\eta$. It
then holds that
\begin{enumerate}
\item $P_{\eta}$ has minimal acceptance rate $\alpha_{0}\geqslant\frac{1}{2}\cdot\left(1-\frac{\eta}{d}\right)^{d}\geqslant\frac{1}{2}\cdot\exp\left(-\eta\cdot\frac{d}{d-\eta}\right)$.
\item The kernel $P_{\eta}$ is $\left(\left|\cdot\right|,\sigma\cdot\left(2\cdot L_{H}\right)^{-1/2}\cdot\alpha_{0},\frac{1}{2}\cdot\alpha_{0}\right)$-close
coupling.
\end{enumerate}
\end{prop}

\begin{proof}
For the first part, following the same steps in the proof of Lemma~\ref{lem:general-accept-x},
we have 
\begin{align*}
\alpha_{0} & \geqslant\frac{1}{2}\int q^{\otimes d}\left(\mathrm{d}z\right)\cdot\exp\left(-\sum_{i=1}^{d}\psi\left(\sigma\cdot\left|z_{i}\right|\right)\right)\\
 & =\frac{1}{2}\left\{ \int q\left(\mathrm{d}z_{1}\right)\cdot\exp\left(-\psi\left(\sigma\cdot\left|z_{1}\right|\right)\right)\right\} ^{d},
\end{align*}
from which the result follows. For the second part, we may write
\[
\left\Vert Q(x,\cdot)-Q(y,\cdot)\right\Vert _{{\rm TV}}=\left\Vert \bigotimes_{i=1}^{d}q_{\frac{x_{i}-y_{i}}{\sigma}}-\bigotimes_{i=1}^{d}q\right\Vert _{{\rm TV}}.
\]
Moreover, we have in general that
\begin{align*}
1-\frac{1}{2}\cdot d_{{\rm H}}\left(\bigotimes_{i=1}^{d}P_{i},\bigotimes_{i=1}^{d}Q_{i}\right)^{2} & =\prod_{i=1}^{d}\left\{ 1-\frac{1}{2}\cdot d_{{\rm H}}(P_{i},Q_{i})^{2}\right\} \\
 & \geqslant1-\sum_{i=1}^{d}\frac{1}{2}\cdot d_{{\rm H}}(P_{i},Q_{i})^{2}.
\end{align*}
We then see that 
\begin{align*}
\frac{1}{2}\cdot\left\Vert \bigotimes_{i=1}^{d}q_{\frac{x_{i}-y_{i}}{\sigma}}-\bigotimes_{i=1}^{d}q\right\Vert _{{\rm TV}}^{2} & \leqslant\sum_{i=1}^{d}\frac{1}{2}\cdot d_{{\rm H}}(q_{\frac{x_{i}-y_{i}}{\sigma}},q)^{2}\\
 & \leqslant\sum_{i=1}^{d}\frac{L_{H}\cdot\left|x_{i}-y_{i}\right|^{2}}{4\cdot\sigma^{2}}\\
 & =\frac{L_{H}\cdot\left|x-y\right|^{2}}{4\cdot\sigma^{2}},
\end{align*}
so that $\left\Vert Q(x,\cdot)-Q(y,\cdot)\right\Vert _{{\rm TV}}\leqslant\left(\frac{L_{H}}{2}\right)^{1/2}\cdot\frac{\left|x-y\right|}{\sigma}$
in general. In particular, taking $\mathsf{d}\left(x,y\right)=\left|x-y\right|\leqslant\sigma\cdot\left(2\cdot L_{H}\right)^{-1/2}\cdot\alpha_{0}$,
we obtain that $\left\Vert Q(x,\cdot)-Q(y,\cdot)\right\Vert _{{\rm TV}}\leq\frac{1}{2}\alpha_{0}$.
We may then conclude by Lemma~\ref{lem:met-tv-bound}.
\end{proof}
The conditions may be verified in various settings with specific choices
of $q$. If $\psi(x)=x^{\alpha}$ for some $\alpha\in[0,2]$ and $q$
has finite $\alpha$th moment then a bound on $\alpha_{0}$ is straightforward:
if $\sigma=\varsigma\cdot d^{-1/\alpha}$ then by Jensen's inequality,
\[
\alpha_{0}\geqslant\frac{1}{2}\cdot\exp\left(-d\int q\left(\mathrm{d}z\right)\cdot\sigma^{\alpha}\left|z\right|^{\alpha}\right)=\frac{1}{2}\cdot\exp\left(-\varsigma^{\alpha}\int q\left(\mathrm{d}z\right)\left|z\right|^{\alpha}\right).
\]

If $q$ has only a smaller moment, then the next Lemma shows that
it is still possible to obtain a bound. For example, if $\alpha=2$
and $q$ has a finite second moment, we obtain the same scaling $\xi(\sigma)\sim\sigma^{2}$
as for Gaussian proposals. If $q$ has fewer moments, or $U$ is rougher,
the scaling of $\sigma$ with dimension is more severe. For example,
if $q$ is Cauchy and $\alpha=2$ then one can only obtain $s$ arbitrarily
close to $1$ in Lemma~\ref{lem:appC}, and so a scaling of $\xi(\sigma)\sim\sigma$
which leads to scaling $\sigma\sim1/d$. In fact, the calculations
for the Cauchy can be done exactly, such that we have
\[
\int q\left(\mathrm{d}z\right)\cdot\exp\left(-\sigma^{2}\cdot z^{2}\right)=\exp\left(\sigma^{2}\right)\cdot\left\{ 1-\frac{2}{\pi}\int_{0}^{\sigma}\exp\left(-t^{2}\right)\,{\rm d}t\right\} ,
\]
which gives the same scaling.
\begin{lem}
\label{lem:appC}Suppose $\psi(x)=x^{\alpha}$ for some $\alpha\in[0,2]$,
and assume $q$ has finite $r$th moment for some $r>0$. Then
\[
\xi(\sigma)\leq\sigma^{s}\cdot\mathbb{E}_{q}\left[\left|Z\right|^{s}\right],
\]
where $s=\min\{\alpha,r\}$.
\end{lem}

\begin{proof}
Observe that $\exp(-t)\geqslant1-t\geqslant1-t^{\beta}$ for any $\beta\in[0,1]$.
Then
\begin{align*}
\int q\left(\mathrm{d}z\right)\cdot\exp\left(-\psi\left(\sigma\cdot\left|z\right|\right)\right) & \geqslant1-\int q\left(\mathrm{d}z\right)\cdot\psi\left(\sigma\cdot\left|z\right|\right)^{\beta}\\
 & =1-\sigma^{\alpha\beta}\mathbb{E}_{q}\left[\left|Z\right|^{\alpha\beta}\right],
\end{align*}
for $\beta\in[0,1]$. We conclude by taking $\beta=\min\{r/\alpha,1\}$.
\end{proof}
One approach to verifying the Hellinger condition is to follow asymptotic
statistical theory using a non-asymptotic variant of differentiability
in quadratic mean \citep[Section~7.2]{vaart_1998}. In particular,
we will assume that the proposal $q$ satisfies
\begin{equation}
\int\left(q\left(x+t\right)^{1/2}-q\left(x\right)^{1/2}\cdot\left(1+\frac{1}{2}\cdot s\left(x\right)\cdot t\right)\right)^{2}\leqslant\varphi\left(\left|t\right|\right)\quad\text{as }t\to0,\label{eq:dqm-psi}
\end{equation}
for some $s$ which is square-integrable under $q$ and some $\varphi$
which vanishes at least quadratically around $0$. The triangle inequality
then yields
\[
d_{{\rm H}}(q,q_{t})\leqslant\frac{1}{2}\cdot\left|s\right|_{\mathrm{L}^{2}\left(q\right)}\cdot\left|t\right|+\varphi\left(\left|t\right|\right)^{1/2}\in\Theta\left(t\right),
\]
as $t\to0^{+}$; boundedness of the Hellinger distance allows one
to conclude the existence of a suitable $L_{H}$. Such an estimate
should thus hold for all proposals corresponding to `regular' statistical
(location) models; see \citet[Example~7.8]{vaart_1998} for further
discussion. Alternatively, if $\log q$ is differentiable then it
may be convenient to bound $d_{{\rm H}}\left(q,q_{t}\right)$ using
the KL divergence. For example, if $\log q$ is $L_{q}$-smooth and
one defines $K(t)={\rm KL}(q,q_{t})$, we have $K(0)=K'(0)=0$ and
\[
\left|K'(s)\right|=\left|\int q(x)\left\{ \left(\log q\right)'(x-s)-\left(\log q\right)'(x)\right\} \mathrm{d}x\right|\leqslant L_{q}\left|s\right|,
\]
from which we may deduce that $d_{{\rm H}}(q,q_{t})^{2}\leqslant K(t)=\int_{0}^{t}K'(s){\rm d}s\leqslant\frac{1}{2}L_{q}t^{2}$.

Finally, the there remains an additional subtlety concerning the positivity
of the RWM Markov operator $P$, which we use to ensure that the spectral
gap and right spectral gap coincide. The approach of \citet{baxendale2005renewal}
that we have used for normal increments applies quite generally to
increment distributions which are decomposable in a particular sense,
but does not hold in complete generality. Hence one may require alternative
arguments to bound the left spectral gap or consider instead the Markov
chain associated with the operator $P_{{\rm lazy}}=\frac{1}{2}(P+{\rm Id})$,
which is necessarily positive. In particular, translating our main
arguments appropriately would then establish that the same quantitative
results (in terms of how the mixing time scales with $d,\kappa$)
hold for a lazy chain, up to some absolute constant factors.

\subsubsection*{Acknowledgements}

The authors would like to thank Persi Diaconis, the anonymous referees,
an Associate Editor and the Editors for their constructive comments
that improved the quality of this paper.

\subsubsection*{Funding}

CA, AL and AQW were supported by EPSRC grant `CoSInES (COmputational
Statistical INference for Engineering and Security)' (EP/R034710/1).
CA and SP were supported by EPSRC grant Bayes4Health, `New Approaches
to Bayesian Data Science: Tackling Challenges from the Health Sciences'
(EP/R018561/1).

\bibliographystyle{abbrvnat}
\bibliography{bib-rwm}

\end{document}

%% file: flowchart.tex
\begin{tikzpicture}[node distance = 2cm]

\node (start) [block] {1. Regular isoperimetric minorant $\tilde I_\pi$};
\node (block1) [block, below of = start] {2. Lower bound on isoperimetric profile};
\node (block2) [block, below of = block1] {3. Three-set isoperimetric inequality};
\node (block3) [block, below of = block2] {4. Bound on conductance profile};
\node (block4) [block, below of = block3] {5. Bound on spectral profile};
\node (block5) [block, below of = block4] {6. Lower bound on Dirichlet form $\mathcal E(P,f)$};
\node (block6) [block, below of = block5] {7. Upper bound on convergence of $P^n$};
\node (block7) [block, below of = block6] {8. Mixing time bounds};

\draw [imparrow] (start) -- node[anchor=west] {Definition~\ref{def:iso-ineq}} (block1);
\draw [imparrow] (block1) -- node[anchor=west] {Lemma~\ref{lem:iso-to-3set}} (block2);
\draw [imparrow] (block2) -- node[align = right,anchor=east] {With close coupling, \\ concavity of $\tilde I_\pi$} node[align = left,anchor=west] {Lemma~\ref{lem:iso-conductance-A-close-coupling}, \\ Corollary~\ref{cor:iso-couple-conductance}}  (block3);
\draw [imparrow] (block3) -- node[anchor=west] {Lemma~\ref{lem:cond-spec-profile-gap}} (block4);
\draw [imparrow] (block4) -- node[anchor=west] {Lemma~\ref{lem:dirichlet-variance-spectral-profile}} (block5);
\draw [imparrow] (block5) -- node[anchor=west] {Theorem~\ref{thm:CP-to-mix}} node[anchor=east] {Positive reversible chain} (block6);
\draw [imparrow] (block6) -- node[anchor=west] {Theorem~\ref{thm:IP-to-mix}} (block7);

\end{tikzpicture}